\documentclass[twoside,a4paper,reqno,11pt]{amsart}
\usepackage{amsfonts, amsbsy, amsmath, amsthm, amssymb, latexsym}
\usepackage{mathrsfs}
\usepackage{pdfsync}
\usepackage[top=30mm,right=30mm,bottom=30mm,left=30mm]{geometry}

\usepackage{stmaryrd}
\usepackage{bm}

\usepackage[bookmarks,bookmarksdepth=3]{hyperref}

\makeatletter
\newcommand{\imod}[1]{\allowbreak\mkern4mu({\operator@font mod}\,\,#1)}
\makeatother

\renewcommand{\a}{\alpha}
\renewcommand{\b}{\beta}
\newcommand{\F}{\mathbb{F}_{q}}
\newcommand{\e}{\epsilon}
 \renewcommand{\L}{\Lambda}
\renewcommand{\l}{\lambda} 
 
 \renewcommand{\to}{\rightarrow}
 \newcommand{\s}{\sigma}
  
 \newcommand{\C}{\mathcal{C}}

\newcommand{\leqs}{\leqslant}
\newcommand{\geqs}{\geqslant}
\newcommand{\la}{\langle}
\newcommand{\ra}{\rangle}

 \newcommand{\vs}{\vspace{3mm}}

\newcommand{\GL}{\operatorname{GL}}

\newcommand{\SL}{\operatorname{SL}}

\newcommand{\SO}{\operatorname{SO}}

\newcommand{\sym}{\operatorname{Sym}}

\newtheorem{theorem}{Theorem}

\newtheorem{remark}{Remark}
\newtheorem{corol}[theorem]{Corollary}

\newtheorem{thm}{Theorem}[section]
\newtheorem{prop}[thm]{Proposition}
\newtheorem{lem}[thm]{Lemma}
\newtheorem{cor}[thm]{Corollary}

\theoremstyle{definition}

\newtheorem{remk}[thm]{Remark}

\begin{document}

 \author{Timothy C. Burness}
 \address{T.C. Burness, School of Mathematics, University of Bristol, Bristol BS8 1TW, UK}
 \email{t.burness@bristol.ac.uk}

  \author{Robert M. Guralnick}
 \address{R.M. Guralnick, Department of Mathematics, University of Southern California, Los Angeles CA 90089-2532, USA}
 \email{guralnic@usc.edu}

 \author{Jan Saxl}
 \address{J. Saxl, Department of Pure Mathematics and Mathematical Statistics, University of Cambridge, Cambridge CB3 0WB, UK}
 \email{j.saxl@dpmms.cam.ac.uk}

\title{On base sizes for algebraic groups}

\subjclass[2010]{Primary 20B15; Secondary 20G15, 20G41, 20D06}

\keywords{Base size; simple algebraic groups; primitive permutation groups, generic stabilizer}

\thanks{The authors thank the organisers of the programme on Algebraic Lie Theory at the Isaac Newton Institute
for their hospitality during several visits in 2009. They also also thank Dr. Ross Lawther for his assistance with
some computations, and for making available his unpublished results on the fusion of unipotent classes in reductive
subgroups of exceptional algebraic groups. The authors thank Prof. Zinovy Reichstein and two anonymous referees for helpful comments.  Burness was supported by
EPSRC grant EP/I019545/1. Guralnick was partially supported by the NSF grants DMS-1001962, DMS-1302886 and
Simons Foundation fellowship 224965.  He also thanks the Institute for Advanced Study in Princeton.}

\begin{abstract}
Let $G$ be a permutation group on a set $\Omega$. A subset of $\Omega$ is a base for $G$ if its pointwise stabilizer is trivial; the base size of $G$ is the minimal cardinality of a base. In this paper we initiate the study of bases for algebraic groups defined over an algebraically closed field. In particular, we calculate the base size for all primitive actions of simple algebraic groups, obtaining the precise value in almost all cases. We also introduce and study two new base measures, which arise naturally in this setting. We give an application concerning the  
essential dimension of simple algebraic groups, and we establish several new results on base sizes for the corresponding finite groups of Lie type. The latter results are an important contribution to the classical study of bases for finite primitive permutation groups.  We also indicate some connections with generic stabilizers for representations of simple algebraic groups. 
\end{abstract}

\date{\today}
\maketitle

\setcounter{tocdepth}{1}
\tableofcontents

\section{Introduction}

Let $G$ be a transitive permutation group on a set $\Omega$ with point stabilizer $H$. A subset of $\Omega$ is a \emph{base} for $G$ if its pointwise stabilizer in $G$ is trivial. The \emph{base size} of $G$, denoted by $b(G,H)$ (or just $b(G)$ if the context is clear), is the minimal size of a base for $G$. Equivalently, the base size is the smallest number $b$ such that the intersection of some $b$ conjugates of $H$ in $G$ is trivial. 

Determining the base size of a given permutation group is a classical problem in permutation group theory, with a long tradition and many applications.  
For finite permutation groups, one of the earliest results is a theorem of Bochert \cite{Boch} from 1889, which states that if $G$ is a primitive permutation group of degree $n$ not containing the alternating group $A_n$, then $b(G) \leqs n/2$. The optimal bound in this situation was obtained by Liebeck \cite{L10}, showing that $b(G) < 9\log n$, unless $n=\binom{m}{k}^r$ and $G$ is a subgroup of $S_m \wr S_r$ containing $(A_m)^r$, where $A_m$ acts on $k$-element subsets of $\{1,\ldots, m\}$.
The proof of this result relies on the Classification of Finite Simple Groups.
By imposing additional conditions on $G$ it is possible to establish stronger bounds.
For example, if $G$ is a finite primitive solvable group then a theorem of Seress \cite{Seress} states that $b(G) \leqs 4$.

Bases arise naturally in several different contexts. For example, Bochert's result was motivated by the classical problem of bounding the order of a finite primitive permutation group, which attracted a lot of attention in the 19th century. Here $|G| \leqs n^{b(G)}$, so an upper bound on the base size of $G$ yields an upper bound on its order. In more recent years, bases have been  used extensively in the computational study of finite permutation groups (see \cite[Chapter 4]{Seress_book} for further details), whence the problem of calculating base sizes has important practical applications. In the graph-theoretic literature, if $\Gamma$ is a graph with automorphism group $G={\rm Aut}(\Gamma)$, then $b(G)$ is called the \emph{fixing number} (also \emph{determining number} or \emph{rigidity index}) of $\Gamma$ and this is a well-studied graph invariant (see \cite{BCam} and the references therein). In a different direction, some classical problems in the representation theory of groups can also be stated in terms of bases. For instance, if $H$ is a group and $V$ is a faithful $H$-module, then $H$ has a regular orbit on $V$ if and only if the corresponding affine group $V \rtimes H \leqs {\rm AGL}(V)$ admits a base of size $2$. 

Recently, a number of papers have investigated bases for finite non-solvable permutation groups (see \cite{BCN, Bur7, BGS, BGS3, BGS2,  BLS, BOB, BS, JF, GG, Hal, HP, James, James2}, for example). One of the central motivations here comes from a conjecture of Cameron and Kantor \cite{CK} on finite almost simple primitive groups. The conjecture asserts that there exists an absolute constant $c$ such that $b(G) \leqs c$ for all such groups $G$, excluding a prescribed list of obvious exceptions involving the action of alternating and symmetric groups on subsets and partitions, and also the action of classical groups on subspaces of the natural module. This conjecture was proved by Liebeck and Shalev \cite{LSh2}, using  probabilistic methods, and more recently it has been shown that $c=7$ is the best possible constant (see the sequence of papers \cite{Bur7, BGS,BLS,BOB}). More precisely, confirming a conjecture of Cameron \cite[p.122]{CamPG}, it is known that $b(G) \leqs 7$, with equality if and only if $G={\rm M}_{24}$ in its $5$-transitive action on $24$ points. Again, the proof uses probabilistic methods.

In this paper, we initiate the study of bases for infinite permutation groups. At this level, very little is known in general, with the exception of a few special cases. For example, in \cite{GG}, Goldstein and Guralnick compute the base size for the action of the classical group ${\rm PGL}_{2n}(k)$ on the set of cosets of the subgroup ${\rm PGSp}_{2n}(k)$, for any field $k$. Bases for the action of ${\rm PGL}_{n}(k)$ on subspace partitions of the natural module are studied by James \cite{James2}. In this paper we conduct a systematic study of bases for primitive actions of simple algebraic groups, motivated in part by the recent advances in our understanding of bases for finite groups of Lie type. Of course, in this context the aforementioned probabilistic methods are no longer available, so we need to develop a new approach and methodology. 

As we will see, the unprecedented scope and precision of our results also sheds new light on the study of bases for finite primitive permutation groups. For instance, earlier work shows that there are infinitely many so-called \emph{non-standard} finite primitive groups $G$ with $b(G)=b$ for $2 \leqs b \leqs 5$, and a unique group with $b(G)>6$, namely, the Mathieu group $G = {\rm M}_{24}$ acting on $24$ points (see \cite[Definition 1.1]{Bur7} for the precise definition of a non-standard group). In this paper, we complete the picture via Theorem \ref{t:cam}, which reveals that there are infinitely many with $b(G)=6$ (only six examples were known previously). More generally, a major project is to determine the exact base size of every finite almost simple primitive group. In \cite{BGS3} and \cite{BGS2} we consider non-subspace actions of finite classical groups (see \cite[Definition 2.1]{Bur7}), and  bases for finite exceptional groups will also be the subject of a future paper. In particular, our work is an important contribution to ongoing efforts to classify the finite primitive permutation groups with base size two.

Let $G$ be a (closed) connected affine algebraic group over an algebraically closed field $K$ of characteristic $p \geqs 0$. Let $\Omega$ be a faithful transitive $G$-variety with point stabilizer $H$, so we may identify $\Omega$ with the coset variety $G/H$. We define three base-related measures that arise naturally in this context:

\begin{itemize}\addtolength{\itemsep}{0.2\baselineskip}
\item[(i)] The \emph{exact base size}, denoted $b(G,H)$, is the smallest integer $c$ such that $\Omega$ contains
$c$ points with trivial pointwise stabilizer.
\item[(ii)] The \emph{connected base size}, denoted $b^0(G,H)$, is the smallest integer $c$ such that $\Omega$
contains $c$ points whose pointwise stabilizer has trivial connected component, i.e. the pointwise stabilizer is finite.
\item[(iii)] The \emph{generic base size}, denoted $b^1(G,H)$, is the smallest integer $c$ such that the product
variety $\Omega^c = \Omega \times \cdots \times \Omega$ ($c$ factors) contains a non-empty open subvariety $\L$ and
every $c$-tuple in $\L$ is a base for $G$.
\end{itemize}

Evidently, we have
$$b^0(G,H) \leqs b(G,H) \leqs b^1(G,H).$$
Our ultimate goal is to determine these base-related measures for all simple algebraic groups $G$ and all closed maximal subgroups $H$ of $G$ (that is, for all primitive actions of simple algebraic groups). Indeed, we essentially achieve this goal by computing these quantities in almost every case. In the handful of exceptional cases, we give a very narrow range for the possible values. Note that if the context is clear, we will sometimes write $b(G)$, $b^0(G)$ and $b^1(G)$ for the three base measures defined above.

\begin{remark}\label{r:1}
\emph{More generally, one can define $b^0(G,X)$, $b(G,X)$ and $b^1(G,X)$ for any  affine algebraic group $G$ and irreducible $G$-variety $X$. For instance, $b^0(G,X) = b$ (respectively, $b^1(G,X) = b$) if and only if $b$ is minimal such that the product variety $X^b$ contains a non-empty open subvariety $X_0$ with the property that the stabilizer of a point in $X_0$ is finite (respectively, trivial). In characteristic $0$, these measures have been studied by various authors when $G$ is semisimple and $X$ is a $KG$-module (see \`{E}la\v{s}vili \cite{Elas, Elas2} and Popov \cite{Popov, Popov2}, for example). In addition, there is a connection between stabilizers of points on Grassmannians and stabilizers in certain tensor products of linear representations. We refer the reader to  \cite{encyc} for a very nice survey of results of this nature in characteristic $0$. In particular, all cases of irreducible modules for simple algebraic groups 
where there is a regular orbit have been determined. This has recently been extended to positive characteristic by Guralnick et al. \cite{GL}.}
\end{remark}
 
\begin{remark}
\emph{The connected base size is also related to the notion of
\emph{subgroup height} appearing in the geometric group theory literature. Following \cite{GMRS}, an infinite subgroup $H$ of a group $G$ has
\emph{height} $n$, denoted $\mu(H)=n$, if there exists a collection
of $n$ distinct $G$-conjugates of $H$ whose common intersection is infinite, but the intersection of any $n+1$ distinct
conjugates of $H$ is finite. In particular, if $H$ is core-free then $\Omega=G/H$ is a faithful transitive $G$-set and $\mu(H)=n$ if there exist $n$ points in
$\Omega$ whose pointwise stabilizer is infinite, but the stabilizer of any $n+1$ points is finite. 
Evidently, if $G$ is an algebraic group with point stabilizer $H$, then $b^0(G,H) \leqs \mu(H)+1$.}
\end{remark}

A simple algebraic group $G$ is either classical or exceptional, and there is a dichotomy in our approach.
The main theorem on the subgroup structure of classical algebraic groups is due to Aschbacher \cite{Asch}
(see also Liebeck and Seitz \cite{LieS}). Roughly speaking, a maximal closed
positive-dimensional subgroup $H$ of $G$ is either contained in one of five \emph{natural}, or
\emph{geometric}, subgroup collections (denoted by $\C_1,\C_2,\C_3,\C_4$ and $\C_6$ in \cite{LieS}), or the connected
component $H^0$ is simple (modulo scalars) and acts irreducibly on the natural $G$-module $V$
(we denote the latter collection by $\mathcal{S}$). The geometric collections include stabilizers of subspaces
of $V$, and normalizers of appropriate direct sum and tensor product decompositions of $V$.

In stating our results for a classical group $G$, we make a distinction between the primitive actions of $G$ in which
a point stabilizer $H$ acts reducibly on $V$, and those in which the stabilizer is irreducible. More precisely, we say
that the action of $G$ on $\Omega$ is a \emph{subspace action} if one of the following holds:
\begin{itemize}\addtolength{\itemsep}{0.2\baselineskip}
\item[(i)] $\Omega$ is an orbit of subspaces of $V$; or
\item[(ii)] the action of $G$ on $\Omega$ is equivalent to the action of an isomorphic classical group $L$ on an orbit of
subspaces of the natural $L$-module.
\end{itemize}
The possibilities that arise in case (ii) are conveniently listed in Table \ref{t:sub}. Here the `type of $H$'
describes the approximate group-theoretic structure of $H$ (this is consistent with the notation used in \cite{LieS}). In addition, we use the abbreviations `n.s.' and `n.d.' to denote the terms `non-singular' and `non-degenerate', respectively. 

It is worth noting that this is the first paper to systematically study the base size of classical groups (finite or infinite) in these natural subspace actions (see Theorem \ref{t:csub}). As noted above (see Remark \ref{r:1}), some of our results for subspace actions of linear groups in characteristic $0$ are related to earlier work of \`{E}la\v{s}vili \cite{Elas}, Popov \cite{Popov} and others.

\begin{table}[h]
$$\begin{array}{llll} \hline
G & \mbox{Type of $H$} & \mbox{Conditions} & \mbox{Equivalent action} \\ \hline
{\rm Sp}_{n} & O_{n} & p = 2, \, n \geqs 4  & \mbox{${\rm SO}_{n+1}$ on n.s. $1$-spaces} \\
{\rm SO}_{8} & {\rm Sp}_{4}\otimes {\rm Sp}_{2} & p \neq 2 & \mbox{${\rm SO}_{8}$ on n.d. $3$-spaces} \\
{\rm SO}_{8} & {\rm GL}_{4} & & \mbox{${\rm SO}_{8}$ on n.d. $2$-spaces} \\
{\rm SO}_{8} & {\rm SO}_{7} & \mbox{$H$ irreducible, $p \neq 2$} & \mbox{${\rm SO}_{8}$ on n.d. $1$-spaces} \\
{\rm SO}_{8} & {\rm Sp}_{6} & \mbox{$H$ irreducible, $p=2$} & \mbox{${\rm SO}_{8}$ on n.s. $1$-spaces} \\
{\rm SL}_{4} & {\rm Sp}_{4} & & \mbox{${\rm SO}_{6}$ on n.d. $1$-spaces} \\
{\rm Sp}_{4} & {\rm Sp}_{2} \wr S_2 & p \neq 2 & \mbox{${\rm SO}_{5}$ on n.d. $1$-spaces}\\
{\rm Sp}_{4} & {\rm Sp}_{2} \wr S_2 & p=2 & \mbox{${\rm SO}_{5}$ on n.s. $1$-spaces}\\ \hline
\end{array}$$
\caption{Some subspace actions}
\label{t:sub}
\end{table}

There is a similar description of the maximal subgroups of an exceptional algebraic group $G$, which is due to Liebeck
and Seitz \cite{LSmem}. Essentially, a positive-dimensional maximal subgroup of $G$ is either parabolic, or it is
of the form $N_G(X)$ for a known reductive subgroup $X$. Once again, we will make a distinction between parabolic
and non-parabolic subgroups.

In order to state our main results, we fix the following notation for the rest of the paper: let $K$ be an algebraically closed field of characteristic $p \geqs 0$, let $G$ be a simple (and in particular connected) algebraic group over $K$ and let $\Omega$ be a primitive $G$-variety with positive-dimensional
point stabilizer $H$. We remark that in our results, we can take $G$ to be any version of the simple algebraic group; the center will
lie in the kernel of the action of $G$ on $\Omega$, and will be ignored in the statements. In addition, $K$ can be any algebraically closed field of characteristic $p$ (see the end of Section \ref{s:prel} for further comments on the underlying field).

In the statement of Theorem \ref{t:camer} we use the notation $P_i$ to denote the standard maximal parabolic subgroup
of $G$ that corresponds to deleting the $i$-th node from the Dynkin diagram of $G$, in terms of the standard labelling (see \cite[p.250]{Bou}).

\begin{theorem}\label{t:camer}
Let $G$ be a simple algebraic group over an algebraically closed field and let $\Omega$ be a primitive $G$-variety with point stabilizer $H$. Assume $G$ is not a classical group in a subspace action. Then $b^1(G,H) \leqs 6$, with equality if and only if $(G,H) = (E_7,P_7)$, $(E_6,P_1)$ or $(E_6,P_6)$.
\end{theorem}

\begin{theorem}\label{t:bur}
Let $G$ be a simple classical algebraic group in a primitive non-subspace action with point stabilizer $H$.
Then $b^1(G,H) \leqs 4$, with equality if and only if $(G,H)=({\rm SL}_{6}, {\rm Sp}_{6})$, $({\rm SO}_{7}, G_2)$
$(p \neq 2)$ or $({\rm Sp}_{6},G_2)$ $(p=2)$.
\end{theorem}

We note that Theorem \ref{t:camer} establishes a strong algebraic group analogue of Cameron's conjecture for finite almost simple primitive groups (strong in the sense that we are able to determine all the cases in which the generic base size is exactly $6$). Also note that Theorem \ref{t:bur} can be viewed as an algebraic group version of the main theorem of \cite{Bur7}. 

\begin{theorem}\label{t:dimhalf}
Let $G$ be a simple algebraic group over an algebraically closed field of characteristic $p \neq 2$ and let $\Omega$ be a primitive $G$-variety with point stabilizer $H$.
Then $b(G,H)>2$ if and only if one of the following holds:
\begin{itemize}\addtolength{\itemsep}{0.2\baselineskip}
\item[{\rm (i)}] $\dim H > \frac{1}{2}\dim G$;
\item[{\rm (ii)}] $G={\rm SO}_{n}$ and $H$ is the stabilizer of a $d$-dimensional non-degenerate subspace of the natural $G$-module, where $n=2d+\ell$ with $2 \leqs \ell \leqs d$ and $\ell^2 \leqs n$;
\item[{\rm (iii)}] $G={\rm SL}_{n}$ and $H$ is of type ${\rm GL}_{n/2} \wr S_2$, where $n \geqs 4$;
\item[{\rm (iv)}] $G={\rm Sp}_{6}$ and $H$ is of type ${\rm Sp}_{2} \wr S_3$;
\item[{\rm (v)}] $G=E_6$ and $H=A_1A_5$.
\end{itemize}
\end{theorem}

\begin{remark}
\emph{As a corollary of Theorem \ref{t:dimhalf} , note that if $b(G,H) > 2$ and $\dim H \leqs \frac{1}{2}\dim G$, then either $H$ contains a maximal torus of $G$, or $G = {\rm SO}_n$ (with $n$ even) and $H$ is the stabilizer of a $d$-dimensional non-degenerate  subspace ($d$ odd), with $d$ satisfying the conditions in part (ii) of Theorem \ref{t:dimhalf}. In addition, in each of these cases we have 
$$b^0(G,H) = b(G,H) = b^1(G,H) = 3.$$ 
We can also state a version of  Theorem \ref{t:dimhalf} when $p=2$. Indeed, if we exclude the cases
$$(G,H) = ({\rm SO}_{n}, O_{n/2} \wr S_2) \; (\mbox{$n/2$ even}),\; (E_7,A_7.2),\; (E_6,A_1A_5),\; (G_2,A_1\tilde{A}_1)$$
then $b(G,H)>2$ if and only if we are in one of the cases (i) -- (iv) in Theorem \ref{t:dimhalf}.}
\end{remark}

Theorems \ref{t:camer} -- \ref{t:dimhalf} follow immediately from the detailed results we present in Theorems \ref{t:csub} -- \ref{t:mainep2} below. First assume $G$ is a classical group in a primitive subspace action,
so the point stabilizer $H$ fixes a proper non-zero subspace $U$ of the natural $G$-module $V$. Note that if $G$
is a symplectic or orthogonal group then primitivity implies that either $U$ is non-degenerate or totally singular
with respect to the relevant underlying form on $V$, or $G$ is orthogonal, $p=2$ and $U$ is a non-singular $1$-space.
Without loss of generality, we may assume that $\dim U \leqs \frac{1}{2}\dim V$.
Our main result on subspace actions is Theorem \ref{t:csub} below (in the statement of this result, $\delta_{i,j}$ denotes the familiar Kronecker delta).

\begin{theorem}\label{t:csub}
Let $G$ be a simple classical algebraic group in a primitive subspace action with point stabilizer $H=G_{U}$, where $d = \dim U$, $n = \dim V$ and $d \leqs n/2$. Set $k = \lceil n/d \rceil$.
\begin{itemize}\addtolength{\itemsep}{0.2\baselineskip}
\item[{\rm (i)}] Suppose $G={\rm SL}_{n}$ and $n \geqs 2$. If $d$ divides $n$ then 
$$b^0(G,H) = b(G,H) = b^1(G,H)=k+\epsilon,$$
where
$$\epsilon = \left\{\begin{array}{ll}
3 & \mbox{if $1<d=n/2$} \\
2 & \mbox{if $1<d<n/2$} \\
1 & \mbox{if $d=1$}.
\end{array}\right.$$
Otherwise, if $d$ does not divide $n$, then
$$k +1 \leqs b^0(G,H) = b(G,H) = b^1(G,H) \leqs k+2+\delta_{3,k}.$$
\item[{\rm (ii)}] Suppose $G = {\rm Sp}_{n}$ and $n \geqs 4$.
Then either 
$$b^0(G,H) = b(G,H) = b^1(G,H)=k$$ 
or one of the following holds:
\begin{itemize}\addtolength{\itemsep}{0.2\baselineskip}
\item[{\rm (a)}] $n=6$, $d=2$ and $b^0(G,H) = b(G,H) = b^1(G,H)=4$;
\item[{\rm (b)}] $U$ is totally singular, $d=n/2$, $b^0(G,H) = b(G,H) = 4$ and $b^1(G,H)=5-\delta_{2,p}$;
\item[{\rm (c)}] $H=O_n$, $p=2$, $b^0(G,H)=b(G,H)=n$ and $b^1(G,H)=n+1$.
\end{itemize}
\item[{\rm (iii)}] Suppose $G = {\rm SO}_{n}$ and $n \geqs 7$, with $p \ne 2$ if $n$ is odd.
Then either 
$$b^0(G,H) = b(G,H) = b^1(G,H)=k$$
or one of the following holds:
\begin{itemize}\addtolength{\itemsep}{0.2\baselineskip}
\item[{\rm (a)}] $n=(k-1)d+1$ (with $k \geqs 4$ if $U$ is totally singular), $b^0(G,H) = b(G,H) = k-1$ and $b^1(G,H)=k-\e$, where $\e=1$ if $n$ is even, otherwise $\e=0$;
\item[{\rm (b)}] $U$ is totally singular, $d=n/2$, $n \neq 10$ and 
$$b^0(G,H) = b(G,H) = b^1(G,H)=c(n),$$
where $c(8)=7$, $c(12)=6$ and $c(n) =5$ for all $n \geqs 14$;
\item[{\rm (c)}] $U$ is totally singular, $n=10$, $d=5$ and $5 \leqs b^0(G,H) \leqs b^1(G,H) \leqs 6$;
\item[{\rm (d)}] $U$ is totally singular, $k=3$ and $b^0(G,H) = b(G,H) = b^1(G,H)=4-\delta_{n,3d}$.
\end{itemize}
\end{itemize}
\end{theorem}

The next result deals with the non-subspace actions of classical groups.

\begin{theorem}\label{t:cmain}
Let $G$ be a simple classical algebraic group in a primitive non-subspace action with point stabilizer $H$.
Then one of the following holds:
\begin{itemize}\addtolength{\itemsep}{0.2\baselineskip}
\item[{\rm (i)}] $b^0(G,H)=b(G,H)=b^1(G,H)=2$;
\item[{\rm (ii)}] $b^0(G,H)=b(G,H) = b^1(G,H) = b>2$ and $(G,H,b)$ is recorded in Table \ref{t:c};
\item[{\rm (iii)}] $b^0(G,H) = b(G,H) = 2$, $b^1(G,H)=3$ and either $G={\rm SL}_{2}$ and $H$ is of type ${\rm GL}_{1}\wr S_2$, or $p \neq 2$ and
$(G,H) = ({\rm SL}_{n}, {\rm SO}_{n})$, $({\rm Sp}_{n},{\rm GL}_{n/2})$ or $({\rm SO}_{n}, O_{n/2}\wr S_2)$;
\item[{\rm (iv)}] $2 = b^0(G,H) \leqs b(G,H) \leqs b^1(G,H) = 3$, $p=2$ and $(G,H) = ({\rm SO}_{n},O_{n/2}\wr S_2)$, where $n \equiv 0 \imod{4}$ and $n \geqs 8$.
\end{itemize}
\end{theorem}

\begin{table}[h]
$$\begin{array}{llll} \\ \hline
G & \mbox{Type of $H$} & \mbox{Conditions} & b \\ \hline

{\rm SL}_{n} & {\rm GL}_{n/2}\wr S_2 & n \geqs 4 & 3  \\
& {\rm Sp}_{n} & n = 6 & 4 \\
& {\rm Sp}_{n} & n \geqs 8 & 3 \\

{\rm Sp}_{n} & {\rm Sp}_{n/2}\wr S_2 & n \geqs 8 & 3 \\
& {\rm Sp}_{n/3} \wr S_3 & n=6 & 3  \\
& G_2 & (n,p)=(6,2) & 4 \\

{\rm SO}_{n} & {\rm GL}_{n/2} & n \geqs 10 & 3 \\
& G_2 & n=7,\;p \neq 2 & 4 \\ \hline
\end{array}$$
\caption{Values of $b$ in Theorem \ref{t:cmain}(ii)}
\label{t:c}
\end{table}

\vspace{4mm}

In the next two theorems we present our results for parabolic and non-parabolic actions of exceptional algebraic groups, respectively.

\begin{theorem}\label{t:mainep1}
Let $G$ be a simple exceptional algebraic group and let $\Omega = G/H$, where $H=P_i$ is a maximal parabolic
subgroup of $G$. Then
$$c-\e \leqs b^0(G,H) \leqs b(G,H) \leqs b^1(G,H) \leqs c,$$
where $c$ is defined in Table \ref{t:ep}. Here an asterisk indicates that $\e=1$, otherwise $\e=0$ and thus $b^0(G,H)=b(G,H) = b^1(G,H) = c$.
\end{theorem}

\begin{table}[h]
$$\begin{array}{r|rlllllll}
 & H=P_{1} & P_{2} & P_{3} & P_{4} & P_{5} & P_{6} & P_{7} & P_{8} \\ \hline
G=E_{8} & 4 \hspace{3.3mm}& 3 & 3 & 3 & 3 & 3 & 4 & 5 \\
E_{7} &  5  \hspace{3.3mm}& 4 & 4 & 3 & 3 & 4 & 6 & \\
E_{6} & 6  \hspace{3.3mm} & 5 & 4 & 4^* & 4 & 6 & & \\
F_{4} & 5^* \hspace{1.5mm} & 4^* & 4^* & 5^* & & & & \\
G_{2} & 4^* \hspace{1.5mm} & 4^* & & & & & & \\
\end{array}$$
\caption{$G$ exceptional, $H$ parabolic}
\label{t:ep}
\end{table}

\vs

\begin{theorem}\label{t:mainep2}
Let $G$ be a simple exceptional algebraic group in a primitive non-parabolic action with point stabilizer $H$.
Then one of the following holds:
\begin{itemize}\addtolength{\itemsep}{0.2\baselineskip}
\item[{\rm (i)}] $b^0(G,H)=b(G,H)=b^1(G,H)=2$;
\item[{\rm (ii)}] $b^0(G,H)=b(G,H) = b^1(G,H) = b>2$ and $(G,H,b)$ is recorded in Table \ref{t:e};
\item[{\rm (iii)}] $b^0(G,H) = b(G,H) = 2$, $b^1(G,H)=3$, $p \neq 2$ and
$$(G,H^0) = (E_8, D_8), \, (E_7, A_7), \, (E_6,C_4), \, (F_4,A_1C_3) \mbox{ or } (G_2, A_1\tilde{A}_1);$$
\item[{\rm (iv)}] $2 = b^0(G,H) \leqs b(G,H) \leqs b^1(G,H) \leqs 3$, $p=2$ and
$$(G,H^0)=(E_7,A_7), \, (E_6,A_1A_5) \mbox{ or } (G_2,A_1\tilde{A}_1).$$
\end{itemize}
\end{theorem}

\begin{table}[h]
$$\begin{array}{llll} \\ \hline
G & H^0 & \mbox{Conditions} & b \\ \hline
E_8 & A_1E_7 & & 3  \\
E_7 & A_1D_6 & & 3  \\
& T_1E_6 & & 3 \\
E_6 & F_4 & & 4 \\
& D_5T_1 & & 3  \\
& A_1A_5 & p \neq 2 & 3  \\
F_4 & B_4 & & 4  \\
& C_4 & p=2 & 4 \\
& D_4 & & 3  \\
& \tilde{D}_4 & p = 2 & 3 \\
G_2 & A_2 & & 3  \\
& \tilde{A}_2 & p = 3 & 3 \\  \hline
\end{array}$$
\caption{Values of $b$ in Theorem \ref{t:mainep2}(ii)}
\label{t:e}
\end{table}

\vs

We outline the idea behind the proof.
Let $G$ be an algebraic group over an algebraically
closed field of characteristic $p \geqs 0$, and let $\Omega=G/H$ be a faithful transitive $G$-variety, where $H$ is a closed
subgroup of $G$. Let $c \geqs 2$ be an integer. The expression
$$\mathcal{Q}(G,c) = \frac{c}{c-1}\cdot \sup_{x\in \mathcal{P}}\left\{\frac{\dim (x^G \cap H)}{\dim x^G}\right\}$$
will play a central role, where $\mathcal{P}$ denotes the set of elements of prime order in $H$
(including all nontrivial unipotent elements if $p=0$) and $x^G$ is the conjugacy class of $x$ in $G$. In Theorem \ref{t:con} we prove that
if $G$ is simple and $H^0$ is reductive, then $b^1(G,H) \leqs c$ if $\mathcal{Q}(G,c)<1$. This result is an essential
tool in our analysis. Bounds on $\dim (x^G \cap H)$ in terms of $\dim x^G$ are obtained for classical groups in
\cite{Bur2} (for $H$ irreducible), and in \cite{LLS} for exceptional groups, so we can compute good estimates
for $\mathcal{Q}(G,c)$.

In order to obtain precise results, we require a lower bound on $b^0(G,H)$. By definition, if $b = b^0(G,H)$ then the product variety $\Omega^{b}$ contains a $G$-orbit of dimension $\dim G$, whence $\dim G \leqs \dim \Omega^b = b\cdot \dim \Omega$ and thus
$$b^0(G,H) \geqs \frac{\dim G}{\dim \Omega} = \frac{\dim G}{\dim G -\dim H}.$$

It turns out that this lower bound, combined with analysis of $\mathcal{Q}(G,c)$, is effective in most cases. However, we
sometimes encounter problems if $\dim G / \dim \Omega = c -\epsilon$ for some integer $c$ and small positive number
$\epsilon$ (with $\epsilon<1/10$, for example). Frequently, in such a situation, the usual analysis yields
$$c \leqs b^0(G,H) \leqs b^1(G,H) \leqs c+1$$
and thus further work is needed to determine the precise base size.
The case $c=2$ is particularly interesting because such a subgroup $H$ often arises as the centralizer
of an involution in ${\rm Aut}(G)$ (at least when $p \neq 2$). Rather surprisingly, we find that the base size
in this situation is determined by whether or not the relevant involution inverts a maximal torus of $G$.

\begin{theorem}\label{inv:main}
Let $G$ be a simple algebraic group of rank $r$ over an algebraically closed field of characteristic $p \neq 2$.
Let $H=C_{G}(\tau)$, where $\tau \in {\rm Aut}(G)$ is an involution, and let $\Omega=G/H$ be the
corresponding coset variety. If $\tau$ inverts a maximal torus of $G$ then
$$b^0(G,H)=b(G,H)=2,\;\; b^1(G,H)=3,$$
otherwise $b^0(G,H) \geqs 3$. More precisely, if $\tau$ inverts a maximal torus then the following hold:
\begin{itemize}\addtolength{\itemsep}{0.2\baselineskip}
\item[{\rm (i)}] $H$ has a unique regular orbit on $\Omega$.
\item[{\rm (ii)}] The generic $2$-point stabilizer has order $2^r$ (more precisely, this is the $2$-torsion
subgroup of a maximal torus).  That is, there exists a non-empty open subvariety $U$ of $\Omega \times \Omega$ such that
$|G_{\a} \cap G_{\b}|=2^r$ for all $(\a,\b) \in U$.
\item[{\rm (iii)}] If $G < A \leqs {\rm Aut}(G)$ then $A$ acts on $\Omega$, 
  $b^0(A, C_A(\tau))=2$ and $b(A, C_A(\tau))=b^1(A, C_A(\tau))=3$.
\end{itemize}
\end{theorem}

\vs

As a special case of Theorem \ref{inv:main}, we deduce that if $G=A_1$, $p \neq 2$ and $H=N_G(T)$ is the normalizer
of a maximal torus of $G$, then the generic $2$-point stabilizer has order $2$. In general,
if $G$ is simple and $H$ is the normalizer of a maximal torus then the generic $2$-point stabilizer is trivial.

\begin{theorem}\label{t:zin}
Let $G$ be a simple algebraic group over an algebraically closed field and consider the action of $G$ on $\Omega=G/H$,
where $H$ is the normalizer of a maximal torus of $G$. Then either $b^1(G,H)=2$, or $G=A_1$ and the generic
$2$-point stabilizer has order $2$.
\end{theorem}

Theorem \ref{t:zin} answers a question posed by Zinovy Reichstein (personal communication), and the proof is an easy application of Theorem \ref{t:con} (see Section \ref{s:zin} for the details). This result has the following corollary, where ${\rm ed}(H)$ denotes the \emph{essential dimension} of an algebraic group $H$.  We refer the reader to Reichstein's ICM survey article \cite{ICM} for further details and references.

\begin{corol}\label{c:zin}
Let $G$ be a simple algebraic group of adjoint type over an algebraically closed field
with $G$ of rank $r \geqs 2$. Let $H$ be the normalizer of a maximal torus of $G$.
 Then
$${\rm ed}(G) \leqs  {\rm ed}(H)  \leqs  \dim G - 2r.$$
\end{corol}

Lemire \cite{lemire} proved this in characteristic $0$ (the first
inequality is well known and follows from results of Springer -- see \cite[Proposition 4.3]{Re}). The point is that the
action of $H$ on $G/H$ is generically free (essentially by Theorem \ref{t:zin}) and is known to be versal  \cite[Example 7.3(c)]{DR}, which gives the desired bound. 
A sketch proof is given in Section \ref{s:zin}, and we refer the reader to \cite{DR, ICM} for more details
and generalities about essential dimension.   Using Theorem \ref{t:zin},  Garibaldi and Guralnick \cite{GaG} gave a slightly easier proof
of Corollary \ref{c:zin}  and improved the bound to $\dim G - 2r -1$.   

\vs

Let $G$ be an algebraic group acting on an irreducible variety $X$. Suppose there is a non-empty open subvariety $X_0$ of $X$ such that the stabilizer $G_x$ has a certain property $\mathcal{P}$ for all $x \in X_0$. In this situation, we say that a \emph{generic stabilizer} in $G$ has property $\mathcal{P}$. For example, notice that if there is at least one point $x \in X$ such that $G_x$ is finite, then a generic stabilizer is finite (more generally, the generic stabilizers will have a fixed dimension). In particular, if $X=G/H$ is a faithful transitive $G$-variety, then $b^0(G,H) \leqs b$ (respectively, $b^1(G,H) \leqs b$) if and only if the generic stabilizer of a point in $X^b$ is finite (respectively, trivial). In almost every case where $G$ is simple and $X=G/H$ is primitive, we show that if $b^0(G,H)=b$ then the generic stabilizers of points in $X^b$ form a single conjugacy class of subgroups of $H$.     

Let $G$ be a simple algebraic group over an algebraically closed field $K$ and let $V$ be a rational finite dimensional $KG$-module. In characteristic $0$, the existence of a generic stabilizer is a nontrivial theorem under suitable hypotheses (that is, there is a non-empty open subvariety of $V$ such that the stabilizer of each vector in this subvariety belongs to a fixed conjugacy class of subgroups).  See for example \cite[Chapter 7]{encyc} and \cite{rich}.
Rather less is known in the positive characteristic setting, although a recent theorem of  Guralnick et al. \cite{GL} guarantees the existence of generic stabilizers when $G$ is simple and $V$ is irreducible. Suppose that $H$ is  the generic stabilizer in $G$ of a vector in $V$ (up to conjugacy), let $c$ be a positive integer and let $W$ be the direct sum of $c$ copies of $V$. The $G$-stabilizer of a vector in an open subvariety of $W$ is just the intersection of $c$  $G$-conjugates of $H$.  In particular, if $b^1(G,H) \leqs c$ (in terms of the action of $G$ on $G/H$) then this generic stabilizer is trivial. Similarly, if $b^0(G,H) \leqs c$ then the generic stabilizer is finite (and if $b(G,H) \leqs c$, then some stabilizer is trivial). 

There is an example in \cite{GL} (in characteristic $2$) of a smooth affine (non-linear) $G$-variety with $G$ simple such that generic stabilizers do not exist.  

\begin{remark}
\emph{Related problems of this nature have been studied extensively in characteristic $0$ (and hence also for large positive characteristic $p$). Indeed, work of \`{E}la\v{s}vili \cite{Elas} and Popov \cite{Popov} provides a complete classification of the rational $KG$-modules $V$ of a simple algebraic group $G$ with the property that a generic stabilizer is finite (respectively, trivial). We refer the reader to \cite{Elas2, Popov2, encyc} for further results (again, in characteristic $0$) for semisimple algebraic groups $G$ and irreducible $KG$-modules. We thank an anonymous referee for drawing our attention to this important earlier work.}
\end{remark}

The interesting situation where the generic stabilizer is finite, but nontrivial, also comes
up extensively in recent work of Bhargava and coauthors (typically in characteristic $0$, but the analogous results in positive characteristic should lead to results concerning function fields over finite fields) -- see \cite{Bh1, Bh2, Bh3, BW}.  

Here we record some special cases that follow immediately from our earlier results.  

\begin{corol} \label{cor:genericstab}    
Let $G$ be a simple algebraic group over an algebraically closed field $K$ of characteristic $p \geqs 0$.   
\begin{itemize}\addtolength{\itemsep}{0.2\baselineskip}
\item[{\rm (i)}] If $p \ne 2$, $G=\SL(U)$ and $W=\sym^2(U) \oplus \sym^2(U)$, then the generic stabilizer of a vector in $W$ is a finite nontrivial elementary abelian $2$-group.
\item[{\rm (ii)}]  Let $G=E_6$ and $W=V \oplus V \oplus V \oplus V$, where $V$ is an irreducible $KG$-module of dimension $27$. Then the generic stabilizer of a vector in $W$ is trivial.  
\item[{\rm (iii)}] Let $G=G_2$ and $W =V \oplus V \oplus V$, where $p \ne 2$ and $V$ is an irreducible $KG$-module of dimension $7$. Then the generic stabilizer of a vector in $W$ is trivial. 
\end{itemize}
\end{corol}

Let us also give an example for a tensor product action.   Suppose that $G_1, G_2$ are subgroups of $\GL(V)$, and consider the natural action of $G=G_1 \otimes G_2$ on 
$W=V \otimes V$.  It is straightforward to see that the generic stabilizer of a vector in $W$ is the intersection of generic conjugates of $G_1$ and $G_2$.  In particular, we obtain the following corollary:

\begin{corol} \label{cor:soxso}   
Consider the natural action of $G =\SO(V) \otimes \SO(V)$  on $W=V \otimes V$. If $p \ne 2$ then the generic stabilizer of a vector in $W$ is a finite nontrivial elementary abelian $2$-group.
\end{corol}

\vs

Our results for algebraic groups have interesting consequences for the corresponding finite groups of Lie type.
Let us briefly recall the general set-up. Let $p$ be a prime, let $G$ be a simple algebraic group over the algebraic closure
$\bar{\mathbb{F}}_{p}$ of the prime field $\mathbb{F}_{p}$, and let $\s$ be a Frobenius morphism of $G$ such that the set of fixed points $G_{\s}$
is a finite group of Lie type over $\F$, for some $p$-power $q$. If $H$ is a closed positive-dimensional
$\s$-stable subgroup of $G$  then we can consider the action of $G_{\s}$ on the set of cosets of $H_{\s}$ in
$G_{\s}$. We write $b(G_{\s},H_{\s})$ for the base size of $G_{\s}$ in this action.

For a positive integer $c$, let $P(G_{\s},c)$ be the probability that $c$ randomly chosen  points in $G_{\s}/H_{\s}$
form a base for $G_{\s}$. We define the \emph{asymptotic base size} of $G_{\s}$, denoted by $b^{\infty}(G_{\s},H_{\s})$,
to be the smallest value of $c$ such that $P(G_{\s},c)$ tends to $1$ as $q$ tends to infinity.
With this set-up, there are five base-related numbers to consider:
$$b(G,H),\; b^0(G,H),\; b^1(G,H),\; b(G_{\s},H_{\s}),\; b^{\infty}(G_{\s},H_{\s}).$$

In Section \ref{s:prel} of this paper we investigate various relations between these base measures. For example,
in Proposition \ref{p:bb2} we use the Lang-Weil estimates to prove that the asymptotic base size of $G_{\s}$
coincides with the generic base size of $G$. We also show that $b^0(G,H) \leqs b(G_{\s},H_{\s})$ if $q>2$. In view of Theorem
\ref{t:bur}, the former observation implies that if $G$ is a classical group in a suitable non-subspace action
then $b^{\infty}(G_{\s},H_{\s}) \leqs 3$ if $\dim V >7$, where $V$ is the natural $G$-module. See \cite[Theorem 1.11]{LSh}
for a similar result, requiring the stronger condition $\dim V >15$. Similarly, if $G$ is an exceptional
algebraic group then using Theorems \ref{t:mainep1} and \ref{t:mainep2} we can compute the precise asymptotic
base size $b^{\infty}(G_{\s},H_{\s})$ in almost all cases; this is a significant strengthening of the general estimate
$b^{\infty}(G_{\s},H_{\s}) \leqs 6$ stated in \cite[Theorem 2]{BLS}.

Recall that if $G$ is a \emph{non-standard} finite almost simple primitive permutation group with point stabilizer
$H$ (so $G$ is not an alternating or symmetric group acting on subsets or partitions, nor a classical group in a subspace action) then $b(G) \leqs 7$, with equality if and only if $G={\rm M}_{24}$ in its
$5$-transitive action on $24$ points. The main theorem of \cite{BLS} reveals that there are infinitely many non-standard groups $G$ with $b(G)=5$, but it is not known whether or not there are infinitely many with $b(G)=6$. Indeed, to date the only known examples $(G,H)$ with $b(G,H)=6$ are the following:
$$(E_{6}(2),P_1), (E_6(2),P_6), ({\rm M}_{23}, {\rm M}_{22}), ({\rm Co}_{3}, {\rm McL}.2), ({\rm Co}_{2},
{\rm U}_{6}(2).2), ({\rm Fi}_{22}.2, 2.{\rm U}_{6}(2).2).$$

Now, according to Theorem \ref{t:mainep1} we have $b^0(G,H)=6$ if
$$(G,H) \in \{(E_6,P_1), (E_6,P_6), (E_7,P_7)\}.$$
Therefore, if $q>2$ we deduce that $b(G_{\s},H_{\s}) \geqs 6$ for the corresponding primitive  actions of $G_{\s}=E_6(q)$ and $E_7(q)$. Now the main theorem of \cite{BLS} yields $b(G_{\s},H_{\s}) \leqs 6$, so $b(G_{\s},H_{\s})=6$ for all
$q>2$ and we conclude that there are infinitely many
non-standard primitive groups with base size $6$ (see Remark \ref{r:cam}).

\begin{theorem}\label{t:cam}
There are infinitely many non-standard finite almost simple primitive permutation groups $G$ with $b(G)=6$.
\end{theorem}

\vs

Finally, we make some remarks on the organisation of the paper. In Section \ref{s:prel}
we present a number of preliminary results that we need for the proof of our main theorems.
Two key results here are Proposition \ref{p:bb} and Theorem \ref{t:con}, which provide effective
lower and upper bounds on the base measures $b^0(G,H)$ and $b^1(G,H)$, respectively.
By considering the fixed points of a Frobenius morphism $\s$, we also investigate the connection
between the base sizes of the algebraic group $G$ and the corresponding finite group $G_{\s}$;
see Proposition \ref{p:bb2}. Next, in Section \ref{s:inv} we consider the special case
where $H=C_{G}(\tau)$ for some involution $\tau \in {\rm Aut}(G)$ (with $p \neq 2$), proving  Theorem \ref{inv:main}.
The next two sections of the paper deal with the remaining primitive actions of classical and
exceptional algebraic groups, respectively, and we complete the proofs of
Theorems \ref{t:csub} -- \ref{t:mainep2}. In Section \ref{s:class} we make a distinction between subspace and non-subspace actions of classical groups; subspace actions are handled in Section \ref{ss:red}, and the remaining possibilities are
considered in Section \ref{ss:irred}. Similarly, in Section \ref{s:exc} we distinguish between
parabolic and non-parabolic actions of exceptional algebraic groups.
Finally, in Section \ref{s:zin} we establish Theorem \ref{t:zin} and we sketch the proof of Corollary \ref{c:zin}.

\section{Preliminaries}\label{s:prel}

In this section we record a number of preliminary results that we will need in the proof of our main theorems. Throughout this section, unless stated otherwise, the terms `variety' and `algebraic group' refer respectively to an algebraic variety and an affine algebraic group defined over an algebraically closed field $K$ of characteristic $p \geqs 0$. We begin with two elementary results on fibers of morphisms. The first result is well known.

\begin{lem}\label{l:gn1}
Let $\phi:X \to Y$ be a morphism of irreducible varieties. Then there exists a non-empty open subvariety $U$ of $\overline{\phi(X)}$ such that each fiber $\phi^{-1}(u)$ has the same dimension for all $u \in U$. Moreover, if $u \in U$ then $\dim \phi^{-1}(u) \leqs \dim \phi^{-1}(v)$ for all $v \in \phi(X)$.
\end{lem}

\begin{lem}\label{l:fm}
Let  $\phi:X \rightarrow Y$ be a dominant morphism of irreducible varieties
such that $\phi^{-1}(y)$ is non-empty and finite for some
$y \in Y$. Then there exists a non-empty open subvariety $U$ of
$Y$, and a positive integer $n$, such that
$|\phi^{-1}(u)| = n$ for all $u \in U$.
\end{lem}

\begin{proof}
This is also well known (see \cite[Corollaire 4]{Groth}), but we give a proof for completeness. First observe that $\dim X = \dim Y$ by Lemma \ref{l:gn1}, so
$K(X)/K(Y)$ is a finite algebraic field extension.
Then by \cite[Theorem 5.1.6(iii)]{springerbook}, we can take $n$ to be the separable degree of $K(X)/K(Y)$.
\end{proof}

\begin{lem}\label{l:fm2}
Let $G$ be an algebraic group, let $\Omega$ be an irreducible $G$-variety and let $\Gamma = \Omega \times \cdots \times \Omega$ with $c \geqs 1$ factors. Set
$$\mu = \min\left\{\dim \left(\bigcap_{i=1}^{c}G_{\a_i}\right) \mid (\a_1, \ldots, \a_c) \in \Gamma \right\}.$$
\begin{itemize}\addtolength{\itemsep}{0.2\baselineskip}
\item[{\rm (i)}] The subset $\{(\a_1, \ldots, \a_c) \in \Gamma \mid \dim \left(\bigcap_{i}G_{\a_i} \right) = \mu \}$ contains a non-empty open subvariety of $\Gamma$.
\item[{\rm (ii)}] If $\mu=0$ then there exists a non-empty open subvariety $U$ of $\Gamma$, and a positive integer $n$, such that $|\bigcap_{i}G_{\a_i}| \leqs n$ for all $(\a_1, \ldots, \a_c) \in U$.
\end{itemize}
\end{lem}

\begin{proof}
We may assume $G$ is connected. Consider the morphism of irreducible varieties
$\phi:G \times \Gamma \rightarrow  \Gamma \times \Gamma$
defined by
$$\phi:(g,\a_1, \ldots, \a_c) \mapsto (g\a_1, \ldots, g\a_c,\a_1, \ldots, \a_c).$$
If $z = (g\a_1, \ldots, g\a_c,\a_1, \ldots, \a_c) \in {\rm im}(\phi)$ then the fiber $\phi^{-1}(z)$ is isomorphic to
$\bigcap_{i}G_{\a_i}$. Therefore Lemma \ref{l:gn1} implies that there exists a non-empty open subvariety $U$ of $\overline{\phi(G \times \Gamma)}$ such that $\dim \phi^{-1}(z) = \mu$ for all $z \in U$. Part (i) now follows since $\phi$ maps onto the second $\Gamma$ factor, and part (ii) follows immediately from Lemma \ref{l:fm}.
\end{proof}

\begin{lem}\label{l:new}
Let $G$ be an algebraic group and let $X,Y$ be faithful irreducible $G$-varieties. Suppose there exists a non-empty open subvariety $U$ of $X$ such that $G_{u}$ is finite for all $u \in U$. Then with respect to the induced action of $G$ on $\Gamma = X \times Y$, there exists a non-empty open subvariety $V$ of $\Gamma$ such that $G_v$ is trivial for all $v \in V$.
\end{lem}

\begin{proof}
Replacing $X$ by a suitable non-empty open subvariety, we may assume that there is an integer $n$ such that $|G_x|=n$ for all $x \in X$. Fix $x \in X$ and set $L=G_x$. For $y \in Y$, the $G$-stabilizer of $(x,y) \in \Gamma$ is $L \cap J$, where $J=G_y$. Suppose $z \in L$ is nontrivial and let $C_{\Gamma}(z)$ denote the set of fixed points of $z$ on $\Gamma$. Then $C_{\Gamma}(z)$ is a proper closed subvariety of $\Gamma$ (since $G$ acts faithfully on $X$ and $Y$), so the finite union $\bigcup_{1 \neq z \in L}{C_{\Gamma}(z)}$ is also contained in a proper closed subvariety of $\Gamma$. Therefore, for each $x \in X$, the set
$$\{(x,y) \mid y \in Y, \; G_{x} \cap G_{y} \neq 1\}$$
is contained in a proper closed subvariety of $\{x\} \times Y$.

Let $\pi:X \times Y \to X$ be the projection map and set
$W=\{(x,y) \in \Gamma \mid G_{x}\cap G_{y} \neq 1\}$. For each $x \in X$ we have $\pi^{-1}(x) = \{x\} \times Y$, so by the above argument we deduce that $\dim (W \cap \pi^{-1}(x))  < \dim Y$. Therefore
$$\dim W \leqs \dim X + \dim (W \cap \pi^{-1}(x)) < \dim X + \dim Y = \dim \Gamma,$$
hence $V=\Gamma \setminus \overline{W}$ is a non-empty open subvariety such that $G_v=1$ for all $v \in V$.
\end{proof}

Let $G$ be a connected algebraic group and let $\Omega$ be a faithful transitive $G$-variety with point stabilizer $H$. Let $b(G,H)$ (or just $b(G)$ if the context is clear) denote the \emph{base size} of the action of $G$ on $\Omega$, so $b(G,H)$ is the minimal integer $c$ such that $\Omega$ contains $c$ points with the property that their pointwise stabilizer in $G$ is trivial. As advertised in the Introduction, we will also study two new base-related measures, which are defined as follows:
\begin{itemize}\addtolength{\itemsep}{0.2\baselineskip}
\item[(i)] The \emph{connected base size}, denoted $b^0(G,H)$, is the smallest integer $c$ such that $\Omega$ contains $c$ points whose pointwise stabilizer has trivial connected component, i.e. the pointwise stabilizer is finite.
\item[(ii)] The \emph{generic base size}, denoted $b^1(G,H)$, is the smallest integer $c$ such that the product variety $\Omega^c = \Omega \times \cdots \times \Omega$ ($c$ factors) contains a non-empty open subvariety $\L$ and every $c$-tuple in $\L$ is a base for $G$.
\end{itemize}
From the definitions, it is clear that
$$b^0(G,H) \leqs b(G,H) \leqs b^1(G,H).$$
The next result records some additional properties of these base measures.

\begin{prop}\label{p:bb}
Let $G$ be a connected algebraic group and let $\Omega=G/H$ be a faithful transitive $G$-variety. Then the following hold:
\begin{itemize}\addtolength{\itemsep}{0.2\baselineskip}
\item[{\rm (i)}] $b^0(G,H) \leqs \dim H +1$.
\item[{\rm (ii)}] If $b^0(G,H)=c$ then there exists a non-empty open subvariety $U$ of $\Omega^c$ such that $|\bigcap_{i}G_{\a_i}|$ is finite for all $(\a_1, \ldots, \a_c) \in U$.
\item[{\rm (iii)}] $b^0(G,H) \geqs \dim G/\dim \Omega$.
\item[{\rm (iv)}] $b^1(G,H) \leqs b^0(G,H)+1$.
\item[{\rm (v)}] If $H$ is finite and nontrivial then $b^0(G,H)=1$ and $b(G,H)=b^1(G,H)=2$.
\end{itemize}
\end{prop}

\begin{proof}
First consider (i). We may assume $H$ is positive-dimensional. Let $H_j$ be the intersection of $j$ distinct $G$-conjugates of $H$ and assume $\dim H_j>0$. Let $K_j$ denote the connected component of $H_j$. We claim that there exists an intersection $H_{j+1}\leqs H_j$ of $j+1$ conjugates of $H$ such that $\dim H_{j+1} < \dim H_j$. If not, then $K_j = K_j \cap H^g$ for all $g \in G$, which implies that $K_j$ is a positive-dimensional normal subgroup of $G$ contained in $H$. This is a contradiction since $H$ is core-free. It follows that there is a chain of subgroups
$$H = H_1 > H_2 > \cdots > H_m,$$
where each $H_j$ is an intersection of $j$ conjugates of $H$, $H_m$ is finite and $\dim H_{j+1} < \dim H_j$ for all $j$. Therefore $b^0(G,H) \leqs m$ and the bound in (i) follows since the dimension drops by at least $1$ at each stage of the above chain. For the remainder set $b^0(G,H)=c$.

Part (ii) follows immediately from Lemma \ref{l:fm2}(ii), and part (iii) is an easy consequence of the fact that $G$ has an orbit of dimension $\dim G$ on the product variety $\Omega^{b^0(G,H)}$.
For part (iv), consider the induced action of $G$ on $X \times Y$, where $X=\Omega^c$ and $Y=\Omega$. By applying part (ii) (with respect to the action of $G$ on $X$) it follows that the hypotheses of Lemma \ref{l:new} are satisfied, whence $b^1(G,H) \leqs c+1$ as required. Finally, part (v) follows immediately from (iv).
\end{proof}

\begin{remk}
Let $G$ be a simple algebraic group over an algebraically closed field of characteristic $p \geqs 0$ and suppose $\Omega=G/H$ is a faithful primitive $G$-variety.
\begin{itemize}\addtolength{\itemsep}{0.2\baselineskip}
\item[(i)] There are examples with $b^0(G,H) > \lceil \dim G/\dim \Omega \rceil$. For instance,
if $G=E_6$, $H=A_1A_5$ and $p \neq 2$ then $\dim G = 78$, $\dim H = 38$ and $b^0(G,H)=3$ (see Lemma \ref{l:e6ai}). Similarly, if $G={\rm Sp}_{6}$ and $H$ is of type ${\rm Sp}_{2} \wr S_3$ then $\dim G = 21$, $\dim H = 9$ and $b^0(G,H)=3$ (see Section \ref{ss:irred}).
\item[(ii)] By inspecting the proof of Theorems \ref{t:csub} -- \ref{t:mainep2}, we observe that $b^0(G,H) = b(G,H)$ in almost all cases. Indeed, the only known exceptions are the cases with $H$ finite.
\end{itemize}
\end{remk}

Let $G$ be a simple algebraic group over the algebraic closure of $\mathbb{F}_p$, where $p$ is a prime, and let $\Omega=G/H$ be a faithful $G$-variety. Let $\s:G \to G$ be a Frobenius morphism of $G$, so the set of fixed points $G_{\s}$ is a finite group of Lie type over $\F$ for some $p$-power $q$. Assume $H$ is $\s$-stable. Then the action of $G$ on $\Omega$ induces an action of $G_{\s}$ on $G_{\s}/H_{\s}$, and we write $b(G_{\s},H_{\s})$ for the corresponding base size. In addition, let
$b^{\infty}(G_{\s},H_{\s})$ be the \emph{asymptotic base size} of $G_{\s}$, which is the smallest integer $c$ such that $P(G_{\s},c)$ tends to $1$ as $q$ tends to infinity, where $P(G_{\s},c)$ is the probability that $c$ randomly chosen elements of $G_{\s}/H_{\s}$ form a base for $G_{\s}$.

By definition, if $b^{\infty}(G_{\s},H_{\s})=c$ and $q$ is sufficiently large then almost every $c$-tuple of points in $G_{\s}/H_{\s}$ forms a base for $G_{\s}$.
Notice that the generic base size $b^1(G,H)$ of $G$ captures this asymptotic property at the algebraic group level, in the sense that if $b^1(G,H)=c$ then there exists a dense subset $\L$ of $\Omega^c$ such that every $c$-tuple in $\L$ is a base for $G$.
Part (i) of the next result reveals that the generic and asymptotic base sizes do indeed coincide.

\begin{prop}\label{p:bb2}
With the notation established, the following hold:
\begin{itemize}\addtolength{\itemsep}{0.2\baselineskip}
\item[{\rm (i)}] $b^{\infty}(G_{\s},H_{\s})=b^1(G,H)$.
\item[{\rm (ii)}] If $q>2$ then $b^0(G,H) \leqs b(G_{\s},H_{\s})$.
\item[{\rm (iii)}] If $q$ is sufficiently large, then $b(G_{\s},H_{\s}) \leqs b^{\infty}(G_{\s},H_{\s})$.
\end{itemize}
\end{prop}

\begin{proof}
First consider (i). Suppose $b^1(G,H)=c$ and set $\Gamma = \Omega^c$ and
$$\L = \{(\a_1, \ldots, \a_c) \in \Gamma \mid \bigcap_{i}G_{\a_{i}} \neq 1 \},$$
so $\L$ is contained in a proper closed subvariety of $\Gamma$. Let $\L(q)$ and $\Gamma(q)$ denote the set of $\mathbb{F}_{q}$-rational points in $\L$ and $\Gamma$, respectively.
By considering the Lang-Weil estimates \cite{LW}, we deduce that the ratio
$$1-P(G_{\s},c)=\frac{|\L(q)|}{|\Gamma(q)|} \approx q^{\dim \L - \dim \Gamma}$$
tends to zero as $q$ tends to infinity, whence $b^{\infty}(G_{\s},H_{\s}) \leqs b^1(G,H)$. A similar argument shows that $b^1(G,H) \leqs b^{\infty}(G_{\s},H_{\s})$, hence equality holds.

Next consider (ii). Suppose $b(G_{\s},H_{\s})=c$ and $b^0(G,H)>c$. Fix distinct points $\a_1, \ldots, \a_c$  in $\Omega$ and set $L=\bigcap_{i}G_{\a_{i}}$. Since $b^0(G,H)>c$, the connected component $L^0$ is infinite, whence the hypothesis $q>2$ implies that $(L^0)_{\s}$ is nontrivial (see \cite[Proposition 8.1]{GG}). In particular, the stabilizer in $G_{\s}$ of any $c$ points in $G_{\s}/H_{\s}$ is nontrivial. This is a contradiction, hence (ii) follows.

Finally, note that if $q$ is sufficiently large then $P(G_{\s},c)>0$, where $c=b^{\infty}(G_{\s},H_{\s})$, so $G_{\s}$ admits a base of size $c$ and (iii) follows.
\end{proof}

\begin{remk}
There are examples with $b(G_{\s},H_{\s})<b^{\infty}(G_{\s},H_{\s})$ for all values of $q$. For example, if $nq$ is odd, $G_{\s}={\rm PGL}_{n}(q)$ and $H$ is of type $O_{n}(q)$ then $b(G_{\s},H_{\s})=2$ and $b^{\infty}(G_{\s},H_{\s})=3$ (see \cite{BGS2}).
\end{remk}

The main goal of this paper is to determine the three base measures $b^0(G)$, $b(G)$ and $b^1(G)$ for every primitive action of a simple algebraic group $G$. Of course, if $b^0(G) \geqs c$ and $b^1(G) \leqs c$ for an integer $c$, then we immediately deduce that
$$b^0(G) = b(G) = b^1(G) = c.$$
Therefore, our initial aim is to obtain accurate lower and upper bounds on $b^0(G)$ and $b^1(G)$, respectively. In Proposition \ref{p:bb}(iii) we established a useful lower bound on $b^0(G)$, so let us consider the generic base size $b^1(G)$. Our main result is Theorem \ref{t:con} below, which provides an effective upper bound on $b^1(G,H)$ in terms of the dimensions of some specific conjugacy classes in $G$ and $H$, assuming that $H^0$ is reductive. In order to prove this key theorem, we require a couple of preliminary results.

In \cite{Lusztig}, Lusztig proved that a simple algebraic group contains only finitely many conjugacy classes of unipotent elements. We require the following extension to algebraic groups with reductive connected component.

\begin{lem}\label{l:fc}
Let $G$ be an algebraic group with $G^0$ reductive. Then there are only finitely many unipotent classes in $G$, and only finitely many conjugacy classes of elements of a given finite order.
\end{lem}

\begin{proof}
This is the main theorem of \cite{gurproc}.
\end{proof}

\begin{prop}\label{p:op}
Let $G$ be an algebraic group with $G^0$ reductive, and let $\Omega$ be an irreducible $G$-variety. Let $\mathcal{C}$ be the set of conjugacy classes of $G$ containing elements of prime order (or arbitrary nontrivial unipotent elements if $p=0$) and set $\L = \bigcup_{C \in \mathcal{C}}\Omega(C)$, where
$$\Omega(C)=\bigcup_{x \in C}{C_{\Omega}(x)}$$
and $C_{\Omega}(x)=\{\a \in \Omega \mid x\a=\a\}$ is the fixed point space of $x$. Then either $\L$ is contained in a proper closed subvariety of $\Omega$, or $\Omega(C)$ contains a non-empty open subvariety of $\Omega$ for some $C \in \mathcal{C}$.
\end{prop}

\begin{proof}
Suppose $\L$ is not contained in a proper closed subvariety of $\Omega$, so $\L$ is dense in $\Omega$.
It suffices to show that $\Omega(C)$ is dense in $\Omega$ for some $C \in \mathcal{C}$: if $x \in C$ then $C_{\Omega}(x)$ is closed and the morphism $\phi :G \times C_{\Omega}(x) \rightarrow \Omega$ defined
by $\phi(g,\a)=g\a$ has image $\Omega(C)$, so $\Omega(C)$ contains a
non-empty open subvariety of $\overline{\Omega(C)}=\Omega$. Set
$m=\min\{\dim G_{\a} \mid \a \in \Omega\}$ and note that $\L = \{\a \in \Omega \mid G_{\a} \neq 1\}$.

If $m>0$ then every $\a \in \Omega$ is fixed by a torus or a unipotent subgroup of $G$, and so either by an element of order $2+\delta_{2,p}$, or a unipotent element (of order $p$ if $p>0$). By Lemma \ref{l:fc}, there are only finitely many $G$-classes of
such elements, say $C_1, \ldots, C_r$, whence $\Omega=\bigcup_{i=1}^r\Omega(C_i)$ and the irreducibility of $\Omega$ implies that $\Omega(C_i)$ is dense in $\Omega$ for some $i$.

Finally, suppose $m=0$. By Lemma \ref{l:fm2}, there exists a non-empty open subvariety $U$ of $\Omega$ and a positive integer $n$ such that $|G_{u}| \leqs n$ for all $u \in U$. If $n=1$ then $U \subseteq \Omega \setminus \L$, which contradicts our initial assumption. Therefore $n>1$ and thus $U=\bigcup_{i=1}^s\Omega(C_i)$ for some $G$-classes $C_i$ of elements of prime order dividing $n$. Since $U$ is irreducible we deduce that $\Omega=\overline{U}=\overline{\Omega(C_i)}$ for some $i$, as required.
\end{proof}

\begin{cor}\label{c:gammac}
Let $G$ be a simple algebraic group and let $\Omega=G/H$ be a transitive $G$-variety with $H^0$ reductive. Let $c \geqs 2$ be an integer and let $\Gamma$ be the irreducible $G$-variety $\Omega^{c-1}$. If
$$\dim \Gamma(C) < \dim \Gamma$$
for every $H$-class $C$ of elements of prime order in $H$ (including all nontrivial unipotent elements if $p=0$) then $b^1(G,H) \leqs c$.
\end{cor}

\begin{proof}
Consider the action of $H$ on $\Gamma$. By Proposition \ref{p:op}, the hypothesis $\dim \Gamma(C) < \dim \Gamma$ for all relevant $H$-classes $C$ implies that $\{\a \in \Gamma \mid H_{\a}=1\}$ contains a non-empty open subvariety of $\Gamma$. We conclude that $b^1(G,H) \leqs c$.
\end{proof}

\begin{lem}\label{lls}
Let $G$ be an algebraic group, let $H$ be a closed subgroup of $G$ and let $\Omega=G/H$. Then for $x \in H$,
$$\dim C_{\Omega}(x) = \dim \Omega - \dim x^G +\dim (x^G \cap H).$$
\end{lem}

\begin{proof}
This is \cite[Proposition 1.14]{LLS}.
\end{proof}

Let $\mathcal{P}$ be the set of elements of prime order in $H$ (including all nontrivial unipotent elements in $H$ if $p=0$) and let $c \geqs 2$ be an integer. We define
$$\mathcal{Q}(G,c) = \frac{c}{c-1}\cdot \sup_{x\in \mathcal{P}}\left\{\frac{\dim (x^G \cap H)}{\dim x^G}\right\}.$$
The next result is a key tool in our later analysis.

\begin{thm}\label{t:con}
Let $G$ be a simple algebraic group and let $\Omega=G/H$ be a transitive $G$-variety, where  $H^0$ is reductive. Let $c \geqs 2$ be an integer such that $\mathcal{Q}(G,c)<1$. Then $b^1(G,H) \leqs c$.
\end{thm}

\begin{proof}
Let $x \in \mathcal{P}$ and set $C=x^H$, $\Gamma= \Omega^{c-1}$ and $\Gamma(C) = \bigcup_{y \in C}{C_{\Gamma}(y)}$. By Corollary \ref{c:gammac}, we need to show that $\dim \Gamma(C) < \dim \Gamma$.

First we claim that
$$\dim \Gamma(C) \leqs (c-1)\dim C_{\Omega}(x) + \dim x^H.$$
To see this, let $C_{\Omega}(x)^{c-1} = C_{\Omega}(x) \times \cdots \times C_{\Omega}(x)$ (with $c-1$ factors) and consider the morphism
\begin{equation}\label{e:mor}
\phi: H \times C_{\Omega}(x)^{c-1} \to \overline{\Gamma(C)}
\end{equation}
sending $(h, \a_1, \ldots, \a_{c-1})$ to $(h\a_1, \ldots, h\a_{c-1})$.
Now ${\rm im}(\phi)=\Gamma(C)$ and
$$\phi((hy,y^{-1}\a_1, \ldots, y^{-1}\a_{c-1})) = \phi((h,\a_1, \ldots, \a_{c-1}))$$
for all $y \in C_{H}(x)$, so $\dim \phi^{-1}(\a) \geqs \dim C_{H}(x)$ for all $\a \in {\rm im}(\phi)$. Therefore
$$\dim \Gamma(C) \leqs \dim H+(c-1)\dim C_{\Omega}(x)  - \dim C_{H}(x)$$
as claimed.

Now, Lemma \ref{lls} gives
$$\dim C_{\Omega}(x) = \dim \Omega-\dim x^G +\dim (x^G\cap H),$$
and we may assume $x^G \cap H=x^H$ since $x^G\cap H$ is a finite union of $H$-classes (see \cite[Theorem 1.2]{gurproc}). Since $c\dim x^H < (c-1)\dim x^G$ we conclude that
$$\dim \Gamma(C) \leqs (c-1)\dim C_{\Omega}(x) + \dim x^H = \dim \Gamma-(c-1)\dim x^G +c\dim x^H < \dim \Gamma,$$
as required.
\end{proof}

\begin{cor}\label{c:con}
Let $c \geqs 2$ be an integer such that
$$\dim x^H < (1-c^{-1})\dim x^G$$
for all $x \in \mathcal{P}$. Then $b^1(G,H) \leqs c$.
\end{cor}

\begin{prop}\label{red1}
Suppose $\Omega=G/H$ with $H^0$ reductive. Let $x \in H$ be a semisimple element of prime order such that
$$\dim Z(C_{G}(x)^0)+{\rm rank\;} H > {\rm rank\;}G.$$
Then there exists $y \in H$ of order $2+\delta_{2,p}$ such that $C_{\Omega}(x) \subseteq C_{\Omega}(y)$.
\end{prop}

\begin{proof}
The bound $\dim Z(C_{G}(x)^0)+{\rm rank\;} H > {\rm rank\;}G$ implies that $Z(C_{G}(x)^0) \cap H$ contains a positive-dimensional torus $L$. Choose $y \in L$ so that $x$ and $y$ are of coprime order and set $z=xy$. We claim that $C_{\Omega}(z)=C_{\Omega}(x)$.

Clearly,  $C_{\Omega}(z) \subseteq C_{\Omega}(x)$ since $x$ is a power of $z$.
Let $T$ be a maximal torus of $H$ containing $L$.
Let $x = x_1,x_2, \ldots, x_m$ in $T$ represent the distinct $H$-classes in $x^G \cap H$,
so $x_i=x^{w_i}$ for some $w_{i}$ in the Weyl group of $H$. Also set $z_i=y^{w_i}$ for all $i$, and note that the $z_{i}$ are also in distinct $H$-classes.
It is an easy exercise to see that $C_{\Omega}(x)$ is the union of $m$
disjoint sets, each an orbit of $C_G(x_i)/C_H(x_i)$ for some $i$. Similarly, $C_{\Omega}(z)$
contains the $C_G(z_i)/C_H(z_i)$-orbits.
Since $C_{H}(z_i) \leqs C_{H}(x_i)$
and $C_G(x)=C_G(z)$, it follows that $C_{\Omega}(x)  \subseteq C_{\Omega}(z)$. This justifies the claim.

Let $S$ be the set of elements $xy \in L$ such that the order of $y \in L$ is relatively prime to the order of $x$. Then $S$ is dense in $L$, so by the previous claim we have $C_{\Omega}(x)=C_{\Omega}(z)$ for all $z \in S$ and thus
$$C_{\Omega}(x) = C_{\Omega}(\overline{S})=C_{\Omega}(L)=\bigcap_{y \in L}C_{\Omega}(y).$$
In particular, if $y \in L$ has order $2+\delta_{2,p}$ then $C_{\Omega}(x) \subseteq C_{\Omega}(y)$, as required.
\end{proof}

Note that if $H$ has maximal rank then Proposition \ref{red1} applies to any semisimple element $x \in G$ with $\dim Z(C_{G}(x)^0)>0$.
The next result allows us to slightly weaken the conditions in the statement of Corollary \ref{c:con}. In order to state the result, let $\mathcal{P}' \subseteq \mathcal{P}$ be the union of the set of unipotent elements in $\mathcal{P}$ and the set of semisimple elements $x \in \mathcal{P}$ with the property that either $x$ has order $2+\delta_{2,p}$ or
$$\dim Z(C_{G}(x)^0)+{\rm rank\;} H \leqs {\rm rank\;}G.$$

\begin{cor}\label{c:sembd}
Let $c \geqs 2$ be an integer such that
$$\dim x^H < (1-c^{-1})\dim x^G$$
for all $x \in \mathcal{P}'$. Then $b^1(G,H) \leqs c$.
\end{cor}

\begin{proof}
By Corollary \ref{c:gammac} it suffices to show that $\dim \Gamma(C)<\dim \Gamma$, where $\Gamma=\Omega^{c-1}$, $C=x^H$, $x \in H$ is semisimple of prime order $r>2+\delta_{2,p}$ and
$$\dim Z(C_{G}(x)^0)+{\rm rank\;} H > {\rm rank\;}G.$$
By Proposition \ref{red1}, there exists an element $z \in H$ of order $2+\delta_{2,p}$ such that $C_{\Omega}(x) \subseteq C_{\Omega}(z)$, so $\Gamma(C) \subseteq \Gamma(C')$ for $C'=z^H$. Since $z \in \mathcal{P}'$ we have $\dim z^H < (1-c^{-1})\dim z^G$, so $\dim \Gamma(C') < \dim \Gamma$ as required.
\end{proof}

The next proposition is a generalization of a result of Guralnick et al. \cite{GL}.

\begin{prop}\label{p:newb}
Suppose $\Omega=G/H$ with $H^0$ reductive, and let $c \geqs 2$ be an integer such that the following conditions hold:
\begin{itemize}\addtolength{\itemsep}{0.2\baselineskip}
\item[{\rm (i)}] There exists a prime $r \neq p$ such that $\dim x^H < (1-c^{-1})\dim x^G$ for all $x \in H$ of order $r$;
\item[{\rm (ii)}] $\dim x^H \leqs (1-c^{-1})\dim x^G$ for all unipotent elements $x \in \mathcal{P}$.
\end{itemize}
Then $b^0(G,H) \leqs c$.
\end{prop}

\begin{proof}
Let $\Gamma=\Omega^{c-1}$. We need to show that there exists a non-empty open subvariety $U \subseteq \Gamma$ such that $H_{\a}:=\bigcap_{i=1}^{c-1}H_{\a_i}$ is finite for all $\a = (\a_1,\ldots, \a_{c-1}) \in U$. Seeking a contradiction, suppose that no such $U$ exists. Then $H_{\a}$ is infinite for all $\a \in \Gamma$, so $H_{\a}$ either contains a torus or a $1$-dimensional unipotent subgroup.

By (i), the set of $\a \in \Gamma$ such that $H_{\a}$ contains a torus is contained in a proper closed subvariety of $\Gamma$ (namely, the subvariety $\bigcup_{y \in \Lambda}C_{\Gamma}(y)$, where $\Lambda$ is the set of elements of order $r$ in $H$, and $r \neq p$ is the prime in (i)). Therefore, $(H_{\a})^0$ is unipotent for all $\a$ in a non-empty open subvariety of $\Gamma$. Since $H_{\a}$ is infinite and there are only finitely many unipotent classes in $H$ (see Lemma \ref{l:fc}), it follows that there is a nontrivial unipotent element $x \in H$ such that $\dim(x^H \cap H_{\a})>0$ for all $\a$ in a non-empty open subvariety $W \subseteq \Gamma$. Set $C=x^H$ and $\Gamma(C) = \bigcup_{y \in C}C_{\Gamma}(y)$. As noted in the proof of Theorem \ref{t:con}, $\Gamma(C)$ is the image of the morphism 
$$\phi: H \times C_{\Omega}(x)^{c-1} \to \Gamma$$ 
defined in \eqref{e:mor}, sending $(h,\a_1, \ldots, \a_{c-1})$ to $(h\a_1, \ldots, h\a_{c-1})$. Since $\dim(x^H \cap H_{\a})>0$ for all $\a \in W$, it follows that
$\Gamma(C)$ contains a non-empty open subvariety of $\Gamma$ and thus $\dim \Gamma = \dim \Gamma(C)$. Clearly, $\phi$ is still a dominant morphism if we replace $x$ by $x^i$ for any positive integer $i$, so if $p>0$ we can assume that $x$ has order $p$.

Set $V(x) = \{ (y,\a) \mid y \in x^H,\, \a \in \Gamma,\, y\a=\a\}$. Then $V(x)$ surjects onto $x^H$ and $\Gamma(C)$ via the two projection maps. Therefore, by considering the fibers of the first  projection we deduce that
$$\dim V(x) = \dim x^H + \dim C_{\Gamma}(x)$$
and by applying Lemma \ref{lls} and the condition in (ii) we get
$$\dim V(x)  = \dim \Gamma + c\dim x^H - (c-1)\dim x^G \leqs \dim \Gamma.$$
Similarly, the second projection shows that
$$\dim V(x) = \dim \Gamma(C) + \dim (x^H \cap H_{\a})$$
for some $\a \in \Gamma$ with $\dim(x^H \cap H_{\a})>0$. Therefore,
$$\dim \Gamma \geqs \dim V(x)>\dim \Gamma(C),$$
which is a contradiction.
\end{proof}

Let $X$ be a simple algebraic group with root system $\Phi$ and root subgroups 
$$U_{\a} = \{x_{\a}(t) \mid t \in K\}, \;\; \a \in \Phi.$$ 
Recall that if $\alpha$ is a long root then $U_{\a}$ is a \emph{long root subgroup}, and $x \in X$ is a \emph{long root element} if $x$ is $X$-conjugate to $x_{\a}(t)$ for some long root $\a$ and $t \in K^*$.

\begin{prop}\label{p:lr}
Suppose $\Omega=G/H$ with $H^0$ simple. Let $C=x^H$, where $x \in H^0$ is a long root element of $H^0$, and assume that each long root subgroup of $H^0$ is a long root subgroup of $G$. Let $c \geqs 2$ be an integer and set $\Gamma=\Omega^{c-1}$. Then
$$\dim \Gamma(C) \leqs \dim H + (c-1)\dim C_{\Omega}(x)-\dim C_{H}(x) -1.$$
\end{prop}

\begin{proof}
Define the morphism $\phi: H \times C_{\Omega}(x)^{c-1} \to \overline{\Gamma(C)}$ as in the proof of Theorem \ref{t:con} (see \eqref{e:mor}). It suffices to show that $\dim \phi^{-1}(\a) \geqs \dim C_{H}(x)+1$ for all $\a \in {\rm im}(\phi)$.

First we claim that $x$ belongs to a unique long root subgroup of $G$. To see this, let $U$ be a long root subgroup of $G$ containing $x$, and suppose that $x$ is contained in another long root subgroup $V$ of $G$. Set $P=N_G(U)$ and note that $P$ is a parabolic subgroup of $G$. Since $U$ and $V$ are $G$-conjugate we have $V=U^{awb}$, where $a,b \in P$ and $w$ is in the Weyl group of $G$. Therefore, $U \cap V = U \cap U^{wb}$ is conjugate to $U \cap U^w$, which is either trivial or equal to $U$. But $x \in U \cap V$, so $U=V$ as required.

Let $W$ be a long root subgroup of $H^0$ containing $x$. By assumption, $W$ is a long root subgroup of $G$, so the previous claim implies that $U=W \leqs H^0$. Therefore, $C_{\Omega}(x)$ is invariant under a $1$-dimensional torus normalizing $U$ (but not centralizing $x$), so the proof of Theorem \ref{t:con} provides the desired bound $\dim \phi^{-1}(\a) \geqs \dim C_{H}(x)+1$.
\end{proof}

\begin{cor}\label{c:lr}
Suppose $\Omega=G/H$ with $H^0$ simple. Assume that each long root subgroup of $H^0$ is a long root subgroup of $G$. Let $c \geqs 2$ be an integer such that
$$\dim x^H \leqs (1-c^{-1})\dim x^G$$
for all $x \in \mathcal{P}$, with equality only if $x \in H^0$ is a long root element. Then $b^1(G,H) \leqs c$.
\end{cor}

\vs

We close this preliminary section with some remarks on the underlying field. In the following sections we will often work over the algebraic closure $k$ of $\mathbb{F}_{p}$, but it is important to note that the same results hold if we replace $k$ by any algebraically closed field of characteristic $p$. Indeed, in almost all cases the arguments do not depend on the choice of field, but there are some cases where we deduce the results for the algebraic group $G$ from results for the corresponding finite group $G_{\s}$ (the fixed points of a suitable Frobenius morphism $\s$). We first make some elementary observations.

Fix a prime $p$ and let $K$ be an algebraically closed field of characteristic $p$.
Let $G$ be a simple algebraic group over $K$ and let $H$ be a maximal
closed subgroup of $G$. Note that $G$ and $H$ are both defined over $k$. Let $c$ be a positive integer and let $X_c$ denote the product variety $(G/H)^c$, with the natural $G$-action.

The algebraic group $G(K)$ has an orbit on $X_c(K)$ with finite stabilizer
if and only if the same is true for the action of $G(k)$ on $X_c(k)$ (since dimension
remains constant under base change). This shows that the connected base size $b^0(G,H)$ remains constant under base change.   Similarly, the generic stabilizer in $G(K)$ of a point in $X_c(K)$ is finite of size $n$ if and only if the same is true for the $G(k)$-stabilizer of a generic point in $X_c(k)$ (because this condition holds on an open subvariety). Therefore, the generic base size $b^1(G,H)$ is also constant under base change.

Next we show that the base size $b(G,H)$ cannot increase
under base change. Indeed, if $G(k)$ has a regular orbit on $X_c(k)$ then this orbit remains regular for $G(K)$ on $X_c(K)$  (the stabilizer in $G(K)$ is still zero-dimensional, and thus finite, but the orbits of ${\rm Aut}(K/k)$ are either trivial or infinite, so this stabilizer must be trivial). In fact, $b(G,H)$ is constant under base change, and all of these assertions follow from more general considerations -- see \cite[p.81-82]{EGA}.

Finally, by a straightforward ultraproduct construction, we see that
for some algebraically closed field of characteristic $0$ (and therefore
for all, by the previous discussion), the quantities $b^0(G,H)$, $b(G,H)$ and $b^1(G,H)$
in characteristic $0$ will be the same as for all sufficiently large characteristics.

\section{Involution-type subgroups}\label{s:inv}

Let $G$ be a simple algebraic group over an algebraically closed field $K$ of characteristic $p \geqs 0$. Let $\Omega$ be a primitive $G$-variety with point stabilizer $H$. In this section we consider the special case where $H=C_{G}(\tau)$ for an involution $\tau \in {\rm Aut}(G)$. Our goal here is to prove Theorem \ref{inv:main}, as well as some auxiliary results that may be of independent interest. We are mostly interested in the case $p \ne 2$, although some of the results will apply for all $p$. 

Our strategy is as follows. In Lemma \ref{invol2}  we show that there exists a unique $G$-class of involutions in ${\rm Aut}(G)$ which invert a maximal torus of $G$, say $C = x^G$. By Proposition \ref{main1}, every element of $G$ is a product of two elements in $C$, so there exists $g \in G$ such that $xx^g=u$ is a regular unipotent element of $G$, and thus $C_{G}(u)$ is abelian and contains no involutions for $p \ne 2$. Moreover, $x$ inverts $C_{G}(u)$ by Proposition \ref{cor1}, and using this we deduce that $b^0(G,H)=b(G,H)=2$ if $\tau \in C$, otherwise $b^0(G,H) \geqs 3$.
We begin by recording some preliminary lemmas.

\subsection{Preliminaries}\label{ss:prel}

We begin with an elementary result that will be needed in the proof of Proposition \ref{main1}. Here ${\rm M}_{n}(K)$ denotes the algebra of $n \times n$ matrices with entries in the field $K$.

\begin{lem}\label{l:linalg}
Let $K$ be an algebraically closed field and let $A \in {\rm M}_n(K)$.  Then $A$ is similar to a symmetric matrix.
\end{lem}

\begin{proof}
Let $p$ denote the characteristic of $K$. We may assume that $A$ is in Jordan canonical form. Furthermore, we may assume that $A$ is a single Jordan block, and we may take $A$ to be nilpotent. Let $B =(b_{i,j}) \in {\rm M}_{n}(K)$, where $b_{i,n+1-i}=1$ for $1 \leqs i \leqs n$, and all other entries are $0$. Then $B$ is symmetric and $B A B^{-1} =A^{\top}$, where $A^{\top}$ denotes the transpose of $A$.

If $p \neq 2$, or a diagonal entry of $B$ is non-zero (i.e. $n$ is odd), then we can write $B = X^{\top}X$ for some non-singular $X$ (because any such non-singular symmetric matrix is congruent to the identity matrix $I$, over an algebraically closed field). If $p=2$ and $n=2m$ is even, then replace $B$ by $B(I + A)$. Note that $B(I + A)$ is non-singular, symmetric and the $(m+1)$-th diagonal entry of $B(I+A)$ is non-zero, so we can write $B(I + A)=X^{\top}X$ for some $X$. Also note that $B(I+A)A(I+A)^{-1}B^{-1} = A^{\top}$.

It is easy to check that $XAX^{-1}$ is symmetric.
\end{proof}

For the remainder of Section \ref{ss:prel}, unless stated otherwise, $G$ will denote a connected reductive algebraic group of rank $r$ over an algebraically closed field $K$ of characteristic $p \geqs 0$.

\begin{lem} \label{involution1}
Let $G$ be a connected algebraic group over an algebraically closed field of characteristic $p \geqs 0$, and let $\tau \in {\rm Aut}(G)$ be an involution. Then either $G$ is abelian and $\tau$ inverts $G$, or $\dim C_{G}(\tau)>0$.
\end{lem}

\begin{proof}
By \cite[Theorem 7.2]{steinberg}, $\tau$ normalizes a Borel subgroup $B=TU$ of $G$, where $T$ is a maximal torus of $G$ and $U$ is the unipotent radical of $B$.
Assume that $\dim C_{G}(\tau)=0$, so $C_G(\tau)$ is finite.

First suppose $p=2$. If $U$ is nontrivial then $C_U(\tau)$ is infinite, so we may assume $G=T$ is a torus. An involutory automorphism of a torus corresponds to
an element of ${\rm GL}_r(\mathbb{Z})$ (acting on the character group of the torus), and  any involution other than $-I_r$ centralizes a positive-dimensional subtorus. Since $C_G(\tau)$ is finite, we conclude that $\tau$ inverts $G$.

Now assume that $p \neq 2$. First we show that $U$
is abelian and $\tau$ inverts $U$.  Let $W$ be a minimal
connected characteristic subgroup of $U$.  Then $W$ is a vector
space over $K$ and $C_{W}(\tau)$ is a subspace, hence $C_{W}(\tau)$ is trivial and thus $\tau$ acts as inversion on $W$. Since $2$ is invertible in $K$,
the fixed points of $\tau$ on $U$ surject onto the fixed points of $\tau$ on
$U/W$. So by induction, $U/W$ is abelian and $\tau$ acts as inversion
on $U/W$. Therefore $C_U(\tau) \leqs C_W(\tau)=1$.  Note that $[U,U] \leqs W$ and so since $\tau$ acts as inversion on $U/W$, $\tau$ centralizes $[U,U]$, whence $U$ is abelian.  Thus, $\tau$ acts as inversion on $U$.

Let $T_0$ be the subgroup of $T$ of elements of odd order. If we consider the action of
$\tau$ on $UT_0$ then the above argument implies that $UT_0$ is abelian and $\tau$ acts as inversion on $UT_0$.  Therefore $B = T \times U$ is abelian and
$\tau$ acts as inversion on $B$. Since $B$ is abelian, it follows that $G$ is solvable and thus $G=B$.
\end{proof}

\begin{lem}\label{l:ba2}
Let $\tau \in {\rm Aut}(G)$ be an involution and let $H=C_{G}(\tau)$. Let $g \in G$ and set
$D=C_G(\tau \tau^g)$.
\begin{itemize}\addtolength{\itemsep}{0.2\baselineskip}
\item[{\rm (i)}] $H \cap H^g = 1$ if and only if $\tau$ acts as inversion on $D$, $D$ is abelian
and $D$ contains no involutions.
\item[{\rm (ii)}] $H \cap H^g$ is finite if and only if $\tau$ acts as inversion on $D^0$, $D^0$ is abelian and  if $p=2$,  $D^0$ is a torus.
\end{itemize}
\end{lem}

\begin{proof}
Suppose $\tau$ inverts $D$.   Then clearly  $H \cap H^g$ consists of the involutions
in $D$, whence the backward implications of both statements follow.

Conversely, assume that $H \cap H^g$ is finite.   Then $C_D(\tau)$ is finite and the result follows from Lemma \ref{involution1}.
\end{proof}

\begin{lem}\label{l:ru}
Suppose $G$ is semisimple and let $x \in G$ be a regular unipotent element contained in a Borel subgroup $B=TU$ of $G$. Then $C_{G}(x) = A \times Z(G)$ where $A$ is an abelian subgroup of $U$. 
\end{lem}

\begin{proof}
This is well known. See \cite{Lou}, for example.
\end{proof}

\begin{cor}\label{c:ru}
Suppose $G$ is semisimple of adjoint type,  $p \neq 2$ and there exist conjugate involutions $\tau, \tau' \in {\rm Aut}(G)$ such that
$\tau\tau'$ is regular unipotent and $\tau$ inverts $C_{G}(\tau\tau')$. Then 
$C_{G}(\tau) \cap C_{G}(\tau')=1$.
\end{cor}

\begin{proof}
By Lemma \ref{l:ru}, $C_{G}(\tau\tau')$ is unipotent and therefore contains no involutions since $p \neq 2$. Now apply Lemma \ref{l:ba2}.
\end{proof}

\begin{lem} \label{invol2}
There exists an involution $\tau \in {\rm Aut}(G)$ that inverts a maximal torus $T$ of $G$, and any two such involutions are $G$-conjugate. Also, $\dim C_{G}(\tau) = \frac{1}{2}(\dim G - r)$. If $G$ is simple
and $p \ne 2$, the type of $C_{G}(\tau)$ and $\tau$ is recorded in Table \ref{t:gi}.
\end{lem}

\begin{table}[h]
$$\begin{array}{lll} \hline
G & \mbox{Type of $C_G(\tau)$} & \mbox{Type of $\tau$} \\ \hline
A_n & {\rm SO}_{n+1} & \mbox{inner if $n=1$, otherwise graph} \\
B_n & {\rm SO}_{n+1} \times {\rm SO}_{n} & \mbox{inner} \\
C_n & {\rm GL}_{n} & \mbox{inner} \\
D_n & {\rm SO}_{n} \times {\rm SO}_{n} & \mbox{inner if $n$ even, otherwise graph} \\
E_8 & D_8 & \mbox{inner} \\
E_7 & A_7 & \mbox{inner} \\
E_6 & C_4 & \mbox{graph}  \\
F_4 & A_1C_3 & \mbox{inner} \\
G_2 & A_1\tilde{A}_1 & \mbox{inner} \\ \hline
\end{array}$$
\caption{Involutions inverting maximal tori, $p \neq 2$}
\label{t:gi}
\end{table}

\begin{proof}   By passing to an isogeneous group, we may assume that $G = S \times G_1$
where $S$ is a torus and $G_1$ is a direct product of simple and simply connected groups.
We induct on $\dim G$.  The case of a torus is clear and so a minimal counterexample
would be a simple group (and again we may assume that it is simply connected). Existence now follows by inspection. (Note that $\tau$ is an involution modulo the center of $G$.) Moreover,
$\dim C_G(\tau) = \frac{1}{2}(\dim G - r)$ since $\tau$ permutes the root subgroups of $G$ without fixed points, and $C_{T}(\tau)$ is finite.

Suppose $\tau$ and $\tau'$ are involutory automorphisms of $G$ which invert a maximal torus. To see that $\tau$ and $\tau'$ are $G$-conjugate we may assume, without loss, that $\tau$ and $\tau'$ both invert the same maximal torus $T$, so $\tau\tau' \in C_{{\rm Aut}(G)}(T)=T$.
Therefore $\tau$ and $\tau'$ belong to the same coset of $T$ in $N_{{\rm Aut}(G)}(T)$, and since they both invert $T$ it follows that $\tau$ and $\tau'$ are $G$-conjugate.
\end{proof}

\begin{prop}\label{main1}
Let $\tau \in {\rm Aut}(G)$ be an involution that inverts a maximal torus of $G$, and 
assume that $p \neq 2$. Then
\begin{itemize}\addtolength{\itemsep}{0.2\baselineskip}
\item[{\rm (i)}] $\tau$ inverts an element in each conjugacy class of $G$; and
\item[{\rm (ii)}] if $C=\tau^G$ then $G=C^2$.
\end{itemize}
\end{prop}

\begin{proof}
Suppose (ii) holds. Let $g \in G$. Replacing $g$ by a conjugate, we may assume $g = \tau\tau^x$ for some $x \in G$. Then $z^{\tau}=z^{-1}$, where $z=g^{x^{-1}}$ is conjugate to $g$.
Therefore (ii) implies (i). Conversely, if $\tau$ inverts $g \in G$, then 
$\tau \tau^g = g^2$ and since $p \ne 2$, squaring is a surjective morphism on $G$, 
whence (i) implies (ii). Therefore, it suffices to show that (i) holds.

We may assume that $G=Z(G)^0 \times A$ where $A$ is simply connected and semisimple. By induction on $\dim G$, we reduce to the case where $G$ is simple and simply connected (the case when $G$ is
a torus is trivial). 

Let $X=x^G$ be a conjugacy class of $G$. If $x$ commutes with a non-central
semisimple element $t \in G$, then we may pass to the connected reductive group
$C_G(t)$  (recall that we are assuming that $G$ is simply connected).  We may assume
that $\tau$ inverts a maximal torus $T$ containing $t$.  In particular, $\tau$ acts on
the connected reductive group $C_G(t)$ which contains $T$, and by induction $\tau$ inverts some conjugate of $x$ in $C_G(t)$.  Consequently, we may 
assume that $x$ is a \emph{semiregular} unipotent element that commutes with no non-central
semisimple element. Thus, $C_G(x) = Z(G) \times U$ for some unipotent subgroup $U$.

Now $G$ contains a unique class of regular unipotent elements, so if $x$ is such an element then $x^y=x^{-1}$ for some $y \in G$.  We can assume that $y$ has order a power of $2$, hence $y$ is an involution (modulo $Z(G)$).   We want to show that $y$ is conjugate to $\tau$ (or equivalently, that $y$ inverts some maximal torus).  
Note that the dimension of the image of the multiplication
map $\mu:y^G \times y^G \to G$ is at least $\dim G - r$, where $r$ is the rank of $G$.  In particular,
\begin{equation}\label{eq1}
\dim y^G \geqs \frac{1}{2}(\dim G - r).
\end{equation}
We now inspect the various possibilities for $G$, beginning with the classical groups.

If $G=A_1$ or $G_2$, then $\tau$ and $y$ are both inner and since there is a unique
class of involutions (modulo the center in the case of $A_1$), the result follows.

If $G=A_{r}$ (with $r > 1$) then we take $\tau$ to be the inverse-transpose automorphism. 
Then $\tau$ inverts any symmetric matrix and the result follows by Lemma \ref{l:linalg}.

If $G=C_{r}$ (with $r > 1$) then an element is semiregular if and only if it is regular unipotent. 
Moreover, $\tau$ and $y$ are inner.  
Let $x$ be such an element and recall that $y \in G$ is an involution inverting $x$.  By conjugating
we may assume that $y$ is in the standard maximal torus of $G$ and so it inverts each
root subgroup corresponding to a simple root.  This forces $y$ to be conjugate to $\tau$. The same proof (without any modification) also applies for groups of type $B_{r}$ (with $r>1$).

Next suppose $G=D_r$ and $r > 3$.   If $x$ does not commute with a non-central semisimple element, then
$x$ has Jordan form $[J_{2e+1},J_{2f+1}]$ on the natural $G$-module, for some $e,f \geqs 0$. Then $x \in L:=B_e \times B_f$ (where $B_0$ is trivial), and we note that $L=C_G(z)$ for a suitable involutory graph automorphism $z$ of $G$. Let $y \in L$ be an involution inverting $x$ and a maximal torus of $L$.  Then either $y$ or $yz$ is conjugate to $\tau$, and inverts $x$.

Now let us assume $G$ is an exceptional algebraic group of rank $r$. Let $\{\a_1, \ldots, \a_r\}$ be a set of simple roots for the root system of $G$, where we label simple roots in the usual way (see \cite{Bou}).
Let $U_{\a} = \{x_{\a}(t) \mid t \in K\}$ be the root subgroup of $G$ corresponding to the root $\a$, and  write $\a=a_1a_2\cdots a_r$ to denote the root $\a=\sum_{i}a_i\a_i$. In addition, we adopt the standard Chevalley notation
$$n_{\a}(t) = x_{\a}(t)x_{-\a}(-t^{-1})x_{\a}(t),\;\; h_{\a}(t) = n_{\a}(t)n_{\a}(1)^{-1}$$
for $t \in K^*$.

If $G=F_4$ then $x$ is regular and $\dim y^G \geqs 24$ by \eqref{eq1}, whence 
$y$ and $\tau$ are $G$-conjugate.

Next suppose that $G=E_6$.  Again, we may assume that $x$ is a regular unipotent element. Further, we may also assume that $x \in H:=F_4 = C_G(\gamma)$, where $\gamma$ is a graph automorphism of $G$. Choose $y \in H$ an involution which inverts $x$, and note that  $\gamma y$ is also an involution inverting $x$. There are precisely two conjugacy classes of graph automorphisms of $E_6$, with centralizers $F_4$ and $C_4$.  Note that $\dim C_H(\gamma y) = \dim C_H(y) = 24$.   Therefore, $\dim C_G(\gamma y) \leqs 50$ and so $C_G( \gamma y) = C_4$, whence $\gamma y$ is in the conjugacy class of involutions inverting a maximal torus of $G$.   

Next consider $G=E_7$.  There are three semiregular classes of unipotent
elements in $G$, with respective centralizers of dimension $7,9$ and $11$.  Also, there are three classes of involutions in $G$, with dimensions $52, 64$ and $70$.  Let $y \in G$ be an
involution inverting the semiregular unipotent element $x$ (these classes are all real
by \cite[Corollary 5]{LSbook}, for example). As before, it suffices to show that $y$ and $\tau$ are $G$-conjugate. Equivalently, we need to show that $\dim y^G = 70$. Note that $\dim y^G \neq 52$ by \eqref{eq1}.

Suppose $\dim y^G = 64$, so $C_{G}(y)= D_6A_1$.
We may view $G$ as a subgroup of $L=E_8$ (note that it is really the double cover of $G$ that is contained
in $E_8$, but $y$ lifts to an involution in the double cover). We claim that $C_L(y)=E_7A_1$.  To see this, we use an argument provided by Ross Lawther (private communication). Take a representative $y=h_{\alpha_3}(-1)h_{\alpha_5}(-1)h_{\alpha_7}(-1)$ of the
$D_6A_1$-class in $E_7$ and observe that
$$\alpha_2,\; \alpha_3,\; \alpha_4,\; \alpha_5,\; \alpha_6,\; \alpha_7,\; 2\alpha_1 + 2\alpha_2 + 3\alpha_3 + 4\alpha_4 + 3\alpha_5 + 2\alpha_6 + \alpha_7$$
is a basis of the root system of $C_G(y)$. The root system of $C_{L}(y)$ has a basis comprising the above roots, together with
$$\alpha_1 + \alpha_3 + \alpha_4 + \alpha_5 + \alpha_6 + \alpha_7 + \alpha_8.$$
These roots form a simple system of type $E_7A_1$. This establishes the claim and we deduce that $\dim y^L = 112$.

By \cite{LSbook} we have $\dim C_{L}(x) = 16,20$ or $24$. Now
$\dim y^L \geqs \frac{1}{2}(248-\dim C_{L}(x))$ and therefore we may assume $\dim C_{L}(x)=24$ (since $\dim y^L = 112$). Let $\mu:y^L \times y^L \rightarrow L$ be the multiplication map and let $W$ be the image of $\mu$.
Since $\dim x^L = \dim (y^L \times y^L)$, it follows that if $x \in W$ then $x^L$ is an open dense subset of $W$ and the generic fiber of $\mu$ is finite.
In particular, $W$ is contained in the set of unipotent elements of $L$.
Thus, the same is true for $\mu$ restricted to $y^G \times y^G$, which we denote by $\mu_{G}$.
Therefore, the dimension of the image of $\mu_{G}$ is at most $126$ (the dimension of
the unipotent variety of $G$), and hence the generic fiber of $\mu$ has dimension at least $2$. This is a contradiction since the generic fiber is finite.

We conclude that $\dim y^G=70$ is the only possibility, so $y$ and $\tau$ are $G$-conjugate, as required.

Finally, let us assume $G=E_8$. There are two classes of involutions
in $G$, of dimensions $128$ and $112$.   Let $x \in G$ be a semiregular unipotent
element. The semiregular unipotent classes have centralizers of dimension $8$, $10$ and $12$, and there is an additional class in characteristic $3$ with $\dim C_{G}(x)=30$.

Suppose $\dim C_G(x)<30$ and let $y \in G$ be an involution that inverts $x$ (by \cite[Corollary 5]{LSbook}, the class $x^G$ is real). Then $\dim y^G \geqs \frac{1}{2}(\dim G - \dim C_{G}(x)) \geqs 118$ and thus $y$ and $\tau$ are $G$-conjugate.

To complete the proof, we may assume $p=3$ and $\dim C_{G}(x) = 30$.  Fix a maximal torus
$T$ of $G$ and a corresponding set of roots.
We may take
\begin{align*}
x= & \;x_{01121100}(1)x_{00111100}(1)x_{11110000}(1)x_{00001110}(1)
x_{01121000}(1) \\
& \;\times x_{00000111}(1)x_{10111000}(1)x_{01011100}(1)x_{01122100}(1)
\end{align*}
(this follows by calculating the Jordan blocks of $x$ on the adjoint module -- see the class labelled $A_7^{(3)}$ in \cite[Table 9]{lawthercom}).
Let
$$y = h_{\a_2}(-1)h_{\a_4}(-1)h_{\a_7}(-1)h_{\a_8}(-1)n_{\a_2}n_{\a_3}n_{\a_5}n_{\a_8}.$$
By inspecting the $E_8$ structure constants given in the appendix of
\cite{LSmem} we see that $y$ is an involution in $N_G(T)$ that inverts
$x$. Indeed, $y$ reverses the order of the root elements concerned and negates each coefficient.

Let $w = gT$ be the corresponding element of the Weyl group. Note that the roots fixed by $w$ are $01121000, 01122221, 22343221, 23465421$ and their negatives. We find that for each such root $\a$, the root vector $e_\a$ is in fact negated by $\mathrm{Ad}(g)$.
As $w$ is the product of four reflections in mutually orthogonal roots, the trace of
$\mathrm{Ad}(g)$ on the Lie algebra $\mathrm{Lie}(T)$ is $0$; hence its trace on
$\mathrm{Lie}(G) =
\mathrm{Lie}(T) \oplus \bigoplus_{\a} Ke_\a$ is $-8$.

Let $s$ and $t$ be involutions in $T$ such that $C_{G}(s)=A_1E_7$ and $C_{G}(t)=D_8$.
Since $\dim C_{G}(s) = 136$, it follows that the trace of $\mathrm{Ad}(s)$ on $\mathrm{Lie}(G)$
is $136-112=24$.  Therefore $y$ is conjugate to $t$, which is conjugate to $\tau$.
\end{proof}

\begin{remk}\label{r:thom}
Part (ii) of Proposition \ref{main1} gives a conjugacy class $C$ of automorphisms of 
$G$ with the property $G=C^2$. This observation is related to a well known open conjecture of 
J.G. Thompson, which asserts that if $G$ is a finite simple group then $G=C^2$ for some 
conjugacy class $C$. This has been verified if $G$ is an alternating or sporadic group, 
and also if $G$ is a simple group of Lie type over $\F$ with $q>8$. 
We refer the reader to \cite{EG} for further details.
\end{remk}

\begin{prop}\label{cor1}  Assume that $p \ne 2$.
Let $\tau \in {\rm Aut}(G)$ be an involution that inverts
a maximal torus of $G$, and let $u$ be a regular element inverted by $\tau$.
Then $\tau$ acts as inversion on $C_G(u)$.
\end{prop}

\begin{proof}
We argue by induction on $\dim G$. In the usual manner, we first reduce to the case where $G$ is simple and simply connected.  Recall that $r$ denotes the rank of $G$.

Consider the multiplication map $\mu:\tau^G \times \tau^G \rightarrow G$.
By Proposition \ref{main1}, this map is surjective. By Lemma \ref{invol2} we have
$\dim \tau^G = \frac{1}{2}(\dim G + r)$, so Lemma \ref{l:gn1} implies that there is a non-empty open subvariety $W$ of $G$ such that $\dim \mu^{-1}(w)=r$ for all $w \in W$.
Therefore
$$\dim \mu^{-1}(u) = r = \dim C_G(u).$$
If $x'y'=u$ with $x',y' \in \tau^G$, then $x' \in \tau C_G(u)$.
In particular, there exists $v \in C_G(u)$ such that the coset $\tau vC_G(u)^0$ consists of involutions, whence $\tau$ acts as inversion on $vC_G(u)^0$.
If $C_G(u)$ is connected, this gives the result.

If $u$ is not unipotent then $\tau$ normalizes
the subgroup of semisimple elements in $C_G(u)$.  Since
this subgroup properly contains $Z(G)$, we can pass to $C_G(t)$ with
$t \in C_G(u)$ semisimple and non-central (note that $C_G(u) \leqs C_G(t)$).  The result follows by induction.

Finally, suppose $u$ is unipotent. Then
$C_G(u)=Z(G) \times U$, where $U$ consists of unipotent elements, and so we may assume that $Z(G)=1$. If $p=0$, or if $p$ is a good prime for $G$ then $C_G(u)$ is connected
and the result follows. In bad characteristic, $C_G(u)$ is disconnected, but we have $C_G(u) = \langle C_G(u)^0, u \rangle$.  Since $\tau$ inverts
$u$ and a coset of $C_G(u)^0$, we conclude that $\tau$ inverts $C_G(u)$.
\end{proof}

\subsection{Proof of Theorem \ref{inv:main}}

We are ready to prove the main statement of Theorem \ref{inv:main}.

\begin{cor}\label{c:3}
Let $G$ be a simple algebraic group of rank $r$ over an algebraically closed field of characteristic $p \neq 2$. Let $H=C_{G}(\tau)$, where $\tau \in {\rm Aut}(G)$ is an involution that inverts a maximal torus of $G$, and let $\Omega=G/H$ be the corresponding homogeneous space.  Then the following hold:
\begin{itemize}\addtolength{\itemsep}{0.2\baselineskip}
\item[{\rm (i)}] $H$ has a unique regular orbit on $\Omega$, so $b^0(G,H)=b(G,H)=2$.
\item[{\rm (ii)}] The generic $2$-point stabilizer has order $2^r$, i.e. there is a non-empty open subvariety $U \subseteq \Omega \times \Omega$ such that
$|G_{\a} \cap G_{\b}|=2^r$ for all $(\a,\b) \in U$.
\item[{\rm (iii)}] $b^1(G,H)=3$.
\item[{\rm (iv)}] If $G < A \leqs {\rm Aut}(G)$ then $A$ acts on $\Omega$
and  $b^0(A,C_A(\tau))=2$, $b(A,C_A(\tau))=b^1(A,C_A(\tau))=3$.
\end{itemize}
\end{cor}

\begin{proof}
By Proposition \ref{main1}, there exists a conjugate $\tau^g$ of $\tau$ such that $\tau \tau^g = u$ is a regular unipotent element. Then $\tau$ inverts $u$, so Proposition \ref{cor1}  implies that $\tau$ inverts $C_G(u)$, whence $H \cap H^g=1$ by Corollary \ref{c:ru}. In particular, we have $b^0(G,H)=b(G,H)=2$.

Conversely, suppose that $C_G(\tau) \cap C_G(\tau')=1$ for some conjugate $\tau'$ of $\tau$. Then $\tau$ acts fixed-point-freely on $C_G(\tau \tau')$, so
$\tau$ inverts $C_G(\tau \tau')$ and thus $C_G(\tau \tau')$ is abelian and contains
no involutions.  Therefore, $u=\tau \tau'$ is a semiregular unipotent element and thus \cite[Theorem 1]{Lawcc} implies that $u$ is regular.
Since any two involutions in $\langle \tau, C_G(u) \rangle$ are conjugate,
it follows that $C_{G}(\tau)$ acts transitively on the set of regular unipotent elements in $G$ inverted by $\tau$.  As we have noted above, the points of $\Omega$ which belong to a regular $C_G(\tau)$-orbit are in bijection with the regular unipotent
elements inverted by $\tau$. Therefore, $H$ has a unique regular orbit on $\Omega$.

Generically, the product of two conjugates of $\tau$ is a regular semisimple element (because the set of such elements is dense in $G$), whence there exists a non-empty open subvariety $U$ of $\Omega \times \Omega$ such that
$G_{\a} \cap G_{\b}$ coincides with the set of involutions in a maximal torus of $G$, for all $(\a,\b) \in U$. Therefore (i) and (ii) hold, while (iii) follows from Proposition \ref{p:bb}(iv).

Finally, let us consider (iv). Since the class of involutions inverting a maximal torus is invariant
under $A$, we have $A=N_A(H)G$ and so $A$ acts on $\Omega$ (with point stabilizer $C_A(\tau)$).  The only possible
regular orbit would be the (unique) regular orbit of $G$, but clearly $A$
is not regular on this orbit so $b(A,C_A(\tau)) \geqs 3$. By (ii), it follows that a generic pair of points has finite (but nontrivial) $A$-stabilizer, so $b^0(A,C_A(\tau))=2$ and $b^1(A,C_A(\tau))=3$ as claimed.
\end{proof}

\begin{remk} \label{r:nnew}
Suppose $p=2$ and $\tau \in {\rm Aut}(G)$ is an involution inverting a maximal torus of $G$. Since tori have no involutions,
it is trivial to see that generically $C_G(\tau) \cap C_G(\tau^g)=1$.  Therefore, with respect to the action of $G$ on $\tau^G$, we see that
$$b^0(G,C_G(\tau))=b(G,C_G(\tau))=b^1(G,C_G(\tau))=2.$$
\end{remk}

Finally, the following result completes the proof of Theorem \ref{inv:main}.

\begin{prop}
Let $G$ be a connected reductive algebraic group over an algebraically closed field of characteristic $p \geqs 0$.
Let $\tau \in {\rm Aut}(G)$ be an involution that does not
invert a maximal torus of $G$. Set $H = C_G(\tau)$ and $\Omega=G/H$. Then $b^0(G,H) \geqs 3$.
\end{prop}

\begin{proof}
If $p=2$, the result follows by Lemma \ref{l:ba2}, so assume that $p \ne 2$.
We proceed by induction on $\dim G$.  Let $\tau' = \tau^g$ be a conjugate of $\tau$. 

The result is clear if $\dim G=1$, or more generally
if $G$ is solvable, as in the proof of Lemma \ref{l:ba2}.  By induction, we may assume that
the solvable radical of $G$ is trivial and so $G$ is semisimple.  Again by induction, we may
assume that $\tau$ permutes the simple components of $G$ transitively and so $G$ is either simple, or a product of two simple groups.  In the latter case, we see that generically the product
$\tau\tau'$ is regular semisimple and the common centralizer of $\tau$ and $\tau'$ is a diagonal torus. In particular, $H \cap H^g$ generically has positive dimension, whence $b^0(G,H) \geqs 3$.

So we may assume that $G$ is simple. Set $t=\tau \tau'$ and $L = C_G(t)$. If $L^0$ is nonabelian then Lemma \ref{involution1} implies that
$\dim C_L(\tau)>0$ and the result follows since
$C_{L}(\tau) = C_{G}(\tau) \cap C_{G}(\tau')$.

Now assume $L^0$ is abelian. By \cite[Theorem 2]{Lawcc},
we deduce that either $t$ is regular, or $p=3$ and
$G=G_2$. In the latter case, ${\rm Aut}(G)$
has a unique class of involutions so $\tau$ must invert a maximal torus, which is a contradiction.
Therefore, we may assume $t$ is regular. As before, if $\dim C_L(\tau)>0$ then the result follows, so let us assume otherwise.

By Lemma \ref{involution1}, $\tau$ inverts $L^0$ and we claim that $\tau$ inverts a maximal torus of $G$.
By the usual reduction argument, we may assume that $t$ is unipotent. We may also assume that $Z(G)=1$.
It follows
that $\gamma C_G(t)=\tau C_G(t)$, where $\gamma \in {\rm Aut}(G)$ is an involution that inverts a maximal torus of $G$.
Therefore $\gamma$ and $\tau$ are conjugate
(since $\gamma$ inverts $C_G(t)$) and thus $\tau$ inverts a maximal torus.
This final  contradiction completes the proof.
\end{proof}

\subsection{Applications}

We will now use Theorem \ref{inv:main} to settle some special cases of Theorems \ref{t:cmain} and \ref{t:mainep2}. Let $G$ be a simple algebraic group over an algebraically closed field $K$ of characteristic $p \geqs 0$ and suppose $H$ is an \emph{involution-type} subgroup of $G$. This means that $H$ is a maximal subgroup of $G$ with the same structure as a centralizer $C_{\tilde{G}}(\tau)$, where $\tilde{G}$ is a simple algebraic group over an algebraically closed field of characteristic $r \neq 2$, $\tau \in {\rm Aut}(\tilde{G})$ is an involution (where ${\rm Aut}(\tilde{G})$ denotes the group of algebraic automorphisms of $\tilde{G}$), and the root systems of $G$ and $\tilde{G}$ are isomorphic. For example, $A_1E_7$ and $D_8$ are the involution-type subgroups of $E_8$.

Note that certain involution-type subgroups of symplectic and orthogonal groups act reducibly on the natural module; we will deal separately with these subspace actions in Section \ref{ss:red}. The non-subspace involution-type subgroups we are interested in here are listed in  Table \ref{t:invl}.

\begin{table}[h]
$$\begin{array}{ll} \hline
G & \mbox{Type of $H$} \\ \hline
{\rm SL}_{n} & {\rm GL}_{n/2}\wr S_2,\; {\rm Sp}_{n}, \; {\rm SO}_{n} \\
{\rm Sp}_{n} & {\rm Sp}_{n/2}\wr S_2 \, (n \geqs 8),\; {\rm GL}_{n/2} \\
{\rm SO}_{n} & O_{n/2}\wr S_2,\; {\rm GL}_{n/2} \\
E_8 & A_1E_7, \; D_8 \\
E_7 & A_1D_6,\; T_1E_6, \; A_7 \\
E_6 & D_5T_1,\; C_4 \, (p \neq 2),\; A_1A_5, \; F_4 \\
F_4 & B_4, \; C_4\, (p=2),\; A_1C_3\, (p \neq 2) \\
G_2 & A_1\tilde{A}_1 \\ \hline
\end{array}$$
\caption{Non-subspace involution-type subgroups}
\label{t:invl}
\end{table}

Our main result on involution-type subgroups is the following:

\begin{thm}\label{t:ai}
Let $G$ be a simple algebraic group over an algebraically closed field $K$ of characteristic $p \geqs 0$, let $H$ be a non-subspace involution-type subgroup of $G$ and let $\Omega=G/H$ be the corresponding coset variety. Then one of the following holds:
\begin{itemize}\addtolength{\itemsep}{0.2\baselineskip}
\item[{\rm (i)}] $b^0(G,H) = b(G,H) = 2$, $b^1(G,H)=3$, and either $G={\rm SL}_{2}$ and $H$ is of type ${\rm GL}_{1} \wr S_2$, or $p \neq 2$ and
\begin{align*}
(G,H) = & \, ({\rm SL}_{n}, {\rm SO}_{n}), ({\rm Sp}_{n},{\rm GL}_{n/2}), ({\rm SO}_{n}, O_{n/2}\wr S_2),(E_8, D_8), (E_7, A_7.2), \\
& \, (E_6,C_4), (F_4,A_1C_3) \mbox{ or } (G_2, A_1\tilde{A}_1);
\end{align*}
\item[{\rm (ii)}] $b^0(G,H)=b(G,H) = b^1(G,H) = b$ and $(G,H,b)$ is recorded in Table \ref{t:inv1};
\item[{\rm (iii)}] $2 = b^0(G,H) \leqs b(G,H) \leqs b^1(G,H) = 3$, $p=2$, $G={\rm SO}_{n}$ and $H$ is of type $O_{n/2}\wr S_2$, where $n \equiv 0 \imod{4}$ and $n \geqs 8$;
\item[{\rm (iv)}] $2 = b^0(G,H) \leqs b(G,H) \leqs b^1(G,H) \leqs 3$, $p=2$ and $(G,H) = (E_7,A_7.2), (E_6,A_1A_5)$ or $(G_2,A_1\tilde{A}_1)$.
\end{itemize}
\end{thm}

\begin{table}[h]
$$\begin{array}{llll} \hline
G & \mbox{Type of $H$} & {\rm Conditions} & b \\ \hline

{\rm SL}_{n} & {\rm GL}_{n/2} \wr S_2 & n \geqs 4 & 3  \\
 & {\rm Sp}_{n} & n \geqs 6 & 3+\delta_{6,n} \\

{\rm Sp}_{n} & {\rm Sp}_{n/2} \wr S_2 & n \geqs 8 & 3 \\

{\rm SO}_{n} & {\rm GL}_{n/2} & n \geqs 10 & 3  \\

E_8 & A_1E_7 & & 3  \\
& D_8 & p = 2 & 2  \\

E_7 & A_1D_6 & & 3 \\
& T_1E_6 & &  3 \\

E_6 & F_4 & & 4 \\
& D_5T_1 & & 3  \\
& A_1A_5 & p \neq 2 & 3  \\

F_4 & B_4 & & 4   \\
& C_4 & p=2 & 4   \\ \hline
\end{array}$$
\caption{Values of $b$ in Theorem \ref{t:ai}(ii)}
\label{t:inv1}
\end{table}

We prove Theorem \ref{t:ai} in a sequence of lemmas. First we record a couple of
useful preliminary results. Let $V$ be a finite dimensional vector space over $K$. We say that $g \in {\rm GL}(V)$ is a \emph{quadratic element} if its minimal polynomial over $K$ is quadratic
(or equivalently if $g$ is semisimple, then  $g$ has precisely two distinct eigenvalues).

\begin{lem}\label{l:qe}
Let $g,h \in \GL(V)$ be quadratic elements and set $G=\langle g, h \rangle$.
Then every composition factor of the $KG$-module $V$ has dimension at most $2$.
\end{lem}

\begin{proof}
We argue by induction on $n=\dim V$.  The result is clear if $n \leqs 2$. By induction, it suffices to prove that $G$ acts reducibly on $V$.

Let $U \subseteq V$ be a $g$-eigenspace of largest dimension. Since $g$ is quadratic, we have $\dim U \geqs n/2$. If $h$ has an eigenvector in $U$, then
$G$ has a $1$-dimensional invariant subspace, so let us assume otherwise.
Now, if $0  \ne v \in hU \cap U$ then the span of
$v$ and $hv$ is $G$-invariant. Therefore, we may assume $hU$ is a complement
to $U$ in $V$, whence $n$ is even and $\dim U = n/2$.  Let $W$ be an
$h$-eigenspace of largest dimension. The same argument shows that
$\dim W = n/2$, so we may assume that
$V = U \oplus W = W \oplus gW$.

Thus, with respect to an appropriate choice of basis, we have
$$
g = \begin{pmatrix}  aI_{n/2}  &   A  \\  0  & bI_{n/2}\\  \end{pmatrix},\;\; h = \begin{pmatrix}  cI_{n/2} &  0 \\ B  & dI_{n/2} \\  \end{pmatrix},
$$
where $A$ and $B$ are invertible. Conjugating by an appropriate
block diagonal matrix, we may assume that
$A$ is the identity matrix and $B$ is diagonal, whence $G$ clearly has
a $2$-dimensional invariant subspace.
\end{proof}

\begin{cor}\label{c:q}
Suppose $G=\SL_n(K)$,
$A \in \GL_n(K)$ is a quadratic element and $H=N_G(K[A])$.
Then
$$\dim (H \cap H^g) \geqs \left\{\begin{array}{ll}
n/2-1 & \mbox{if $n$ is even} \\
(n-1)/2 & \mbox{if $n$ is odd}
\end{array}\right.$$
for all $g \in G$.
\end{cor}

\begin{proof}
Let $V$ be the natural $KG$-module. There exists a non-empty open subset $U$ of $G$ such that $\langle A, A^g \rangle$ contains a regular semisimple element of $G$ for all $g \in U$.
Therefore, if $g \in U$ then Lemma \ref{l:qe} implies that
$V$ is a direct sum of $1$- and $2$-dimensional non-isomorphic irreducible
$K\langle A,A^g \rangle$-modules, whence $C_{G}(A) \cap C_{G}(A^g)$ contains a torus of dimension $n/2-1$ if $n$ is even, and one of dimension $(n- 1)/2$ if $n$ is odd. The resulting lower bound on $\dim(H \cap H^g)$ holds for all $g \in \overline{U}$, hence all $g \in G$ since $U$ is dense in $G$.
\end{proof}

We are now ready to give the proof of Theorem \ref{t:ai}. For the remainder of this section, let $\mathcal{P}$ be the set of elements of prime order in $H$ (including all nontrivial unipotent elements if $p=0$).

\begin{lem}\label{p:aic}
Theorem \ref{t:ai} holds if $G$ is a classical group.
\end{lem}

\begin{proof}
First suppose $G={\rm SL}_{n}$. If $H$ is of type ${\rm GL}_{n/2}\wr S_2$ then $\dim x^H \leqs \frac{1}{2}\dim x^G$ for all $x \in \mathcal{P}$
(see the proof of \cite[Proposition 2.1]{Bur5}), whence Corollary \ref{c:con} yields $b^1(G,H)\leqs 3$. Now $H=N_G(K[A])$ for a suitable quadratic element $A \in \GL_n$, so if $n \geqs 3$ then Corollary \ref{c:q} implies that $b^0(G,H)\geqs 3$ and thus
\begin{equation}\label{e:all3}
b^0(G,H)=b(G,H)=b^1(G,H)=3.
\end{equation}
Now suppose $n=2$, so $H=N_G(T)$ is the normalizer of a maximal torus $T$ of $G$. If $p \neq 2$ then $H=C_{G}(\tau)$, where $\tau$ is an involution inverting $T$,
so in this case Theorem \ref{inv:main} implies that
\begin{equation}\label{e:32}
b^0(G,H)=b(G,H)=2,\;\; b^1(G,H)=3.
\end{equation}
Now assume $p=2$ and let $X_1, X_2$ be distinct tori in $G$. If $X_1$ and $X_2$ are not contained in a common
Borel subgroup then it is straightforward to see that the intersection $N_G(X_1) \cap N_G(X_2)$ has order $2$. On the other hand, if $X_1, X_2$ are contained in the same Borel subgroup of $G$ then $N_G(X_1) \cap N_G(X_2)$ is trivial, whence we have the same answer as for $p \ne 2$.

If $H={\rm Sp}_{n}$ then \cite[Theorem 1.1]{GG} yields
$$b^0(G,H)=b(G,H)=b^1(G,H)=3+\delta_{6,n}+2\delta_{4,n}$$
(note that the case $n=4$ is equivalent to a subspace action -- see Table \ref{t:sub}).
Finally, if $H$ is of type ${\rm SO}_{n}$ then $p \neq 2$ (since $H$ is a maximal subgroup of $G$) and $H=C_{G}(\tau)$ for a suitable involutory graph automorphism $\tau$. By Theorem \ref{inv:main}, since $p \neq 2$ and $\tau$ inverts a maximal torus of $G$, we conclude that \eqref{e:32} holds.

Next assume $G={\rm Sp}_{n}$. If $H$ is of type ${\rm Sp}_{n/2} \wr S_2$ then
$\dim H >\frac{1}{2}\dim G$
and thus $b^0(G,H) \geqs 3$ by Proposition \ref{p:bb}(iii). Here \cite[Proposition 2.1]{Bur5} states that 
$$\dim (x^G \cap H) \leqs \left(\frac{1}{2}+\frac{1}{n}\right)\dim x^G$$ 
for all $x \in \mathcal{P}$, so
\begin{equation}\label{e:bb}
\dim x^H < \frac{2}{3}\dim x^G
\end{equation}
if $n \geqs 8$. In particular, Corollary \ref{c:con} implies that \eqref{e:all3} holds when $n \geqs 8$. (As noted in Table \ref{t:sub}, if $n=4$ then the action of $G$ is equivalent to a subspace action -- see Remark \ref{r:sp2s2}.) Now suppose $H$ is of type ${\rm GL}_{n/2}$. If $p=2$ then $H$ is contained in a subgroup of type ${\rm SO}_{n}$, so we may assume $p \neq 2$. Here $H=C_{G}(\tau)$, where $\tau \in G$ is an involution inverting a maximal torus of $G$, so \eqref{e:32} holds by Theorem \ref{inv:main}.

Finally, let us turn to the case $G={\rm SO}_{n}$. First suppose $H$ is of type ${\rm GL}_{n/2}$. Here $\dim H > \frac{1}{2}\dim G$ and thus $b^0(G,H) \geqs 3$ by Proposition \ref{p:bb}(iii). According to the proof of \cite[Lemma 4.2]{Bur2}, if $n \geqs 10$ then
$$\dim x^H \leqs \left(\frac{1}{2}+\frac{1}{n-2}\right)\dim x^G$$
for all $x \in \mathcal{P}$, so Corollary \ref{c:con} implies that \eqref{e:all3} holds. (Note that if
$n=8$ then the action of $G$ is equivalent to the action of ${\rm SO}_{8}$ on non-degenerate $2$-spaces of the natural module -- see Table \ref{t:sub}.)

Now suppose $H$ is of type $O_{n/2} \wr S_2$, so $H$ is the stabilizer of a pair of complementary non-degenerate spaces. If $p \neq 2$ then $H=C_{G}(\tau)$ for an involution $\tau \in {\rm Aut}(G)$ which inverts a maximal torus of $G$, whence \eqref{e:32} holds by Theorem \ref{inv:main}. Now assume $p=2$, so $n/2$ is even. By \cite[Proposition 2.1]{Bur5} we have $\dim x^H \leqs \frac{1}{2}\dim x^G$ for all $x \in \mathcal{P}$ (with equality if and only if $x$ is an involution of type $c_s$ (with $2 \leqs s \leqs n/2$ and $s$ even) or $a_{n/2}$, in the notation of Aschbacher and Seitz \cite{AS}), so Corollary \ref{t:con} yields $b^1(G,H) \leqs 3$. In fact, by applying Lemma \ref{l:so11} (see Section \ref{ss:o2}) we deduce that
$b^0(G,H)=2$ and $b^1(G,H)=3$ (the fact that $b^0(G,H)=2$ also follows from Proposition \ref{p:newb}). However, we have been unable to determine the exact value of $b(G,H)$ in this case.
\end{proof}

In order to complete the proof of Theorem \ref{t:ai}, we may assume that $G$ is an exceptional group, and we will consider each possibility for $G$ in turn. Let us say a few words on the notation and terminology we will use in the remainder of this section. Given a semisimple subgroup $X \leqs G$ we write $\Phi(X)$ (respectively, $\Phi^+(X)$) for the set of roots (respectively, positive roots) of $X$ with respect to a fixed maximal torus. If $W$ is a $KG$-module then $W \downarrow X$ denotes the restriction of $W$ to $X$. For each simple factor $Y$ of $X$ we fix a set of fundamental dominant weights $\{\l_1, \l_2, \ldots\}$ (numbered in the usual way, following \cite{Bou}), and we write $L(\l)$ for the irreducible $KY$-module with highest weight $\l$. If $W$ is a $KX$-module then $W^*$ denotes its dual. The Lie algebra of $X$ is denoted by $\mathrm{Lie}(X)$, and we write $T_{i}$ for an $i$-dimensional torus. In addition, $J_i$ denotes a standard unipotent Jordan block of size $i$, and we adopt the notation of \cite{lawthercom} for labelling the unipotent classes in $G$.

Let $\Phi$ be a root system and let $\Psi$ be a subsystem of $\Phi$. Following \cite[Section 5]{LLS}, we say that
$\Psi$ is \emph{$A_2$-dense in $\Phi$} if every subsystem of $\Phi$ of type $A_2$ meets $\Psi$. Note that if $\Phi_1$ is a subsystem of $\Phi$, and $\Psi$ is $A_2$-dense in $\Phi$, then $\Psi \cap \Phi_1$ is $A_2$-dense in $\Phi_1$. Such subsystems are called \emph{anti-open} in \cite{Law2}, and the complete list of all proper anti-open subsystems of irreducible root systems is given in \cite{Law2} (also see \cite[Lemma 5.1]{LLS}).

\begin{lem}\label{l:e8ai}
Theorem \ref{t:ai} holds if $G=E_8$.
\end{lem}

\begin{proof}
Here $H=N_{G}(X)$ with $X =A_1E_7$ or $D_8$ (see Table \ref{t:inv1}). Suppose $X=A_1E_7$, so $\dim H = 136$ and $b^0(G,H) \geqs 3$ by Proposition \ref{p:bb}(iii).
We claim that \eqref{e:bb} holds for all $x \in \mathcal{P}$. If $x$ is unipotent then we can calculate the precise dimensions of $x^H$ and $x^G$ from the information on class fusions recorded in \cite[Table 23]{Lawunip} (see \cite[Chapter 22]{LSbook} for a convenient list of unipotent class dimensions in exceptional algebraic groups), and the claim quickly follows. Now assume $x$ is semisimple. Since $\dim x^H \leqs 128$, we may assume that $\dim x^G \leqs 192$, in which case the desired result follows from \cite[Theorem 2]{LLS}. For example, if $C_G(x)$ does not have an $E_7$ or $D_8$ factor then \cite[Table 7.4]{LLS} indicates that $\dim x^G - \dim x^H \geqs 70$ and the result follows. This justifies the claim and we conclude that \eqref{e:all3} holds.

Now assume $X=D_8$. If $p\neq 2$ then $H=C_{G}(\tau)$ and $\tau \in G$ is an involution which inverts a maximal torus, so
\eqref{e:32} holds by Theorem \ref{inv:main}.
Now suppose $p=2$. If $x \in H$ is an involution then by inspecting \cite[Table 22]{Lawunip} we quickly deduce that
\begin{equation}\label{e:bb2}
\dim x^H < \frac{1}{2}\dim x^G.
\end{equation}

For the remainder, let us assume $x \in G$ is a semisimple element of prime order $r$ with $D=C_G(x)$. If $r=3$ then $D^0 = A_8$, $A_2E_6$, $E_7T_1$ or $D_7T_1$ (see \cite[Proposition 1.2]{LLS}). In the latter case we have $\dim x^G = 156$ and \cite[Theorem 2]{LLS} states that $\dim x^G - \dim x^H \geqs 80$, so $\dim x^H \leqs 76<\frac{1}{2}\dim x^G$. Next suppose $D^0=E_7T_1$, so $\dim x^G = 114$ and  \cite[Theorem 2]{LLS} yields $\dim x^H \leqs 58$. In fact, we claim that $\dim x^H \leqs 56<\frac{1}{2}\dim x^G$. First observe that $D^0<L=E_7A_1$. By  \cite[Lemma 5.1]{LLS}, the root systems $\Phi(L)$ and $\Phi(H)$ are $A_2$-dense in $\Phi(G)$, so $\Phi(L \cap H)$ is $A_2$-dense in both $\Phi(L)$ and $\Phi(H)$. A further application of \cite[Lemma 5.1]{LLS} implies that $L\cap H = A_7T_1$ or $D_6A_1^2$. Therefore $D \cap H = A_7T_1$ or $D_6A_1T_1$, whence $\dim x^H \leqs 56$ as claimed. The other two cases are similar. For example, suppose $D^0=A_8$ so $\dim x^G = 168$. Since $\Phi(D\cap H)$ is $A_2$-dense in $\Phi(D)$, \cite[Lemma 5.1]{LLS} indicates that $\Phi(D \cap H)= A_1A_6$, $A_2A_5$ or $A_3A_4$, so $\dim x^H \leqs 80< \frac{1}{2}\dim x^G$. Similarly, if $D^0=A_2E_6$ then $\dim x^G = 162$ and once again the $A_2$-density of
$\Phi(D\cap H)$  in $\Phi(D)$ implies that $\dim x^H \leqs 80 < \frac{1}{2}\dim x^G$.

Next suppose $r \geqs 5$ and $D^0$ is semisimple, so $D^0$ does not contain a
positive-dimensional central torus. Then it is easy to see that $r=5$ and $D^0=A_4A_4$ is the only possibility, so $\dim x^G=200$. Now $\Phi(D\cap H)$ is $A_2$-dense in $\Phi(D)$, and by applying \cite[Lemma 5.1]{LLS} we deduce that $|\Phi^{+}(D \cap H)| \geqs |\Phi^{+}(A_2^2A_1^2)|=8$, whence
$$\dim x^G - \dim x^H  = 2\left(|\Phi^+(G)| - |\Phi^+(H)| - |\Phi^+(D)| + |\Phi^{+}(D \cap H)|\right) \geqs 104$$
and thus $\dim x^H \leqs 96$. We conclude that \eqref{e:bb2} holds for all $x \in H$ of order $2$ or $3$, and also for any $x \in H$ of prime order $r \geqs 5$ such that $C_G(x)^0$ is semisimple. Therefore Corollary \ref{c:sembd} implies that $b^1(G,H)=2$.
\end{proof}

\begin{lem}\label{l:e7ai}
Theorem \ref{t:ai} holds if $G=E_7$.
\end{lem}

\begin{proof}
Here $H=N_{G}(X)$ with $X = A_1D_6,T_1E_6$ or $A_7$. First assume  $X=A_1D_6$, so $\dim H = 69$ and thus $b^0(G,H) \geqs 3$. We claim that \eqref{e:bb} holds for all $x \in \mathcal{P}$. If $x$ is unipotent then the desired bound follows from the fusion information in \cite[Table 19]{Lawunip}, so let us assume $x$ is semisimple with centralizer  $D=C_G(x)$. If $D$ contains an $E_6$, $D_6$ or $A_7$ factor then \eqref{e:bb} follows from \cite[Theorem 2]{LLS}. For example, if $D^0=A_7$ then \cite[Table 7.4]{LLS} indicates that $\dim x^G - \dim x^H \geqs 32$, whence $\dim x^H \leqs 38<\frac{2}{3}\dim x^G=140/3$. For the remaining semisimple elements, \cite[Theorem 2]{LLS} yields $\dim x^G - \dim x^H \geqs 40$, so we may assume $\dim x^G \geqs 120$. However, $\dim x^H \leqs 62$ for all $x \in H$ and thus \eqref{e:bb} holds in all cases. This justifies the claim and we conclude that \eqref{e:all3} holds.

Next suppose $X=T_1E_6$. As in the previous case, we have $b^0(G,H) \geqs 3$ since $\dim H>\frac{1}{2}\dim G$, and again we claim that \eqref{e:bb} holds for all $x \in \mathcal{P}$, giving \eqref{e:all3}. If $x$ is unipotent then the conjugacy classes $x^H$ and $x^G$ have the same Bala-Carter label and we quickly deduce that \eqref{e:bb} holds. The argument for semisimple elements is entirely similar to the previous case, using \cite[Theorem 2]{LLS}.

Finally, let us assume $X=A_7$, so $H=A_7.2$. If $p \neq 2$ then $H=C_{G}(\tau)$ for a suitable involution $\tau$ that inverts a maximal torus of $G$. Therefore \eqref{e:32} holds by Theorem \ref{inv:main}. Now assume $p=2$. As in the statement of Corollary \ref{c:sembd}, let $\mathcal{P}'$ be the set of $x \in H$ of prime order $r$, where either $r \leqs 3$ or $C_{G}(x)^0$ is semisimple. We claim that $\dim (x^G \cap H) \leqs \frac{1}{2}\dim x^G$ for all $x \in \mathcal{P}'$, with equality if and only if $r=2$ and $x$
belongs to one of the classes labelled $(3A_1)', (3A_1)''$ or $4A_1$. In particular, a combination of Corollary \ref{c:sembd} and Proposition \ref{p:newb} implies that
\begin{equation}\label{e:nn}
2 = b^0(G,H) \leqs b(G,H) \leqs b^1(G,H) \leqs 3,
\end{equation}
but we are unable to determine the exact values of $b(G,H)$ and $b^1(G,H)$ in this case.

To justify the claim, first assume $x \in \mathcal{P}'$ has odd order $r$. Let $D=C_G(x)$. Since $\dim Z(D^0)>0$ if $r>3$, we may assume $r=3$ and thus $D^0=E_6T_1$, $D_6T_1$, $A_6T_1$, $A_1D_5T_1$ or $A_2A_5$ (see \cite[Proposition 1.2]{LLS}). If $D^0 \in \{E_6T_1,A_6T_1,A_1D_5T_1\}$ then the bound in \eqref{e:bb2} follows from \cite[Theorem 2]{LLS}. For example, if $D^0=A_1D_5T_1$ then $\dim x^G = 84$ and \cite[Theorem 2]{LLS} states that $\dim x^G - \dim x^H \geqs 44$, giving the required bound. Next suppose $D^0=D_6T_1$. Here $\dim x^G = 66$ and
\cite[Theorem 2]{LLS} yields $\dim x^H \leqs 34$. In order to improve this bound, first note that $D^0<L=D_6A_1$, and the root systems $\Phi(L)$ and $\Phi(H)$ are $A_2$-dense in $\Phi(G)$ (see \cite[Lemma 5.1]{LLS}), so $\Phi(L \cap H)$ is $A_2$-dense in both $\Phi(L)$ and $\Phi(H)$. Using \cite[Lemma 5.1]{LLS} to determine the possibilities for $L \cap H$, we deduce that $|\Phi^{+}(D\cap H)|\geqs 12$ and thus $\dim x^H \leqs 32$ as required. Finally, suppose that $D=A_2A_5$, so $\dim x^G = 90$. Using the $A_2$-density of $\Phi(D\cap H)$ in $\Phi(D)$, we deduce that $|\Phi^{+}(D\cap H)|\geqs 6$ and thus $\dim x^H \leqs 44$. This justifies the claim for semisimple elements.

Finally let us assume $r=2$. The fusion of the $H$-classes of involutions in the connected component $H^0=A_7$ is recorded in \cite[Table 20]{Lawunip}, and we quickly deduce that \eqref{e:bb2} holds for all involutions $x \in H^0$, unless
$x$ has Jordan form $[J_2^4]$ on the natural module for $A_7$ (where $J_2$ denotes a standard unipotent Jordan block of size $2$). Here $x$ is in the $G$-class labelled $(3A_1)'$, so $\dim x^H = \frac{1}{2}\dim x^G = 32$. Finally, suppose $x \in H \setminus H^0$ is an involution. Let $V_{56}$ be the $56$-dimensional irreducible $KG$-module. Then \cite[Proposition 2.3]{LSM} gives
$$V_{56} \downarrow A_7 = L(\lambda_2) \oplus L(\lambda_6) = L(\lambda_2) \oplus L(\lambda_2)^*$$
and $x$ interchanges the $A_7$-modules $L(\lambda_2)$ and $L(\lambda_2)^*$, so $x$ has Jordan form $[J_2^{28}]$ on $V_{56}$. In particular, \cite[Table 7]{lawthercom} indicates that $x$ is in one of the $G$-classes labelled $(3A_1)''$ or $4A_1$. In fact, if $C_{H^0}(x) = C_{C_4}(t)$ (where $t \in C_4$ is a long root element) then the proof of \cite[Lemma 4.1]{LLS} reveals that $x$ is in $4A_1$, so $\dim x^H = \frac{1}{2}\dim x^G = 35$. Now, if $C_{H^0}(x) = C_4$ then we can calculate the Jordan form of $x$ on the Lie algebra $\mathrm{Lie}(G)$ (using the fact that $\mathrm{Lie}(G) \downarrow A_7 = \mathrm{Lie}(A_7) \oplus L(\lambda_4)$, as noted in \cite[Proposition 2.1]{LSM}). We find that $x$ has Jordan form $[J_2^{53},J_{1}^{27}]$, and by inspecting \cite[Table 8]{lawthercom} we conclude that $x$ is in the class
$(3A_1)''$. Therefore $\dim x^H = \frac{1}{2}\dim x^G = 27$. This justifies the claim.
\end{proof}

\begin{lem}\label{l:e6ai}
Theorem \ref{t:ai} holds if $G=E_6$.
\end{lem}

\begin{proof}
We have $H=N_{G}(X)$ and $X \in \{D_5T_1, C_4\,(p\neq 2), A_1A_5,F_4\}$. First assume  $X=D_5T_1$. Since $\dim H = 46>\frac{1}{2}\dim G$ we deduce that $b^0(G,H) \geqs 3$ by Proposition \ref{p:bb}(iii). We claim that \eqref{e:bb} holds for all $x \in \mathcal{P}$. If $x$ is unipotent then the Bala-Carter labels for the classes $x^G$ and $x^H$ are the same, and we quickly deduce that \eqref{e:bb} holds. Now assume $x \in H$ is semisimple and set $D=C_{G}(x)$. If $D$ has a $D_5$ or $A_5$ factor then \eqref{e:bb} follows from \cite[Theorem 2]{LLS}; in all other cases, the same result gives $\dim x^G-\dim x^H \geqs 20$, so we may assume $\dim x^G \geqs 60$. In fact, since $\dim x^H \leqs 40$ for all $x \in H$, we reduce to the case $\dim x^G=60$, so $D^0=A_2^2T_2$ or $A_3T_3$. Here the proof of \cite[Lemma 4.17]{BLS} yields $\dim x^G - \dim x^H \geqs 24$. This justifies the claim, and we conclude that $b^0(G,H)=b(G,H)=b^1(G,H)=3$.

Next consider the case $X=C_4$. Here $p \neq 2$ and $H=C_{G}(\tau)$, where $\tau \in {\rm Aut}(G)$ is an involutory graph automorphism that inverts a maximal torus of $G$. Therefore Theorem \ref{inv:main} implies that \eqref{e:32} holds.

Next suppose $X=F_4$. Here $\dim H = 52 =\frac{2}{3}\dim G$, so $b^0(G,H) \geqs 3$. We claim that $b^0(G,H)=b(G,H)=b^1(G,H)=4$. To see this, let $\{\omega_1, \ldots, \omega_6\}$ be a set of fundamental dominant weights for $G$ and let $V=L(\omega_1)$ be the irreducible $27$-dimensional $KG$-module with highest weight $\omega_1$.
Then $H$ is the $G$-stabilizer of a generic $1$-dimensional subspace of $V$, so we can identify $\Omega$ with a non-empty open subvariety in the projective space $\mathbb{P}(V)$. We may also identify $V$ with the coset variety $E_7/P_7$, where $P_7$ is a maximal parabolic subgroup of $E_7$ with Levi subgroup $L=E_6T_1$. Now $F_4$ is the $L$-stabilizer of a generic point in $E_7/P_7$, so $L/F_4$ is open in $V$ (with $T_1$ acting as scalars). In particular, the generic $3$-point (respectively $4$-point) stabilizer for the original action of $G$ on $\Omega$ is the same as the generic $5$-point (respectively $6$-point) stabilizer in the action of $E_7$ on $E_7/P_7$. In Proposition \ref{p:ep2} we will show that 
$$b^0(E_7,P_7) = b(E_7,P_7)=b^1(E_7,P_7)=6,$$ 
whence $b^0(G,H)=b(G,H)=b^1(G,H)=4$ as claimed.

Finally, let us assume $X=A_1A_5$. If $p \neq 2$ then $H=C_{G}(z)$ for a suitable involution $z \in G$, and Theorem \ref{inv:main} implies that $b^0(G,H) \geqs 3$ since $z$ does not invert a maximal torus of $G$.
We claim that \eqref{e:bb} holds for all $x \in \mathcal{P}$, whence $b^0(G,H)=b(G,H)=b^1(G,H)=3$ via Corollary \ref{c:con}. If $x \in H$ is unipotent then the desired bound quickly follows from the information in \cite[Table 17]{Lawunip}, so assume $x$ is semisimple and let $D=C_G(x)$. If $D$ has a $D_5$ or $A_5$ factor then \cite[Theorem 2]{LLS} is sufficient. In all other cases we have $\dim x^G-\dim x^H \geqs 24$ by \cite[Theorem 2]{LLS}, so we may assume $\dim x^G = 72$. However, $\dim x^H \leqs 32$ for all $x \in H$, and the claim follows.

Now suppose $p=2$. As in Corollary \ref{c:sembd}, let $\mathcal{P}'$ be the set of $x \in H$ of prime order $r$, where either $r \leqs 3$ or $C_{G}(x)^0$ is semisimple.
We claim that $\dim (x^G \cap H) \leqs \frac{1}{2}\dim x^G$ for all $x \in \mathcal{P}'$, with equality if and only if $r=2$ and $x$ belongs to one of the $G$-classes labelled $2A_1$ or $3A_1$. In particular, a combination of Corollary \ref{c:sembd} and Proposition \ref{p:newb} implies that \eqref{e:nn} holds, but we have not determined the exact values of $b(G,H)$ and $b^1(G,H)$ in this case.

If $x \in H$ is an involution then the claim quickly follows from the fusion information presented in \cite[Table 17]{Lawunip}. For example, if $x = ([J_2],[J_2^3]) \in H$ then $x$ is in the $G$-class labelled $3A_1$ and thus $\dim x^H = 2+18 = 20$ and $\dim x^G = 40$. Now assume $x \in \mathcal{P}'$ has odd prime order $r$. Set $D=C_G(x)$. If $r>3$ then $\dim Z(D^0)>0$, so we may assume $r=3$, in which case $D^0=A_5T_1$, $D_4T_2$ or $A_2^3$ (see \cite[Proposition 1.2]{LLS}). According to \cite[Lemma 5.1]{LLS}, the root system $\Phi(H)$ is $A_2$-dense in $\Phi(G)$, whence $\Phi(D\cap H)$ is $A_2$-dense in $\Phi(D)$. First assume $D=A_5T_1$, so $\dim x^G = 42$. Here the $A_2$-density of $\Phi(D\cap H)$ in $\Phi(D)$ implies that $\Phi(D \cap H) = A_3A_1$ or $A_2^2$. In particular, $|\Phi^{+}(D\cap H)| \geqs 6$ so $\dim x^H \leqs 20$ as required. Next suppose $D^0 = D_4T_2$. Here $\dim x^G = 48$ and the usual argument implies that $\Phi(D \cap H) = A_3$, whence $|\Phi^{+}(D\cap H)| \geqs 6$ and thus $\dim x^H \leqs 20$. Finally, if $D^0=A_2^3$ then $\dim x^G = 54$ and the result follows since
$\dim y^H \leqs 26$ for all $y \in H$ of order $3$. This justifies the claim.
\end{proof}

\begin{lem}\label{l:f4ai}
Theorem \ref{t:ai} holds if $G=F_4$.
\end{lem}

\begin{proof}
Here $H=N_{G}(X)$ with $X \in \{B_4,C_4 \,(p=2), A_1C_3\,(p \neq 2)\}$. First assume $X=B_4$, so $\dim H = 36>\frac{2}{3}\dim G$ and thus $b^0(G,H) \geqs 4$. A combination of \cite[Table 13]{Lawunip} and  \cite[Theorem 2]{LLS} implies that $\dim x^H \leqs \frac{3}{4}\dim x^G$ for all $x \in \mathcal{P}$, with equality if and only if $x$ is a long root element. Therefore $b^1(G,H) \leqs 4$ by Corollary \ref{c:lr} (note that every long root subgroup of $H=H^0$ is a long root subgroup of $G$), hence $b^0(G,H)=b(G,H)=b^1(G,H)=4$.
The case $X=C_4$ (with $p=2$) now follows immediately since subgroups of type $C_4$ and $B_4$ are conjugate in ${\rm Aut}(G)$; they are interchanged by an involutory graph automorphism.

Finally, suppose $X=A_1C_3$ and $p \neq 2$ (note that $A_1C_3 < F_4$ is non-maximal when $p=2$; see \cite[Table 10.3]{LSmem}). Here $H=C_{G}(\tau)$, where $\tau \in G$ is an involution which inverts a maximal torus of $G$. Therefore $b^0(G,H)=b(G,H)=2$ and $b^1(G,H)=3$ by Theorem \ref{inv:main}.
\end{proof}

\begin{lem}\label{l:g2ai}
Theorem \ref{t:ai} holds if $G=G_2$.
\end{lem}

\begin{proof}
Here $H=A_1\tilde{A}_1$. If $p \neq 2$ then $H=C_{G}(\tau)$, where $\tau \in G$ is an involution that inverts a maximal torus of $G$, so \eqref{e:32} holds by Theorem \ref{inv:main}.

Now assume $p=2$. We claim that $\dim (x^G \cap H) \leqs \frac{1}{2}\dim x^G$ for all $x \in \mathcal{P}$, with equality if and only if $x$ is an involution in the $G$-class labelled
$\tilde{A}_{1}$. To justify the claim, first assume $x \in H$ is semisimple. If $C_{H}(x)^0=A_2$ then $\dim x^G=6$ and \cite[Theorem 2]{LLS} indicates that $\dim x^H =2$; in every other case we have $\dim x^G \geqs 10$ and the desired bound follows since $\dim x^H \leqs 4$ for all $x \in H$. Finally, the result for involutions is easily deduced from \cite[Table 10]{Lawunip}. Therefore Proposition \ref{p:newb} implies that \eqref{e:nn} holds when $p=2$, but we are unable to determine the exact values of $b(G,H)$ and $b^1(G,H)$.
\end{proof}

\section{Classical groups}\label{s:class}

Let $G$ be a simple classical algebraic group over an algebraically closed field $K$ of characteristic $p \geqs 0$ with natural module $V$. In this section we complete the proofs of Theorems \ref{t:csub} and \ref{t:cmain}.

The main theorem on the subgroup structure of $G$ is a result of Liebeck and Seitz \cite{LieS}, which provides a natural algebraic group analogue of Aschbacher's celebrated structure theorem \cite{Asch} for finite classical groups. Following \cite[Section 1]{LieS}, we introduce six natural, or \emph{geometric}, collections of closed subgroups of $G$, labelled $\C_i$ for $1 \leqs i \leqs 6$, and we set $\C=\bigcup_{i}\C_i$. A rough outline of the subgroups in each $\C_i$ collection is given in Table \ref{t:subs}. 

\begin{table}[h]
$$\begin{array}{cl} \hline
 & \mbox{Rough description} \\ \hline
\C_1 & \mbox{Stabilizers of subspaces of $V$} \\
\C_2 & \mbox{Stabilizers of orthogonal decompositions $V=\bigoplus_{i}V_i$} \\
\C_3 & \mbox{Stabilizers of totally singular decompositions $V=V_1 \oplus V_2$} \\
\C_4 & \mbox{Stabilizers of tensor product decompositions $V=\bigotimes_i V_i$} \\
\C_5 & \mbox{Normalizers of symplectic-type $r$-groups, $r \neq p$ prime}  \\
\C_6 & \mbox{Classical subgroups} \\ \hline
\end{array}$$
\caption{The $\C_i$ collections}
\label{t:subs}
\end{table}

The main theorem of \cite{LieS} provides the following description of the maximal closed subgroups of $G$.

\begin{thm}\label{t:ls}
Let $H$ be a closed subgroup of $G$. Then one of the following holds:
\begin{itemize}\addtolength{\itemsep}{0.2\baselineskip}
\item[{\rm (i)}] $H$ is contained in a member of $\C$;
\item[{\rm (ii)}] modulo scalars, $H$ is almost simple and $E(H)$ (the unique quasisimple normal subgroup of $H$) is irreducible on $V$. Further, if $G={\rm SL}(V)$ then $E(H)$ does not fix a non-degenerate form on $V$. In addition, if $H$ is infinite then $E(H)=H^0$ is tensor-indecomposable on $V$.
\end{itemize}
\end{thm}

\begin{proof}
This is \cite[Theorem 1]{LieS}.
\end{proof}

We will write $\mathcal{S}$ to denote the collection of maximal closed subgroups of $G$ that arise in part (ii) of Theorem \ref{t:ls}. Note that the subgroups in $\C_5$ are finite so they can be discarded. Also notice that the members of $\C_3$ and $\C_6$ are involution-type subgroups, which  we considered in the previous section.

In studying bases for a classical group $G$, it is natural to make a distinction between the primitive actions of $G$ in which the point stabilizer $H$ acts reducibly on $V$ (the so-called \emph{subspace actions} of $G$), and those in which $H$ is irreducible. Indeed, for subspace actions it is easy to see that the various base measures can be arbitrarily large, whereas Theorem \ref{t:bur} states that $b^1(G,H) \leqs 4$ when $H$ is irreducible. We begin by considering subspace actions.

\subsection{Subspace actions}\label{ss:red}

Let $G$ be a classical algebraic group in a primitive \emph{subspace action} on a $G$-variety
$\Omega$. As defined in the Introduction, this means that either
\begin{itemize}\addtolength{\itemsep}{0.2\baselineskip}
\item[(i)] $\Omega$ is an orbit of subspaces of the natural $G$-module $V$; or
\item[(ii)] The action of $G$ on $\Omega$ is equivalent to the action of an isomorphic classical group $L$ on an orbit of subspaces of the natural $L$-module (see Table \ref{t:sub}).
\end{itemize}
The purpose of this section is to prove Theorem \ref{t:csub}, and we begin our analysis by dealing with the linear groups.

\subsubsection{Linear groups}\label{sss:lin}

\begin{prop}\label{p:s1}
Let $G={\rm SL}_{n}$ and let $\Omega$ be the set of $d$-dimensional subspaces of $V$,  with $d \leqs n/2$.
Set $k = \lceil n/d \rceil$.
\begin{itemize}\addtolength{\itemsep}{0.2\baselineskip}
\item[{\rm (i)}] If $d$ divides $n$ then $b^0(G) = b(G) = b^1(G)=k+\epsilon$, where
$$\epsilon = \left\{\begin{array}{ll}
3 & \mbox{if $1<d=n/2$} \\
2 & \mbox{if $1<d<n/2$} \\
1 & \mbox{if $d=1$}.
\end{array}\right.$$
\item[{\rm (ii)}] If $d$ does not divide $n$ then
$$k +1 \leqs b^0(G) =  b(G) = b^1(G) \leqs k+2+\delta_{3,k}.$$
\end{itemize}
\end{prop}

\begin{proof}
First note that the stabilizer in ${\rm GL}(V)$  of any given collection of subspaces of $V$ coincides with the unit group
of a suitable $K$-algebra and is therefore connected.  Thus, the same is true in $G={\rm SL}(V)$.  It follows that
$b^0(G)=b(G)=b^1(G)$ in all cases.

Consider (i). The case $d=1$ is trivial. Next suppose $d=n/2$ and $n \geqs 4$. Choose four \emph{generic} $d$-dimensional subspaces of $V$, say $V_1, V_2, V_3$ and $V_4$. By generic we mean that $V_i \cap V_j = 0$ for all $i \neq j$ (note that this is an open condition), so $V = V_1 \oplus V_2$ in particular. Let $L$ be the stabilizer in $G$ of $V_1$ and $V_2$, so $L$ is of type ${\rm GL}(V_1) \times {\rm GL}(V_2)$. Conjugating by $L$, we may assume that
$V_3 = \{(v,f(v)) \mid v \in V_1\}$  and $V_4 =\{(v,g(v)) \mid v \in V_1\}$ for generic isomorphisms $f,g:V_1 \to V_2$. By fixing a suitable basis we have $f,g \in {\rm M}_{n/2}(K)$ (the algebra of $n/2 \times n/2$ matrices over $K$) and we are free to assume that $f$ is the identity matrix. Since $V_4$ has been chosen generically, it follows that $g$ is a regular semisimple matrix.

Suppose $x \in G$ fixes each of the subspaces $V_1$, $V_2$ and $V_3$. Then $x$ is a block-diagonal matrix of the form ${\rm diag}[y,y]$ with $y \in {\rm GL}_{n/2}(K)$, so $x$ fixes $V_4$ if and only if $y$ commutes with $g$. It follows that the common stabilizer of $V_1, \ldots, V_4$ is a torus of dimension $n/2-1$, whence $b^0(G)>4$. Now take $V_5=\{(v,h(v)) \mid v \in V_1\}$ so that $g$ and $h$ generate the full matrix algebra ${\rm M}_{n/2}(K)$ (note that this is another generic condition). If $x = {\rm diag}[y,y] \in G$ stabilizes each $V_i$ ($1 \leqs i \leqs 5$) then $y$ commutes with both $g$ and $h$, whence $y$ (and thus $x$) is a scalar and we conclude that $b^1(G) = 5$, as claimed.

Now suppose $1 < d  < n/2$ and $d$ divides $n$, so $k=n/d$. Let $V_1, \ldots, V_k$ be generic $d$-dimensional subspaces of $V$, so $V = V_1 \oplus \cdots \oplus V_k$. Choose another $d$-dimensional subspace $V_{k+1}$ so that its intersection with any sum of the subspaces $V_1, \ldots, V_k$ (except $V = \sum_{i \leqs k}V_i$ of course) is trivial. The $G$-stabilizer of $V_1, \ldots, V_k$ and $V_{k+1}$ is isomorphic to ${\rm GL}(V_1)$ (embedded diagonally), whence
$b^0(G)>k+1$. Now take 
$$V_{k+2} = \{(v, f_2(v), \ldots, f_k(v)) \mid v \in V_1\},$$ 
where the matrices $f_i \in {\rm M}_{d}(K)$ generate the full matrix algebra.
As before we deduce that the stabilizer in $G$ of these $k+2$ subspaces consists of scalars, hence $b^1(G)=k+2$ as claimed.

Finally, let us turn to (ii), so $d$ does not divide $n$. Since $k=\lceil n/d \rceil$ we have 
$$(k-1)d+1 \leqs n \leqs kd-1.$$ 
First we claim  that $b^0(G)  \geqs k + 1$.  Let $V_1, \dots, V_k$ be generic $d$-dimensional subspaces. We may assume that $W=\sum_{i<k}V_{i}$ is a direct sum, and we may write $V_k = U \oplus (V_k \cap W)$ for some nontrivial subspace $U$. Clearly, any $x \in G$ that preserves $U$ and acts as a scalar on $W$ preserves each $V_i$. The claim follows.

To complete the proof, it suffices to produce $k + 2 + \delta_{3,k}$ subspaces of dimension $d$ whose common stabilizer in $G$ consists of scalars (for then the generic stabilizer of $k + 2 + \delta_{3,k}$ $d$-dimensional subspaces is finite, and therefore trivial by the remark at the beginning of the proof).

First assume that $k \geqs 5$.  Let $V_1, \ldots, V_k$ be generic $d$-dimensional spaces and set $W_1 = \sum_{i<k}V_i$ and $W_2 = \sum_{i>1}V_i$. We may assume that $W_1$ and $W_2$ are direct sums. Let $U_1$ be the diagonal $d$-dimensional
subspace of $W_{1}$, and let $U_2 =\{( v, f_3(v), \ldots, f_k(v)) \mid v \in V_2\}$ where each $f_i$ is a generic isomorphism from $V_2$ to $V_i$.  Arguing as above, if $x \in G$ preserves $U_1$, $U_2$ and each $V_i$ then $x$ must be a scalar on $V_2 \oplus \cdots \oplus V_{k-1}$.  Furthermore, since
$x$ preserves $U_1$ and $U_2$ it follows that $x$ induces the same scalar on $V_1$ and $V_k$, whence $b^1(G) \leqs k+2$ as required.

For $k=4$ we need to work a bit harder.  Again, let $V_1, \ldots, V_4$ be generic $d$-dimensional spaces. We may assume that $V = \sum_{i=1}^{4}V_i$, $\dim (V_1 \cap V_4) = 1$ and that $V_1 + V_2 +V_3$ and $ V_2 +  V_3 +  V_4$ are direct sums. Define $U_1$ and $U_2$ as in the previous paragraph. Suppose $x \in G$ preserves $U_1,U_2$ and each $V_i$. Let $x_i = x|_{V_i}$ denote the restriction of $x$ to $V_i$. Then $x_1 = x_2 =x_3$  and $x_1$ commutes with
$f_3$. Moreover, $x_4$ is uniquely determined by $x_2$ and $f_4$.  In particular, since $x_4$ must preserve $V_1 \cap V_4$, generically this forces $x_4$ to be a scalar, so $x$ itself is a scalar.

Finally, let us assume $k=3$.  Let $V_1$ and  $V_2$  be generic $d$-dimensional subspaces.
Let $U_1$ be the standard diagonal $d$-dimensional subspace of $V_1 \oplus V_2$ and let $U_2$ be an additional
generic $d$-dimensional subspace of $V_1 \oplus V_2$.  Also, let $V_3$ be another $d$-dimensional space such
that $V = \sum_{i=1}^{3}V_{i}$ and $\dim(V_1 \cap V_3)=1$.  Let $U_3$ be a generic $d$-dimensional subspace
of $V_2 \oplus V_3$.  Arguing as above, we deduce that any $x \in G$ preserving each of the subspaces $V_i, U_i$ (for $1 \leqs i \leqs 3$) is a scalar, so $b^1(G) \leqs 6$.
\end{proof}

\subsubsection{Symplectic groups}\label{sss:symp}

Now assume $G={\rm Sp}_{n}$, where $n \geqs 4$ is even. There are three cases to consider:
\begin{itemize}\addtolength{\itemsep}{0.2\baselineskip}
\item[{\rm (i)}] $H=G_{U}$ is the stabilizer of a non-degenerate $d$-dimensional subspace $U$ of $V$, where $2 \leqs d < n/2$ is even;
\item[{\rm (ii)}] $H=G_{U}$ is the stabilizer of a totally singular $d$-dimensional subspace $U$ of $V$, where $1 \leqs d \leqs n/2$;
\item[{\rm (iii)}] $H=O_n$ and $p=2$ (see Table \ref{t:sub}).
\end{itemize}

First we deal with non-degenerate subspaces. Our main result is the following:

\begin{prop}\label{p:sp1}
Let $G={\rm Sp}_{n}$ and let $\Omega$ be the set of $d$-dimensional non-degenerate subspaces of $V$,  with $d < n/2$. Set $k = \lceil n/d \rceil$. Then either
$$b^0(G) = b(G) = b^1(G) = k,$$
or $n=6$, $d=2$ and $b^0(G) = b(G) = b^1(G) = 4$.
\end{prop}

We require the following lemma concerning the (imprimitive) action of $G$ on the set of $\frac{n}{2}$-dimensional non-degenerate subspaces of $V$.

\begin{lem}\label{l:sp1}
Let $G={\rm Sp}_{n}$, where $n \equiv 0 \imod{4}$, and let $\Omega$ be the set of $\frac{n}{2}$-dimensional non-degenerate subspaces of $V$. Then
$$b^0(G) = b(G) = b^1(G) = 3+\delta_{4,n}.$$
\end{lem}

\begin{proof}
By dimension, $b^0(G) \geqs 3$.
First assume $n \geqs 8$. Let $V_1, V_2, V_3$
be generic subspaces in $\Omega$.
Let $W$ be the orthogonal complement $V_1^{\perp}$, so $V = V_1 \perp W$.
Without loss of generality we may assume that $V_i=\{(v,f_i(v)) \mid v  \in V_1\}$ for $i=2,3$,
where each $f_i$ is an isomorphism $f_i:V_1 \to W$. Suppose $x \in G$ stabilizes each $V_i$. Since $x$ stabilizes $V_1$ we can write
$x = (x_1, x_2) \in {\rm Sp}(V_1) \times {\rm Sp}(W)$. Now $x$ stabilizes
$V_2$ and $V_3$ if and only if
$$(x_1(v), x_2f_i(v))= (x_1(v), f_ix_1(v))$$ 
for all $v \in V_1$ and $i=2,3$, or equivalently,
$x_2f_i = f_ix_1$ for $i=2,3$. Therefore, $t:=f_2^{-1}f_3$ must commute with $x_1$. Now generically, $t$ is a regular semisimple element of ${\rm GL}_{n/2}$ and so its centralizer is a maximal torus $T$ of ${\rm GL}_{n/2}$ (and an open subvariety of maximal tori in ${\rm GL}_{n/2}$ are of this form). If $n \geqs 8$ and $T< {\rm GL}_{n/2}$ is a generic maximal torus then
the linear span of any even number of eigenspaces will be non-degenerate, whence $T \cap {\rm Sp}_{n/2}$ is central and the result holds in this case.

Finally, suppose $n=4$. Here $T \cap {\rm Sp}_2$ is a $1$-dimensional torus and so the stabilizer of any three non-degenerate $2$-dimensional spaces is positive-dimensional. Using the notation as above, let $V_4$ be another $2$-dimensional space in $\Omega$.  Then $x_1$ must centralize $f_2^{-1}f_3$ and $f_3^{-1}f_4$, and generically they will have a trivial common centralizer. Therefore $b^0(G) = b(G) = b^1(G) = 4$ in this case.
\end{proof}

\begin{proof}[Proof of Proposition \ref{p:sp1}]
First observe that $k \geqs 3$ and $(k-1)d+2 \leqs n \leqs kd$, whence $b^0(G) \geqs k$. Indeed, if  $V_1, \ldots, V_{k-1}$ are generic elements of $\Omega$ then $W= V_1 \oplus \cdots \oplus V_{k-1}$ is non-degenerate of dimension $(k-1)d$, so ${\rm Sp}(W^{\perp})$ stabilizes each
$V_i$ and has positive dimension.

To begin with, let us assume $d \geqs 4$. First consider the case $k=3$. Let $V_1, V_2, V_3$
be generic elements of $\Omega$, so $W=V_1 \oplus V_2$ is non-degenerate of dimension $2d$. Further, we may assume that there is a non-degenerate $d$-dimensional subspace $W_1$ of $W$ such that $V_3=\{(u, f(u)) \mid u \in W_1\}$ (in terms of the decomposition $V = W \perp W^{\perp}$) for some surjective linear map $f:W_1 \rightarrow W^{\perp}$.
Suppose $x \in G$ preserves each $V_i$, and consider the restriction of $x$ to $W$, which we denote by $x_1 \in {\rm Sp}(W)$. Then $x_1$ preserves $V_1, V_2$ and $W_1$, so Lemma \ref{l:sp1} implies that $x_1$ is a scalar. Without loss of generality we may assume that $x_1=1$. Let $x_2$ denote the restriction of $x$ to $W^{\perp}$. Then $x(u,f(u))=(u,x_2f(u)) =(u,f(u))$ and thus
$x_2=1$ on the image of $f$. The result follows.

Next suppose $d,k \geqs 4$. Let $V_1, \ldots, V_k$ be generic elements of $\Omega$ and assume
$x \in G$ preserves each $V_i$. By the analysis of the case $k=3$ in the previous paragraph, $x$ acts as a scalar on each $V_i \oplus V_j \oplus V_{\ell}$ with $1 \leqs i < j < \ell \leqs k$, and the desired result follows.

Finally, let us consider the case $d = 2$. First assume $k=3$ (so $n=6$).
Let $V_1, V_2$ and $V_3$ be generic $2$-spaces in $\Omega$. Set $W_1= V_1$,
$W_2 = V_1^{\perp} \cap V_{12}$ and $W_3 = V_{12}^{\perp}$, where $V_{12}
=V_1 \oplus V_2$, so $V = W_1 \perp W_2 \perp W_3$. Note that if $x \in G$ stabilizes each $V_i$ then it also stabilizes each $W_i$.
We may assume that 
$$V_2 =\{(v,f(v),0) \mid v \in V_1\},\;\; V_3 = \{(v,f_2(v), f_3(v)) \mid v \in V_1\},$$ 
where $f, f_2:V_1 \rightarrow
W_2$ and $f_3:V_1 \rightarrow
W_3$ are isomorphisms.
In particular, if $x \in G$ stabilizes each $V_i$ then we may write
$x = (x_1, x_2, x_3) \in {\rm Sp}_{2} \times {\rm Sp}_{2} \times {\rm Sp}_{2}$, where $x_2f = fx_1$,
$x_2f_2 = f_2x_1$ and $x_3f_3 = f_3x_1$. It is straightforward
to see that generically $x_1$ belongs to a torus of ${\rm Sp}_2$,
whence $b^0(G) > 3$. Arguing as above shows that $b^1(G)=4$, as required.

Now assume that  $d=2$ and $k \geqs 4$.  Let $V_1, \ldots, V_k$ be generic $2$-spaces in $\Omega$. Set $W_i$ to be the (direct) sum of all $V_j, j \ne i$.
Then each $W_i$ is non-degenerate of codimension $2$ in $V$.  Assume that we have handled the
case $k=4$, then by induction any $x$ preserving $V_j$ ($j \ne i$) is a scalar on $W_i$ whence
on $V$.  So consider the case $k=4$ and assume that $x$ preserves each $V_i$.
Write $V=W_4 \oplus W_4^{\perp}$.
Let $V_4'$ be the projection of $V_4$ into $W_4$ with respect to this orthogonal decomposition
of $V$.   Then generically $V_4'$ is a non-degenerate $2$-space.   If $x$ preserves
each $V_i$, then $x$ also preserves $V_4'$ and so by the case $k=3$,  $x$ is a scalar
on $W_4$ (and so similarly on $W_i$ for each $i$), whence $x$ is a scalar.
 We conclude that $b^1(G)=k$.
\end{proof}

Next, let us turn our attention to stabilizers of totally singular subspaces.

\begin{prop}\label{p:sp3}
Let $G={\rm Sp}_{n}$ and let $\Omega$ be the set of $d$-dimensional totally singular subspaces of $V$,  with $d \leqs n/2$. Set $k = \lceil n/d \rceil$. Then either
$$b^0(G) = b(G) = b^1(G) = k,$$
or one of the following holds:
\begin{itemize}\addtolength{\itemsep}{0.2\baselineskip}
\item[{\rm (i)}] $n=6$, $d=2$ and $b^0(G) = b(G) = b^1(G)=4$;
\item[{\rm (ii)}] $d=n/2$ and $b^0(G) = b(G) = 4$, $b^1(G) = 5 - \delta_{2,p}$.
\end{itemize}
\end{prop}

We prove this result in a sequence of lemmas. First observe that $b^0(G) \geqs k$. Indeed, if
$V_1, \ldots, V_{k-1}$ are elements of $\Omega$ then there is a positive-dimensional unipotent subgroup of $G$ that acts trivially on a hyperplane containing $V_1 +  \cdots +  V_{k-1}$.

\begin{lem}\label{l:sp5}
If $d=1$ then $b^0(G)=b(G)=b^1(G)=n$.
\end{lem}

\begin{proof}
As above, $b^0(G) \geqs n$. First assume $n=4$.
Let $\{e_1, e_2, e_3, e_4\}$ be a generic basis for $V$ and let $\{f_1, f_2\}$ be a basis for $\langle e_1, e_2 \rangle^{\perp}$, so
$e_3 = a_1e_1 + a_2e_2 + b_1f_1 + b_2f_2$ with all coefficients
non-zero. Suppose $x \in G$ stabilizes each $\la e_i\ra$. Then $xe_3 =ce_3$ for some scalar $c \in K$, so $xe_1=ce_1$ and $xe_2=ce_2$ since $x$ preserves
$\langle f_1, f_2 \rangle = \langle e_1, e_2 \rangle^{\perp}$. Since $\langle e_1, e _2 \rangle$
is non-degenerate, it follows that $c = \pm 1$. Therefore $xe_i = \pm e_i$
for all $i$, so $x$ is a scalar and thus $b^1(G)=4$ as required.

Now assume $n \geqs 6$. Let $V_1, \ldots, V_n$ be generic elements of $\Omega$.  In particular,  we may assume that any four distinct $V_i$
generate a non-degenerate $4$-dimensional subspace.   Suppose that $x \in G$ preserves each $V_i$.  By the previous paragraph,
$x$ is a scalar on the sum of any given four of the $V_i$. But the $V_i$ generate $V$, so this implies that $x$ is a scalar on $V$.
\end{proof}

\begin{lem}\label{l:sp6}
If $d=n/2$ then $b^0(G) = b(G) = 4$ and $b^1(G) = 5 - \delta_{2,p}$.
\end{lem}

\begin{proof}
Let $H$ be the stabilizer of an element of $\Omega$ and let $Q$ denote the unipotent radical of $H$. A generic $2$-point stabilizer is a Levi subgroup $L={\rm GL}_{n/2}$ of $H$ (see Lemma \ref{l:parab} in Section \ref{s:expar}). Moreover, since $Q$ has a dense regular orbit on $\Omega$, it suffices to compute the base size for the action of $L$ on $Q$ by conjugation.

As an $L$-module, $Q$ is isomorphic to the symmetric square of
the natural $L$-module, so a stabilizer in the conjugation action of $L$ on $Q$ corresponds to the stabilizer of a non-degenerate symmetric bilinear form. Now, if $p \neq 2$ then such a stabilizer is an orthogonal group $O_{n/2}$, and Theorem \ref{t:ai}(i) implies that the intersection of two generic conjugates of $O_{n/2}$ is finite but not trivial. We conclude that  $b^0(G)=b(G)=4$ and $b^1(G)=5$.

Now assume $p=2$ and consider the $L$-stabilizer of a pair of generic non-degenerate symmetric bilinear forms. By conjugating
we may assume that the first form is represented by the identity matrix $I=I_{n/2}$ and the second is represented by an invertible symmetric matrix $S$. The stabilizer of this pair consists of all
$x \in L$ with $xx^{\top}=I$ and $xSx^{\top} = xSx^{-1}=S$.
Generically, $S$ is a regular semisimple matrix, so $x$ is a polynomial
in $S$ and therefore $x$ is symmetric. Thus, $x^{\top}=x = x^{-1}$
and so $x^2=1$. However, no involution
commutes with a regular semisimple element, so the pairwise stabilizer is trivial and thus $b^0(G)=b(G)=b^1(G)=4$.
\end{proof}

\begin{lem}\label{l:sp7}
Suppose $d \geqs 2$ and $k = 3$. Then
$$b^0(G)=b(G)=b^1(G) = 3+\delta_{2,d}.$$
\end{lem}

\begin{proof}
By definition of $k$ we have $2d+2 \leqs n \leqs 3d$. Let $V_1, V_2, V_3$ be generic elements of $\Omega$. We may assume that
$W_1 = V_1 \oplus V_2$ is non-degenerate of dimension $2d$. Set $W_3 = W_1^{\perp}$, so $V = V_1 \oplus V_2 \oplus W_3$ and note that we may assume that
$V_3=\{(v, f_2(v), f_3(v)) \mid v \in V_1\}$, where $f_2:V_1 \to V_2$ is an isomorphism and $f_3:V_1 \to W_3$ is a linear surjection. Suppose $x \in G$ stabilizes each $V_i$ (and therefore also $W_3$). Let $x_1,x_2$ denote the restriction of $x$ to $V_1,V_2$, respectively, and let $x_3 \in {\rm Sp}_{n-2d}$ be the restriction of $x$ to $W_3$. Note that $x$ preserves each $V_i$ if and only if
$$x_2 = x_1^{-\top}, \; x_1^{\top}f_2 x_1 = f_2 \mbox{ and }
x_3f_3=f_3x_1.$$

It is not difficult to see that for a generic $f_2$,
the subgroup $\{y \in {\rm GL}_{d}  \mid y^{\top}f_2 y=f_2\}$ is a torus $T$ of dimension $\lfloor d/2 \rfloor$
(and the only scalar in $T$ is in the center of ${\rm Sp}_d$).
Note that $x_1 \in T$, so $x_1$ preserves the form defined by $f_2$. If $n<3d$ then $d \geqs 3$ and
$f_3$ has a nontrivial kernel $L$, whence $L$ must be $x_1$-invariant (since $x_3f_3=f_3x_1$).  Given a generic
subspace $L$, no nontrivial element of $T$ preserves $L$, whence $x_1$ is trivial.
Since $x_2 = x_1^{-\top}$ and $x_3f_3=f_3x_1$,
we deduce that $x_2$ and $x_3$ are also trivial, so $x$ is trivial and thus
$b^0(G)=b(G)=b^1(G) = 3$.

Finally, let us assume $n=3d$, so $d \geqs 2$ is even. Here $f_3$ is an isomorphism and
thus $x_3 = f_3 x_1 f_3^{-1}$. If $d \geqs 4$ then $f_3 T f _3^{-1} \cap {\rm Sp}_{d}$
coincides with the center of ${\rm Sp}_d$,  so $x_1 = \pm 1$ and $x$ is a scalar.
Again, we conclude that $b^0(G)=b(G)=b^1(G) = 3$. However, if $d=2$ then the same argument shows that
the stabilizer of three generic subspaces is a $1$-dimensional torus, whence
$b^1(G) \geqs 4$. An easy argument now yields $b^0(G)=b(G)=b^1(G) = 4$.
\end{proof}

\begin{lem}\label{l:sp8}
Suppose $d \geqs 2$ and $k \geqs 4$. Then $b^0(G)=b(G)=b^1(G)= k$.
\end{lem}

\begin{proof}
Let $V_1, \ldots, V_k$ be generic subspaces in $\Omega$ and suppose $x \in G$ fixes each
$V_i$. Set $W = V_1 + V_2 + V_3$ and note that we may assume this is a direct sum.  Further, if $d$ is even then we may assume $W$ is non-degenerate. If $d \geqs 4$ is even then Lemma \ref{l:sp7} implies that the restriction of $x$ to $W$ is a scalar, and the result quickly follows. Now, if $d \geqs 3$ is odd then we may assume that $W$ has a $1$-dimensional radical $R$ and that each $V_i$ ($1 \leqs i \leqs 3$) intersects
$R$ trivially. By Lemma \ref{l:sp7}, $x$ is a scalar on $W/R$
and is therefore a scalar on each $V_i$ (and necessarily the same scalar). Again the result follows.

Finally, suppose $d=2$. Arguing as above, we see that it
suffices to prove the result for $k=4$ (so $n=8$). A minor variation of the previous argument gives the result; we leave the reader to check the details.
\end{proof}

\vs

This completes the proof of Proposition \ref{p:sp3}. Finally, we deal with the one extra case that arises when $p=2$.

\begin{prop}\label{p:sp9}
Suppose $G={\rm Sp}_{n}$, $p=2$ and $\Omega = G/H$, where $H=O_{n}$. Then $b^0(G)= b(G)=n$ and $b^1(G)=n+1$.
\end{prop}

\begin{proof}
We may view $G$ as acting indecomposably on the orthogonal
module $M$ of dimension $n+1$, so we can identify $\Omega$ with the set of non-degenerate hyperplanes in $M$. Now, if $V_1, \ldots, V_{n-1}$ are generic hyperplanes in $\Omega$ then their intersection is a $2$-dimensional non-degenerate subspace $X$ of $M$. Therefore, there is a positive-dimensional subgroup of $G$ acting trivially on
$M/X$, whence $b^0(G) \geqs n$.

Let $V_0 \in \Omega$ denote the non-degenerate hyperplane fixed by $H$. Let $V_1, \ldots, V_{n-1}$ be generic elements of $\Omega$, so $U_i = V_i \cap V_0$ is a hyperplane in $V_0$ for all $i \geqs 1$. Generically, the radical of each $U_i$ (with respect to the $H$-invariant alternating form on $V_0$) will be a $1$-dimensional non-degenerate subspace. Let $\L$ denote the set of $1$-dimensional non-degenerate subspaces of $V_0$. By Lemma \ref{l:so10} (see Section \ref{ss:o2}), the stabilizer in $H^0$
of $n-1$ generic elements in $\L$ is trivial, but the corresponding stabilizer in $H$ has order $2$ (indeed, there is a transvection $x \in H \setminus H^0$ fixing all $n-1$ hyperplanes; see Remark \ref{r:sop2}). Therefore $b^0(G)=n$ and $b^1(G)=n+1$.

To complete the proof, note that we can choose $V_1, \ldots, V_{n-1}$ so that the intersection of the $V_i$, $0 \leqs i \leqs n-1$, is a $1$-dimensional totally singular subspace of $M$. In this situation, there is no nontrivial element of $H$ fixing the hyperplanes $V_{1}, \ldots, V_{n-1}$, so $b(G)=n$ as required.
\end{proof}

\begin{remk}\label{r:sp2s2}
Proposition \ref{p:sp9} implies that $b^0(G)=b(G)=n$ and $b^1(G)=n+1$ for the equivalent action of $G={\rm SO}_{n+1}$ (with $p=2$) on the set of $1$-dimensional  non-singular subspaces of the natural module for $G$ (see Table \ref{t:sub}). In particular, we deduce that $b^0(G)=b(G)=4$ and $b^1(G)=5$ if $G = {\rm Sp}_{4}$, $p=2$ and $H$ is a $\C_2$-subgroup ${\rm Sp}_{2}\wr S_2$ (by Lemma \ref{l:so2}, the same conclusion holds if $p \neq 2$).
\end{remk}

\subsubsection{Orthogonal groups, $p \neq 2$}

In this section we deal with the subspace actions of orthogonal groups ${\rm SO}_{n}$,
where $p \neq 2$. We start by considering the stabilizers of non-degenerate subspaces.

\begin{lem}\label{l:so1}
Let $G={\rm SO}_{n}$, where $p \neq 2$ and $n \geqs 4$ is even. Let $\Omega$ be the set of
$\frac{n}{2}$-dimensional non-degenerate subspaces of $V$. Then
$b^0(G)=b(G)=2$ and $b^1(G)=3$.
\end{lem}

\begin{proof}
Let $H$ be the stabilizer of a subspace in $\Omega$, so $H$ is of type $O_{n/2} \times O_{n/2}$.
If $n \geqs 6$ then the result follows from Theorem \ref{t:ai}(i), so let us assume $n=4$. Here $G=A_1A_1$ and $H$ is contained in the normalizer of a maximal torus, so the same conclusion holds in this case too.
\end{proof}

\begin{prop}\label{p:so1}
Let $G={\rm SO}_{n}$ with $n \geqs 7$ and let $\Omega$ be the set of $d$-dimensional non-degenerate  subspaces of $V$,  with $1 \leqs d < n/2$. Set $k = \lceil n/d \rceil$ and assume $p \neq 2$. Then either
$$b^0(G) = b(G) = b^1(G) = k,$$
or $n=(k-1)d+1$, $b^0(G) = b(G) = k-1$ and $b^1(G)=k-\e$, where $\e=1$ if $n$ is even, otherwise $\e=0$.
\end{prop}

The proof of Proposition \ref{p:so1} is given in the next two lemmas.

\begin{lem}\label{l:so2}
If $d=1$ then $b^0(G)=b(G) = n-1$ and $b^1(G)=n-\e$, where $\e=1$ if $n$ is even, otherwise $\e=0$.
\end{lem}

\begin{proof}
It is convenient to prove this result for all $n \geqs 3$.
First observe that $b^0(G) \geqs n -1$. Indeed,
the sum of $n-1$ generic non-degenerate $1$-spaces is a non-degenerate hyperplane, so
the sum of $n-2$ generic elements of $\Omega$ is non-degenerate
and thus their common stabilizer is positive-dimensional.

If $n=3$ then the stabilizer of a non-degenerate $1$-space is the normalizer of a maximal torus, whence the result is clear in this case (see the proof of Lemma \ref{p:aic}, for example). Now assume $n \geqs 4$. By induction, any $x \in G$ stabilizing $n-1$ generic elements of $\Omega$ must act as $\pm 1$ on the corresponding non-degenerate hyperplane (the sum of $n-1$ spaces). If $n$ is even, this forces $x$ to be a scalar and the result follows. Now assume $n$ is odd. Here, either $x$ is a scalar or $-x$ is a reflection, so in this situation we have $b^1(G)=n$. Now we can also choose $n-1$  elements of $\Omega$ so that their sum is a hyperplane
with a $1$-dimensional radical; this forces $x$ to be a scalar, so $b^0(G) = b(G)=n-1$.
\end{proof}

\begin{lem}\label{l:so3}
If $d \geqs 2$ then either $b^0(G)=b(G) = b^1(G)=k$, or $n=(k-1)d+1$, $b^0(G)=b(G)=k-1$ and $b^1(G)=k-\e$, where  $\e=1$ if $n$ is even, otherwise $\e=0$.
\end{lem}

\begin{proof}
By definition of $k$ we have $(k-1)d+1 \leqs n \leqs kd$ and $k \geqs 3$. First assume $n > (k-1)d+1$. As before, we have $b^0(G) \geqs k$. Suppose $k=3$ and let $V_1, V_2, V_3$ be generic elements of $\Omega$. Without loss of generality we may assume that $W=V_1 \oplus V_2$ is a non-degenerate $2d$-space. By Lemma \ref{l:so1}, the stabilizer of $V_1$ and $V_2$ in ${\rm SO}(W)$ is finite.
The common $G$-stabilizer of $V_1, V_2$ and $V_3$ preserves
the orthogonal projection of $V_3$ into $W$, so this stabilizer acts as a scalar on $V = V_1 + V_2 + V_3$ and thus $b^0(G)=b(G) = b^1(G) = 3$. More generally, if $k \geqs 4$ and $V_1, V_2, V_3$ are generic elements of $\Omega$ then any $x \in G$ that preserves each $V_i$ acts as a scalar on $V_1 \oplus V_2 \oplus V_3$. In particular, if $x \in G$ stabilizes $k$ generic elements of $\Omega$ then $x$ is a scalar and the result follows.

Finally, let us assume $n = (k-1)d+1$. Here  $b^0(G) \geqs k-1$.   Let $V_1, \ldots, V_{k-1}$ be generic elements of $\Omega$.
Then  $W= V_1 \oplus \cdots \oplus V_{k-1}$ is a non-degenerate hyperplane. By the previous paragraph, any $x \in G$ preserving each $V_i$ acts as $\pm 1$ on $W$. If $n$ is even then we immediately deduce that $b^1(G)=k-1$. If $n$ is odd, let $x$ be the reflection with fixed space $W$.  Then $-x \in G$ fixes each of the $V_i$, whence $b^1(G)=k$ in this case.

We can also choose the $V_i$ so that $W$ is a hyperplane with a $1$-dimensional radical $R$.  It follows by induction that any $x \in G$ preserving each $V_i$ must be a scalar on $W/R$.  Since we may assume that $R$ is not contained in any of the $V_i$, this implies that $x$ is a scalar on $W$, and thus a scalar on the whole space $V$. We conclude that $b(G)= k -1$.
\end{proof}

\vs

This completes the proof of Proposition \ref{p:so1}. Next we turn our attention to totally singular subspaces, and we continue to assume that $p \neq 2$. Let $H$ be the stabilizer of a totally singular $d$-dimensional subspace of $V$ and set $\Omega=G/H$. Note that if $d<n/2$ then $\Omega$ is the set of all totally singular $d$-dimensional subspaces of $V$, whereas if $d=n/2$ then there are two distinct $G$-orbits on such subspaces, which are interchanged by a graph automorphism of $G={\rm SO}_{n}$. In particular, if $d=n/2$ then the two $G$-actions are permutation isomorphic. Our main result is the following, which we prove in Lemmas \ref{l:so5} -- \ref{l:so8} below.

\begin{prop}\label{p:so4}
Let $G={\rm SO}_{n}$ with $n \geqs 7$, let $H$ be the stabilizer of a totally singular $d$-dimensional subspace of $V$ with $1 \leqs d \leqs n/2$ and set $\Omega = G/H$. Assume $p \neq 2$ and set $k = \lceil n/d \rceil$. Then either
$$b^0(G) = b(G) = b^1(G) = k,$$
or one of the following holds:
\begin{itemize}\addtolength{\itemsep}{0.2\baselineskip}
\item[{\rm (i)}] $d=n/2$, $n \neq 10$ and $b^0(G) = b(G) = b^1(G)=c(n)$, where $c(8)=7$, $c(12)=6$ and $c(n) = 5$ for all $n \geqs 14$;
\item[{\rm (ii)}] $n=10$, $d=5$ and $5 \leqs b^0(G) \leqs b^1(G) \leqs 6$;
\item[{\rm (iii)}] $k=3$ and $b^0(G) = b(G) = b^1(G)=4-\delta_{n,3d}$;
\item[{\rm (iv)}] $k \geqs 4$, $n=(k-1)d+1$, $b^0(G) = b(G) = k-1$ and $b^1(G) = k-\e$, where $\e=1$ if $n$ is even, otherwise $\e=0$.
\end{itemize}
\end{prop}

\begin{lem}\label{l:so5}
If $n \geqs 5$ and $d =1$ then $b^0(G) = b(G)=n-1$ and $b^1(G) = n-\e$, where $\e=1$ if $n$ is even, otherwise $\e=0$.
\end{lem}

\begin{proof}
Clearly, $b^0(G) \geqs n -1$. We can choose
$n-1$ subspaces in $\Omega$ so that their sum is a hyperplane with
a $1$-dimensional radical (and the radical does not coincide with any of the $n-1$ spaces). It follows that any element stabilizing this hyperplane is a scalar, whence
$b^0(G) = b(G) = n-1$. Generically, the hyperplane is non-degenerate and we now complete the argument by proceeding as in the proof of Lemma \ref{l:so2}.
\end{proof}

\begin{lem}\label{l:so6}
If $d=n/2$ then either $n=10$ and $5 \leqs b^0(G) \leqs b^1(G) \leqs 6$, or $b^0(G) = b(G) = b^1(G)=c(n)$ where $c(8)=7$, $c(12)=6$ and $c(n) = 5$ for all $n \geqs 14$.
\end{lem}

\begin{proof}
First observe that $\dim \Omega = n^2/8-n/4$ so Proposition \ref{p:bb}(iii) yields
$b^0(G) \geqs 5$. To begin with, let us assume $d$ is even. The intersection of two generic conjugates of $H$ is a Levi
subgroup $L \cong {\rm GL}_d$ of $H$ (see Lemma \ref{l:parab}). Let $Q$ be the unipotent radical of $H$. Now $Q$ has a dense orbit on $\Omega$, so the intersection of three generic conjugates
of $H$ coincides with the stabilizer in $L$ of a non-degenerate alternating form
on the natural $d$-dimensional $L$-module $U$. Since $d$ is even, this stabilizer is a symplectic group ${\rm Sp}_d$.
Consequently, a generic $4$-point stabilizer in $G$ is the intersection
of the $L$-stabilizers of two non-degenerate
alternating forms on $U$. The desired result now follows from \cite[Theorem 1.1]{GG}.

Next suppose $d \geqs 7$ is odd.
Let $V_1, \ldots, V_5$ be generic subspaces in $\Omega$.
Then $W = V_1 + V_2$ is a hyperplane with a $1$-dimensional
radical $R$ (note that any two complementary $d$-dimensional totally singular subspaces
of $V$ are in different $G$-orbits, since $d$ is odd). In particular, if $i > 2$ then $V_i \cap W$ is a
$(d-1)$-dimensional totally singular subspace of $W$ (which intersects $R$ trivially).
We may assume that the $V_i \cap W$ are all in the same ${\rm SO}(W/R)$-orbit, so by the previous paragraph it follows that the common
$G$-stabilizer of the $V_i$ induces a scalar on $W/R$. In fact, since $W = \sum_{i}(V_i \cap W)$, it follows that the common $G$-stabilizer acts as a scalar on $W$, and thus a scalar on $V$. We conclude that $b^0(G)=b(G)=b^1(G)=5$ as required.

Finally, the same argument shows that $b^1(G) \leqs 6$ if $d = 5$.
\end{proof}

To complete the proof of Proposition \ref{p:so4}, we may assume that $k \geqs 3$ and $d \geqs 2$.

\begin{lem}\label{l:so7}
Suppose $k \geqs 5$ and $d \geqs 2$. Then either $b^0(G)=b(G)=b^1(G)=k$, or $n=(k-1)d+1$, $b^0(G)=b(G)=k-1$ and $b^1(G)=k-\e$, where $\e=1$ if $n$ is even, otherwise $\e=0$.
\end{lem}

\begin{proof}
By definition of $k$ we have $(k-1)d+1 \leqs n \leqs kd$. For now let us assume $n > (k-1)d+1$, in which case $b^0(G) \geqs k$ (since the sum
of any $k-1$ subspaces in $\Omega$ has codimension at least $2$). Let $V_1, \ldots, V_k$ be generic elements of $\Omega$ and let $W = V_1 \oplus \cdots \oplus V_4$, so $\dim W = 4d$ and $W$ is non-degenerate. We may also assume that $V_1 \oplus  V_2$ and $V_3 \oplus V_4$ are non-degenerate. By (the proof of) Lemma \ref{l:so1}, the common stabilizer of $V_1 \oplus  V_2$ and $V_3 \oplus V_4$ in ${\rm SO}(W)$ is an elementary abelian $2$-group.  The same is true for any such combination of $V_1, \ldots, V_4$ into complementary non-degenerate $2d$-dimensional spaces.  It quickly follows that the common stabilizer of $V_1, V_2, V_3$ and $V_4$ in $G$  induces only scalars on $W$.  The same is true for the sum of any four of the $V_i$, so the $G$-stabilizer of $V_1, \ldots, V_k$ acts as scalars on $V$ and the result follows.

Now assume $n=(k-1)d+1$, so $b^0(G) \geqs k-1$. Let $V_1, \ldots, V_{k-1}$ be generic elements of $\Omega$, so $W=\sum_{i}V_i$ is a non-degenerate hyperplane. The common $G$-stabilizer of the $V_i$ acts as a scalar $\pm 1$ on $W$, so $b^0(G)=b(G)=k-1$ and in the usual way we deduce that $b^1(G)=k-\e$, where $\e=1$ if $n$ is even, otherwise $\e=0$.
\end{proof}

\begin{lem}\label{l:so71}
Suppose $k =4$ and $d \geqs 2$. Then either $b^0(G)=b(G)=b^1(G)=4$, or $n=3d+1$, $b^0(G)=b(G)=3$
and $b^1(G)=4-\e$, where $\e=1$ if $n$ is even, otherwise $\e=0$.
\end{lem}

\begin{proof}
Note that $3d+1 \leqs n \leqs 4d$ and
$$\frac{\dim G}{\dim \Omega} = \frac{n(n-1)}{d(2n-3d-1)} \geqs 3,$$
with equality if and only if $n=3d+1$.

If $n=3d+1$ then $b^0(G) \geqs 3$ (by Proposition \ref{p:bb}(iii)) and the result follows by repeating the argument in the final paragraph of the proof of Lemma \ref{l:so7}. Now assume $n \geqs 3d+2$, so $b^0(G) \geqs 4$. Let $V_1, \ldots, V_4$ be generic subspaces in $\Omega$.
Let $W=V_1 + V_2$, so $W$ is non-degenerate and $2d$-dimensional.
Generically, the orthogonal projections of $V_3$, $V_4$ into $W$ and $W^{\perp}$
are injective with non-degenerate images (these are open conditions,
so one only has to see that it is possible). If $x \in G$ preserves each $V_i$ then $x$ preserves two non-degenerate $d$-dimensional subspaces of $W$, namely the projections of $V_3$ and $V_4$. By (the proof of) Lemma \ref{l:so1}, the stabilizer of these $d$-spaces in ${\rm SO}(W)$ is a finite $2$-group.  However, $x$ also preserves $V_1$ and $V_2$, so $x$ must induce a scalar on $W$.  By symmetry, the same is true for each combination $V_i + V_j$, whence $x$ is a scalar on $V$ and the result follows.
\end{proof}

Finally, let us assume $k=3$. Note that $d \geqs 3$ since $n \geqs 7$.

\begin{lem}\label{l:so8}
Suppose $k=3$ and $d \geqs 3$. Then $b^0(G)=b(G)=b^1(G)=4-\delta_{n,3d}$.
\end{lem}

\begin{proof}
Since $k=3$ we have $2d+1 \leqs n \leqs 3d$.  As in the proof of the previous lemma we have
$\dim G / \dim \Omega \geqs 3$, with equality if and only if $n=3d$.  Thus, $b^0(G) \geqs4$ unless possibly $n=3d$, in which case $b^0(G) \geqs 3$.

Let $H$ be the stabilizer of a subspace in $\Omega$ and note that the generic intersection of two conjugates of $H$ is a
Levi subgroup $L \cong {\rm GL}_d \times {\rm SO}_{n-2d}$ (see Lemma \ref{l:parab}). Now $H=QL$, where $Q$ is the unipotent radical of $H$,
and since $Q$ has a regular
dense orbit on $\Omega$ it suffices to compute the base size for the action of $L$ on $Q$ by conjugation. Let $X$ and $Y$ be the natural modules for ${\rm GL}_d$ and ${\rm SO}_{n-2d}$, respectively.

Now $Q$ has a normal subgroup $Q_1$ with
$Q_1 \cong \L^2(X)$ as a ${\rm GL}_d$-module (with ${\rm SO}_{n-2d}$
acting trivially) and $Q/Q_1 \cong Q_2 \cong X \otimes Y$
as an $L$-module.  Since $Q \cong Q_1 \times Q_2$
as varieties, it is sufficient to consider the action of $L$ on $Q_1  \times Q_2$.
Here the point stabilizer in $L$ of a generic point $(q_1,q_2)$ is isomorphic
to the intersection of the stabilizer $L_1< {\rm GL}_{d}$ of a non-degenerate alternating
form on $X$ and a subgroup $L_2$ fixing an $(n-2d)$-dimensional
subspace $X'$ of $X$ (more precisely, $L_2$ acts as ${\rm SO}(X')$ on $X'$).

First assume $n = 3d$.  Then $L_1 \cong {\rm Sp}_d$   and $L_2 \cong {\rm SO}_d$,
embedded diagonally in $L$.  In particular, we see  that $L_1 \cap L_2=1$, whence $b^1(G)=3$ and the result follows.
Now assume $n<3d$. It is straightforward to show that the intersection of
two generic conjugates of such a subgroup $L_1 \cap L_2$ of $L$ is trivial.
For example, if $n =2d+1$ then $L_1$ is the stabilizer of an alternating form and $L_2$ is the stabilizer of a vector $x \in X$. The result follows.
\end{proof}

\vs

This completes the proof of Proposition \ref{p:so4}.

\subsubsection{Orthogonal groups, $p=2$}\label{ss:o2}

To complete the analysis of subspace actions we may assume that $G={\rm SO}_{n}$, where $n \geqs 7$ and $p=2$. The arguments are similar (and often easier) to the case $p \neq 2$. The main difference here is that we may assume $n$ is even. In addition, in the analysis of non-degenerate $d$-dimensional subspaces we may assume that $d=1$ or $d$ is even. Indeed, any odd dimensional space has a radical when considered as an alternating space, so the action is imprimitive if the dimension is greater than $1$. (For convenience, we will refer to $1$-dimensional \emph{non-degenerate} subspaces, although strictly speaking we should use the term \emph{non-singular}.)

\begin{prop}\label{p:so9}
Let $G={\rm SO}_{n}$, where $p=2$ and $n \geqs 8$ is even. Let $\Omega$ be the set of $d$-dimensional
non-degenerate subspaces of $V$,  where $d < n/2$ and either $d=1$ or $d$ is even. Set $k = \lceil n/d \rceil$. Then
$b^0(G) = b(G) = b^1(G) = k-\e$, where $\e=1$ if $n=(k-1)d+1$, otherwise $\e=0$.
\end{prop}

\begin{lem}\label{l:so10}
Assume that $n \geqs 4$ is even and let $\Omega$ be the set of $1$-dimensional non-degenerate or totally singular subspaces of $V$. Then $b^0(G)=b(G)=b^1(G) = n -1$.
\end{lem}

\begin{proof}
We induct on $n$. First assume that $n=4$, so $G={\rm SL}_2 \times {\rm SL}_2$.  The stabilizer of a singular $1$-space is a Borel subgroup and thus $b^0(G)=b(G)=b^1(G) =3$.   The stabilizer of a non-degenerate $1$-space is a diagonal copy
of $\rm{SL}_2$ (the centralizer of an outer involution), and the result is an easy computation. (In the latter case we could start the induction at $n=2$, where the stabilizer of a non-degenerate $1$-space is trivial.)

Clearly we have $b^0(G) \geqs n-1$.  Let $V_1, \ldots, V_{n-1}$
be generic subspaces in $\Omega$. We may assume that the sum $W=\sum_iV_i$ is a hyperplane with a $1$-dimensional radical $R$ (with respect to the underlying symmetric  form on $V$). Moreover, we may assume that the defining quadratic form on $V$ does not vanish on $R$. By induction, it follows that the common $G$-stabilizer of each $V_i$ is trivial on $W/R$, and therefore trivial on each $V_i$. In particular, the common stabilizer is trivial on $W$ and so also on $V$. The result follows.
\end{proof}

\begin{remk}\label{r:sop2}
In the previous lemma, if $G=O_n$ is the full orthogonal group and $\Omega$ is the set of $1$-dimensional non-degenerate subspaces of $V$, then $b^0(G)=b(G)=n-1$ and $b^1(G) = n$. More precisely, the $G$-stabilizer of $n-1$ generic elements of $\Omega$ contains a transvection and has order $2$.
\end{remk}

The next result shows that the conclusion to Lemma \ref{l:so1} also holds when $p=2$.

\begin{lem}\label{l:so11}
Suppose $n \equiv 0 \imod{4}$ and let $\Omega$ be the set of $\frac{n}{2}$-dimensional non-degenerate subspaces of $V$. Then $b^0(G)=b(G)=2$ and $b^1(G) = 3$.
\end{lem}

\begin{proof}
Let $H$ be the stabilizer of a subspace in $\Omega$, so $H$ is of type $O_{n/2} \times O_{n/2}$. It is straightforward to see that the $G$-stabilizer of two generic subspaces in $\Omega$ coincides with the intersection in ${\rm GL}_{n/2}$ of two generic conjugates of $O_{n/2}$.

Now, the generic intersection in ${\rm GL}_{n/2}$ of two conjugates
of ${\rm Sp}_{n/2}$ is isomorphic to the direct product of $n/4$ copies of ${\rm SL}_2$. Therefore, the intersection of generic conjugates of $O_{n/2}$ and ${\rm Sp}_{n/2}$
is the normalizer of a torus in the direct product $({\rm SL}_2)^{n/4}$. Consequently,
we deduce that the intersection in ${\rm GL}_{n/2}$ of two generic conjugates of
$O_{n/2}$ is elementary abelian of order $2^{n/4}$, and it can be trivial. The result follows.
\end{proof}

The remainder of the proof of Proposition \ref{p:so9} is entirely similar to the argument given in the case $p \neq 2$, the only difference being that certain cases do not arise when $p=2$. We leave the details to the reader.

Finally, let us consider the stabilizers of totally singular subspaces.

\begin{prop}\label{p:so12}
Let $G={\rm SO}_{n}$, where $p=2$ and $n \geqs 8$ is even. Let $H$ be the stabilizer of a totally singular $d$-dimensional subspace of $V$ with $1 < d \leqs n/2$ and set $\Omega=G/H$ and $k = \lceil n/d \rceil$. Then either
$$b^0(G) = b(G) = b^1(G) = k,$$
or one of the following holds:
\begin{itemize}\addtolength{\itemsep}{0.2\baselineskip}
\item[{\rm (i)}] $d=n/2$, $n \neq 10$ and $b^0(G) = b(G) = b^1(G)=c(n)$, where $c(8)=7$, $c(12)=6$ and $c(n) = 5$ for all $n \geqs 14$;
\item[{\rm (ii)}] $n=10$, $d=5$ and $5 = b^0(G) \leqs b(G) \leqs b^1(G) \leqs 6$;
\item[{\rm (iii)}] $k=3$ and $b^0(G) = b(G) = b^1(G)=4-\delta_{n,3d}$;
\item[{\rm (iv)}] $k \geqs 4$, $n=(k-1)d+1$ and $b^0(G) = b(G) = b^1(G) = k-1$.
\end{itemize}
\end{prop}

Once again, the proof of this proposition is very similar to the case $p \neq 2$ (see Proposition \ref{p:so4}). We leave the reader to make the necessary minor modifications. Note that $b^0(G)=5$ in case (ii): the usual argument yields $b^0(G)\geqs 5$, and a straightforward {\sc Magma} calculation gives
$b(G_{\s}) \leqs 5$ for the corresponding action of $G_{\s} = \Omega^{+}_{10}(4)$, so $b^0(G) \leqs 5$ by  Proposition \ref{p:bb2}(ii).

\vs

This completes the proof of Theorem \ref{t:csub}.

\subsection{Non-subspace actions}\label{ss:irred}

Here we complete the proof of Theorem \ref{t:cmain}. Let
$G$ be a simple classical algebraic group over an algebraically closed field of characteristic $p \geqs 0$ and let
$\Omega$ be a primitive non-subspace $G$-variety with point stabilizer $H$. By the main theorem of \cite{LieS} (see Theorem \ref{t:ls}), we may assume that $H$ is a positive-dimensional subgroup in one of the collections $\C_2, \C_3, \C_4, \C_6$ or
$\mathcal{S}$. In fact, in view of Theorem \ref{t:ai}, we may assume that $H \in \C_2 \cup \C_4 \cup \mathcal{S}$. Our first result deals with the tensor product subgroups in $\C_4$ and the irreducible almost simple subgroups in $\mathcal{S}$.

\begin{prop}\label{p:bgs2}
If $H \in \C_4 \cup \mathcal{S}$ is positive-dimensional then one of the following holds:
\begin{itemize}\addtolength{\itemsep}{0.2\baselineskip}
\item[{\rm (i)}] $b^0(G,H) = b(G,H) = b^1(G,H) = 2$; or
\item[{\rm (ii)}] $(G,H) = ({\rm SO}_{7},G_2)$ (with $p \neq 2$) or $({\rm Sp}_{6},G_2)$ (with $p=2$), and
$$b^0(G,H) = b(G,H) = b^1(G,H) = 4.$$
\end{itemize}
\end{prop}

\begin{proof}
Assume $p>0$ and let $K$ be the algebraic closure of the prime field $\mathbb{F}_{p}$. Let $\s$ be a Frobenius
morphism of $G$ such that $G_{\s}$ is an almost simple classical group over $\F$, where $q$ is a $p$-power.
We may assume $H$ is $\s$-stable. If $(G,H) \neq ({\rm SO}_{7},G_2)$ or $({\rm Sp}_{6},G_2)$ then the proof of the main theorem of
\cite{BGS2} implies that $b(G_{\s},H_{\s})=b^{\infty}(G_{\s},H_{\s})=2$ with respect to the action of $G_{\s}$ on $G_{\s}/H_{\s}$.
Therefore Proposition \ref{p:bb2}(i) yields $b^1(G,H)=2$, hence (i) holds. (Note that we could verify this independently of \cite{BGS2},
by applying Theorem \ref{t:con}, but it is convenient to use our results for the corresponding finite group actions.
In this way we see that the same conclusion holds if $K$ is \emph{any} algebraically closed field of characteristic $p \geqs 0$ (this is
discussed in more detail at the end of Section \ref{s:prel}).)

Now assume $(G,H) = ({\rm SO}_{7},G_2)$ or $({\rm Sp}_{6},G_2)$. By considering the corresponding action of
$G_{\s}$ on the set of cosets of $G_2(q)$, and by inspecting the proof of \cite[Proposition 2]{PLS},
we deduce that the generic $2$-point stabilizer in the action of $G$ on $\Omega$ has connected component $A_2$.
Now $\dim G_2+\dim A_2 > \dim G$, so every $3$-point stabilizer is positive-dimensional and thus $b^0(G,H) \geqs 4$.
According to the proof of \cite[Lemma 7.7]{Bur2} we have $\dim x^H \leqs \frac{3}{4}\dim x^G$ for all $x \in H$ of prime order
(including all nontrivial unipotent elements if $p=0$), with equality if and only if $x$ is a long root element.
Therefore Corollary \ref{c:lr} implies that $b^1(G,H) \leqs 4$, as required (note that each long root subgroup of $H=H^0$ is a long root subgroup of $G$).
\end{proof}

\vs

\noindent \emph{Proof of Theorem \ref{t:cmain}.}

\vs

We may assume $H$ is a $\C_2$-subgroup that stabilizes a direct sum decomposition 
$$V=V_1 \oplus \cdots \oplus V_t$$ 
with $t \geqs 3$ (if $t=2$ then $H$ is one of the involution-type subgroups considered in Lemma \ref{p:aic}). If $G={\rm SL}_{n}$ or ${\rm SO}_{n}$ then \cite[Proposition 2.1]{Bur5} implies that $\dim x^H \leqs \frac{1}{t}\dim x^G$ for all $x \in H$ of prime order, whence Corollary \ref{c:con} yields
\begin{equation}\label{e:all2}
b^0(G,H)=b(G,H)=b^1(G,H)=2.
\end{equation}

Now assume $G={\rm Sp}_{n}$ and $H$ is of type ${\rm Sp}_{n/t}\wr S_t$ with $t \geqs 3$. Here \cite[Proposition 2.1]{Bur5} yields
$$\dim x^H \leqs \left(\frac{1}{t}+\frac{2}{n+2}\right)\dim x^{G}$$
for all $x \in H$ of prime order (and all nontrivial unipotent elements if $p=0$), so Corollary \ref{c:con} implies that \eqref{e:all2} holds unless $(n,t)=(6,3)$.  Here $b^1(G,H) \leqs 3$ and we claim that
$$b^0(G,H)=b(G,H)=b^1(G,H)=3.$$
By \cite[Lemma 4.1]{GG}, there is a \emph{self-adjoint} element $g \in {\rm GL}_{6}$ such that $C_{{\rm GL}_{6}}(g)=H$, so according to \cite[Lemma 2.2]{GG} there exists $x \in {\rm GL}_{6}$ with $G \cap G^x = H$.
In particular, if $y \in G$ then
$$H \cap H^y = G \cap G^x \cap G^{xy},$$
so \cite[Lemma 5.7]{GG} implies that $\dim(H \cap H^y)>0$. Therefore $b^0(G,H) \geqs 3$ and the claim follows.

\vs

This completes the proof of Theorem \ref{t:cmain}.

\section{Exceptional groups}\label{s:exc}

In this section we complete the proof of Theorems \ref{t:mainep1} and \ref{t:mainep2}.
Let $G$ be a simple exceptional algebraic group over an algebraically closed field $K$ of characteristic $p \geqs 0$. Let us recall the main theorem on the subgroup structure of $G$, which is due to Liebeck and Seitz \cite{LSmem}.

\begin{thm}\label{e:max}
Let $H$ be a positive-dimensional maximal closed subgroup of $G$. Then one of the following holds:
\begin{itemize}\addtolength{\itemsep}{0.2\baselineskip}
\item[{\rm (i)}] $H$ is a parabolic subgroup;
\item[{\rm (ii)}] $G=E_7$, $p \neq 2$ and $H=(2^2 \times D_4).S_3$;
\item[{\rm (iii)}] $G=E_8$, $p \neq 2,3,5$ and $H=A_1 \times S_5$;
\item[{\rm (iv)}] $H=N_{G}(X)$, with $X$ given in Table \ref{t:max}.
\end{itemize}
\end{thm}

\begin{proof}
This is \cite[Corollary 2]{LSmem}. Note that in Table \ref{t:max}, $D_4<F_4$ is the subgroup generated by all long root subgroups, and if $p=2$ we write $\tilde{D}_4<F_4$ to denote the subgroup generated by all short root subgroups. Similarly, we define $A_2<G_2$ and $\tilde{A}_2<G_2$ (if $p=3$).
\end{proof}

\begin{table}[h]
$$\begin{array}{lll} \hline
G & X & N_G(X)/X \\ \hline
E_8 & A_1,\; B_2,\; A_1A_2,\; A_1G_2^2\,(p\neq 2),\; G_2F_4 & 1,\; 1,\; Z_2,\; Z_2,\; 1\\
& D_8,\; A_1E_7,\; A_8, A_2E_6, \; A_4^2,\; D_4^2  & 1,\; 1,\; Z_2,\; Z_2,\; Z_4,\; Z_2 \times S_3 \\
& A_2^4,\; A_1^8,\; T_8 & {\rm GL}_{2}(3),\; {\rm AGL}_{2}(3), \; 2.O_{8}^{+}(2) \\
E_7 & A_1,\; A_2,\; A_1^2,\; A_1G_2,\; A_1F_4,\; G_2C_3 & 1,\; Z_2,\; 1,\; 1,\; 1,\; 1 \\
& T_1E_6,\; A_1D_6,\; A_7,\; A_2A_5 & Z_2,\; 1,\; Z_2,\; Z_2 \\
& A_1^3D_4,\; A_1^7,\;T_7 & S_3,\; {\rm GL}_{3}(2),\; Z_2 \times {\rm Sp}_{6}(2) \\
E_6 & A_2,\; G_2,\; C_4\,(p\neq 2),\; F_4, \; A_2G_2 & Z_2,\; 1,\; 1,\; 1,\; Z_2 \\
& T_1D_5,\; T_2D_4,\; A_1A_5,\; A_2^3,\; T_6 & 1,\; S_3,\; 1,\; S_3,\; O_{6}^{-}(2) \\
F_4 & A_1,\; G_2,\; A_1G_2,\;A_1C_3 & 1,\; 1,\; 1,\; 1 \\
& B_4,\; C_4\,(p=2),\; D_4, \; \tilde{D}_4\,(p=2),\; A_2\tilde{A}_2 & 1,\; 1,\; S_3,\; S_3,\; Z_2 \\
G_2 & A_1,\; A_1\tilde{A}_{1},\; A_2,\; \tilde{A}_{2}\,(p=3) & 1,\; 1,\; Z_2,\; Z_2  \\ \hline
\end{array}$$
\caption{Some maximal non-parabolic subgroups of exceptional groups}
\label{t:max}
\end{table}

\subsection{Parabolic actions}\label{s:expar}

First let us consider Theorem \ref{t:mainep1}, so $\Omega=G/H$ and $H$ is a maximal parabolic subgroup of $G$. Recall that $H$ is conjugate to a standard parabolic subgroup $P_i$ for some $1 \leqs i \leqs r$, where $r$ denotes the rank of $G$. Further, this notation indicates that if $L_i$ is a Levi subgroup of $P_i$ then the root system of the semisimple group $L_i'$ corresponds to the Dynkin diagram of $G$ with the $i$-th node deleted. We continue to follow Bourbaki \cite{Bou} in the labelling of Dynkin diagrams.

\begin{prop}\label{p:parab}
Let $G$ be a simple exceptional algebraic group and let $\Omega=G/H$, where $H=P_i$ is a maximal parabolic subgroup of $G$. Then
$$c-\e \leqs b^0(G,H) \leqs b(G,H) \leqs b^1(G,H) \leqs c,$$
where $c$ is defined in Table \ref{t:ep22}. Here an asterisk indicates that $\e=1$, otherwise $\e=0$ and thus $b^0(G,H)=b(G,H) = b^1(G,H) = c$.
\end{prop}

\begin{proof}
Let $P_i=Q_iL_i$ be a Levi decomposition of $P_i$ and observe that $\dim \Omega = \dim Q_i = |\Phi^+(G)|-|\Phi^+(L_i')|$; for the reader's convenience we record this dimension in Table \ref{t:parab}. By Proposition \ref{p:bb}(iii) we have $b^0(G,H) \geqs \dim G /\dim \Omega$, while an upper bound for $b^1(G,H)$ is  obtained by combining Proposition \ref{p:bb2}(i) and \cite[Theorem 3]{BLS}. The result follows.
\end{proof}

\begin{table}[h]
$$\begin{array}{r|rlllllll}
 & H=P_{1} & P_{2} & P_{3} & P_{4} & P_{5} & P_{6} & P_{7} & P_{8} \\ \hline
G=E_{8} & 78 & 92 & 98 & 106 & 104 & 97 & 83 &  57 \\
E_{7} & \hspace{1.7mm} 33 & 42 & 47 & 53 & 50 & 42 & 27 & \\
E_{6} & 16  & 21 & 25 & 29 & 25 & 16 & & \\
F_{4} & 15 & 20 & 20 & 15 & & & & \\
G_{2} & 5 & 5 & & & & & & \\
\end{array}$$
\caption{$G$ exceptional, $\dim G/P_i$}
\label{t:parab}
\end{table}

\begin{table}[h]
$$\begin{array}{r|rlllllll}
 & H=P_{1} & P_{2} & P_{3} & P_{4} & P_{5} & P_{6} & P_{7} & P_{8} \\ \hline
G=E_{8} & 4  \hspace{3.3mm} & 3 & 3 & 3 & 3 & 3 & 4^* & 5 \\
E_{7} &  5 \hspace{3.3mm} & 4 & 4^* & 3 & 3 & 4 & 6^* & \\
E_{6} & 6^* \hspace{1.5mm} & 5^* & 4 & 4^* & 4 & 6^* & & \\
F_{4} & 5^* \hspace{1.5mm} & 4^* & 4^* & 5^* & & & & \\
G_{2} & 4^* \hspace{1.5mm} & 4^* & & & & & & \\
\end{array}$$
\caption{$G$ exceptional, $H$ parabolic}
\label{t:ep22}
\end{table}

In order to complete the proof of Theorem \ref{t:mainep1} we may assume that $(G,H)$ is one of the following cases:
$$(E_8,P_7),\; (E_7,P_3),\; (E_7,P_7),\; (E_6,P_1),\; (E_6,P_2),\; (E_6,P_6).$$

The next lemma is a key result in our analysis, and it holds for any semisimple algebraic group $G$ over an algebraically closed field.

\begin{lem}\label{l:parab}
Let $P$ be a maximal parabolic subgroup of $G$ and set $\Omega = G/P$. Let  $P=QL$ be a Levi decomposition of $P$, and  assume $P$ is $G$-conjugate to its opposite parabolic $P^{-}=UL$. Then the generic $2$-point stabilizer in the action of $G$ on $\Omega$ is conjugate to $L$.
\end{lem}

\begin{proof}
First observe that $Q \cap P^{-}=1$, so $Q$ has a regular orbit on $\Omega$.  Since $\dim \Omega = \dim G/P = \dim Q$,  this orbit is open and dense in $\Omega$.   Moreover, this orbit is also $L$-invariant.
Therefore, if $P = G_{\a}$ then the $2$-point stabilizer $G_{\a,\b}$ is a conjugate of $L$ for any point $\b$ in the open $Q$-orbit. The result follows.
\end{proof}

Note that if $G$ is an exceptional group, the previous lemma applies unless $G=E_6$ and $P =P_1,P_3,P_5$ or $P_6$.

\begin{prop}\label{p:ep}
Suppose $G=E_6$ and $H=P_1$ or $P_6$. Then $b^0(G,H) = b(G,H) = b^1(G,H) = 6$.
\end{prop}

\begin{proof}
Since $P_1$ and $P_6$ are interchanged by an involutory graph automorphism of $G$, we may assume $H=P_1$. Here $\dim H = 62$, $\dim \Omega = 16$ and 
$$5 \leqs b^0(G,H) \leqs b^1(G,H) \leqs 6$$ 
(see Proposition \ref{p:parab}), so it remains to show that the generic $5$-point stabilizer is positive-dimensional. To do this, we may assume that $p>0$.

Let $q$ be a $p$-power. In the terminology of Cohen and Cooperstein \cite{CC}, the corresponding action of $E_6(q)$ is equivalent to the action on the subset of \emph{white points} in the standard $27$-dimensional $E_6(q)$-module. This transitive action has permutation rank $3$, and from the description of the suborbits (see \cite[(P.1), p.470]{CC}) we deduce that the generic $2$-point stabilizer for the original parabolic action of $G$ is of the form $UD_4T_2$, where $U$ is a $16$-dimensional unipotent subgroup. Moreover, $U$ is a vector space and $U\downarrow D_4 = U_1 \oplus U_2$, where $U_1$ and $U_2$ are distinct irreducible $8$-dimensional modules for $D_4$. It follows that $U$ has a $16$-dimensional regular orbit $\mathcal{O}$ on $\Omega = G/H$, whence $\mathcal{O}$ is open (and thus dense) in $\Omega$. In particular, we may identify $\mathcal{O}$ with $U$ and thus the generic $5$-point stabilizer of $G$ on $\Omega$ is the same as the generic $2$-point stabilizer of $D_4T_2$ on $U$.

Consider two generic points in $U = U_1 \oplus U_2$, say $u_1+u_2$ and $v_1+v_2$, where $u_i, v_i \in U_i$ and each $\la u_i,v_i\ra$ is a non-degenerate $2$-space. The $D_4$-stabilizer of these two generic points is the subgroup fixing each vector $u_1,v_1,u_2,v_2$, which is of the form $D_3 \cap D_3^g$ for some $g \in D_4$. Now $\dim D_4 = 28$ and $\dim D_3 = 15$, so $\dim (D_3 \cap D_3^g) \geqs 2$ and thus the generic $5$-point stabilizer of $G$ on $\Omega$ is at least $2$-dimensional. Therefore $b^0(G,H) \geqs 6$ and the result follows.
\end{proof}

\begin{prop}\label{p:ep2}
Suppose $(G,H)=(E_7,P_7)$. Then $b^0(G,H) = b(G,H) = b^1(G,H) = 6$. Moreover, the generic $5$-point stabilizer is $8$-dimensional.
\end{prop}

\begin{proof}
Here $\dim H = 106$, $\dim \Omega = 27$ and $5 \leqs b^0(G,H) \leqs b^1(G,H) \leqs 6$ (see Proposition \ref{p:parab}), so as in the proof of the previous proposition we need to show that $b^0(G,H)>5$. Let $H=QL$ be a Levi decomposition, so $L=E_6T_1$ and $Q$ is abelian. By Lemma \ref{l:parab}, we may assume that $L$ is the generic $2$-point stabilizer. Moreover, $Q$ has a regular open orbit $\mathcal{O}$ on $\Omega$, on which $L$ acts by conjugation, so it suffices to show that the generic $3$-point stabilizer in the conjugation action of $L$ on $Q$ is
positive-dimensional. We may assume $p>0$.

Let $q$ be a $p$-power. At the finite level, we may identify $Q$ with the standard $27$-dimensional $E_6(q)$-module. In the terminology of Cohen and Cooperstein \cite{CC}, the generic $E_6(q)$-orbit on this module coincides with the subset of \emph{black points}. This orbit has  point stabilizer $F_4(q)$, so we see that
$F_4$ is the generic $3$-point stabilizer in the original action of $G$. Further, by considering \cite[Table 2]{CC} we deduce that the generic $4$-point stabilizer has connected component $D_4$. Now
$$Q \downarrow D_4 = V_1 \oplus V_2 \oplus V_3 \oplus 0 \oplus 0 \oplus 0,$$
where $V_1,V_2,V_3$ are the distinct irreducible $8$-dimensional $D_4$-modules, and $0$ is the $1$-dimensional trivial $D_4$-module (see \cite[Proposition 2.3]{LSM}). A generic vector in $Q$ is of the form $v = v_1+v_2+v_3$, where each $v_i \in V_i$ spans a non-degenerate subspace, so the generic $5$-point stabilizer in $G$ is the intersection in $D_4$ of three conjugates of a subgroup $B_3<D_4$. Since $\dim D_4 = 28$ and $\dim B_3 =21$, it follows that the generic $5$-point stabilizer is at least $7$-dimensional, whence $b^0(G,H)>5$ as required.

Finally, let us show that the generic $5$-point stabilizer is $8$-dimensional. First observe that the intersection of two generic conjugates of $B_3<D_4$ is a subgroup $G_2<D_4$ (one way to see this is to consider the corresponding situation at the level of finite groups; see the proof of \cite[Proposition 3]{PLS}). Moreover, the intersection of $G_2$ with an additional generic conjugate of $B_3$ is isomorphic to $A_2$ (again, this follows from the proof of \cite[Proposition 3]{PLS}). The claim follows.  Indeed, this shows that the generic $5$-point stabilizer is precisely $A_2$.
\end{proof}

\begin{remk}\label{r:cam}
Recall that Theorem \ref{t:cam} states that there are infinitely many non-standard finite almost simple primitive permutation groups with base size $6$. This quickly follows from Proposition \ref{p:ep} above. Indeed, assume $p>0$, let $G=E_6$ and let $H$ be a $\s$-stable $P_1$ parabolic subgroup of $G$, where $G$ is defined over the algebraic closure $\bar{\mathbb{F}}_{p}$ and $\s$ is a Frobenius morphism of $G$ so that $G_{\s}$ has socle $E_6(q)$ for some $p$-power $q$. Now $b^0(G,H) = 6$ by Proposition \ref{p:ep}, so Proposition \ref{p:bb2}(ii) implies that $b(G_{\s},H_{\s}) \geqs 6$ for all $q>2$, while the main theorem of \cite{BLS} yields $b(G_{\s},H_{\s}) \leqs 6$. Therefore $b(G_{\s},H_{\s}) = 6$ for all $q>2$, and this establishes  Theorem \ref{t:cam}. In fact, by using a suitable permutation representation of $E_6(2)$, it is straightforward to show that $b(G_{\s},H_{\s})=6$ when $q=2$ (see \cite[Remark 1]{BLS}). Similarly, if $G=E_7$ and $H=P_7$ then Proposition \ref{p:ep2} implies that $b(G_{\s},H_{\s})=6$ for all $q$. (Since the generic $5$-point stabilizer in $G$ is $8$-dimensional, it is not a split torus and thus \cite[Proposition 8.1]{GG} implies that every $5$-point stabilizer in $G_{\s}$ is nontrivial when $q=2$.)
\end{remk}

In order to complete the proof of Theorem \ref{t:mainep1} we may assume that
$(G,H) = (E_8,P_7)$, $(E_7,P_3)$ or $(E_6,P_2)$.
In particular, note that Lemma \ref{l:parab} applies in each of these cases.

We need to introduce some new notation and terminology that we will use for the remainder of this section.
Fix a maximal torus $T$ of $G$, let $\Phi$ denote the root system of $G$, $\Delta = \{\a_1, \ldots, \a_r\}$ a set of simple roots (with the usual labelling), $\Phi^+$ the corresponding set of positive roots, and let $\{U_{\a} \mid \a \in \Phi\}$ be the root subgroups of $G$. Suppose $H=P_i$. Let $\Phi_{J}$ be the root system spanned by the simple roots $J=\Delta \setminus \{\a_i\}$ and set $\Phi_{J}^+ = \Phi_{J} \cap \Phi^+$. By replacing $H$ by a suitable conjugate, we may assume that $H=QL$ is a Levi decomposition of $H$, with Levi factor $L=\la T, U_{\pm \a} \mid \a \in J\ra$ and
unipotent radical $Q=\prod U_{-\b}$, the product taken over all $\b \in \Phi^+\setminus \Phi_{J}^+$.

Let $\b \in \Phi^+ \setminus \Phi_{J}^+$, say $\b=d_i\a_i+ \sum_{j\neq i}c_j\a_j$. Following \cite{ABS}, we define the \emph{level} and \emph{height} of $\b$ by
$$\mbox{level}(\b)=d_i,\;\; \mbox{height}(\b)=d_i+\sum_{j \neq i} c_j.$$
For each positive integer $j$ we define $Q_j = \prod U_{-\b}$, where the product is over the roots $\b \in \Phi^+ \setminus \Phi_{J}^+$ of level $j$. Finally, again following \cite{ABS}, we say that $G$ is \emph{special} if $(G,p) = (F_4,2)$, $(G_2,3)$ or $(G_2,2)$.

The next result is a special case of \cite[Theorem 2]{ABS}.

\begin{thm}\label{t:abs}
Let $G$ be a simple exceptional algebraic group and assume that $G$ is not special. Let $H=QL$ be a maximal parabolic subgroup of $G$, let $j \geqs 1$ be an integer and define $Q_j \leqs Q$ as above. Let $T_{L'}$ be a maximal torus of $L'$ contained in $T$. The following hold:
\begin{itemize}\addtolength{\itemsep}{0.2\baselineskip}
\item[{\rm (i)}] $Q_j$ is invariant under conjugation by $L$.
\item[{\rm (ii)}] $Q_j$ is an irreducible $KL'$-module with highest weight $-\b|_{T_{L'}}$, where $\b \in \Phi^+$ is the unique root of minimal height with {\rm level}$(\b)=j$.
\item[{\rm (iii)}] $L$ has an open dense orbit on $Q_j$.
\end{itemize}
\end{thm}

Let $(G,H)$ be one of the remaining cases that we have to consider and let $H=QL$ be a Levi decomposition. By Lemma \ref{l:parab}, we may assume that $L$ is the generic $2$-point stabilizer in the action of $G$ on the coset variety $\Omega=G/H$. Moreover, $Q$ has a regular dense orbit on $\Omega$ so we can reduce the problem to computing the base size for the action of $L$ on $Q$. As an $L$-variety,
$$Q \cong Q_1 \times Q_2 \times \cdots \times Q_m$$
where $m \geqs 1$ is the maximal level of a root $\b \in \Phi^+ \setminus \Phi_J^+$. In particular, the stabilizer in $L$ of a generic point in $Q$ is the intersection of the generic stabilizers of $L$ on each $Q_j$.

The derived subgroup $L'$ is a product of simple groups $L_1, \ldots, L_k$ for some $k \geqs 1$. By Theorem \ref{t:abs}(ii), $Q_j$ is an irreducible $KL'$-module, so we can write $Q_j \cong L(\mu_1) \otimes \cdots \otimes L(\mu_k)$ as $KL'$-modules, where $L(\mu_i)$ denotes the irreducible $KL_i$-module with highest weight $\mu_i$. For each factor $L_i$ we express $\mu_i$ in terms of a set of fundamental dominant weights $\{\l_1, \l_2, \ldots\}$ (with respect to the usual ordering), unless $L_i=A_1$ when we will write $L(m)$ rather than $L(m\l_1)$. We write $0$ for the trivial $1$-dimensional $KL_i$-module. Finally, let $\{\omega_1, \ldots, \omega_r\}$ be a set of fundamental dominant weights for $G$.

\begin{prop}\label{p:ep3}
Suppose $(G,H) = (E_8,P_7)$ or $(E_7,P_3)$. Then $b^0(G,H) = b(G,H) = b^1(G,H) = 4$.
\end{prop}

\begin{proof}
According to Proposition \ref{p:parab}, in both of these cases we have $b^1(G,H) \leqs 4$, so it suffices to show that the intersection of three generic conjugates of $H$ is positive-dimensional. First consider the case $(G,H) = (E_8,P_7)$. By Lemma \ref{l:parab}, the generic $2$-point stabilizer is a Levi subgroup $L=E_6A_1T_1$ of $H$ and so by the above discussion it suffices to show that $\dim C_L(q) >0$ for a generic element $q \in Q$, where $Q$ is the unipotent radical of $H$. Let $j$ be a positive integer and define $Q_j$ as above. By applying Theorem \ref{t:abs} we deduce that each $Q_j$ is an irreducible $KL'$-module with
$$Q_1 \cong L(\l_6) \otimes L(1), \; Q_2 \cong L(\l_1) \otimes 0, \; Q_3 \cong 0 \otimes L(1)$$
as $KL'$-modules. For example, $\b=\a_7$ is clearly the unique root of minimal height at level $1$, so $Q_1$ is an irreducible $KL'$-module with highest weight $-\a_7 = \omega_6-2\omega_7+\omega_8$ (restricted to a suitable maximal torus of $L'=E_6A_1$), whence $Q_1 \cong L(\l_6) \otimes L(1)$ as claimed.

A generic point in $Q_1 \times Q_2 \times Q_3$ has the form
$q=(a_1 \otimes b_1 + a_2 \otimes b_2, c, d)$, where $a_1,a_2 \in L(\l_6)$, $b_1,b_2,d \in L(1)$ and $c \in L(\l_1)$. As in the proof of Proposition \ref{p:ep2}, we see that $F_4$ is the generic stabilizer in the action of $E_6$ on the $27$-dimensional modules $L(\l_1)$ and $L(\l_6)$, so $C_{E_6}(q)$ is the intersection of three conjugates of $F_4$ in $E_6$. By Theorem \ref{t:mainep2}(ii) (see the proof of Lemma \ref{l:e6ai}), the intersection of any three conjugates of $F_4$ in $E_6$ is positive-dimensional, so $\dim C_L(q) \geqs \dim C_{E_6}(q)>0$ and thus $b^0(G,H) = b(G,H) = b^1(G,H) = 4$ as required.

The case $(G,H) = (E_7,P_3)$ is similar. Here the generic $2$-point stabilizer is $L = A_1A_5T_1$ and once again it suffices to show that $\dim C_L(q) >0$ for a generic $q \in Q$. In this case, using Theorem \ref{t:abs}, we calculate that
$$Q_1 \cong L(1) \otimes L(\l_2), \; Q_2 \cong 0 \otimes L(\l_4), \; Q_3 \cong L(1) \otimes 0$$
as $KL'$-modules, and a generic point $q \in Q_1 \times Q_2 \times Q_3$ has the form
$q=(a_1 \otimes b_1 + a_2 \otimes b_2, c, d)$, where $a_1,a_2,d \in L(1)$, $b_1,b_2 \in L(\l_2)$ and $c \in L(\l_4)$. The generic stabilizer in $A_5$ with respect to the $15$-dimensional modules $L(\l_2)$ and $L(\l_4) \cong L(\l_2)^*$ is $C_3$ (note that $L(\l_2)= \L^2(W)$, where $W$ is the natural $A_5$-module, so $L(\l_2)$ and $L(\l_4)$ can be identified with the space of alternating forms on $W$). Therefore $C_{A_5}(q)$ is the intersection of three conjugates of $C_3$ in $A_5$, which is positive-dimensional by Theorem \ref{t:cmain}(ii) (see \cite[Theorem 1.1]{GG}). The desired conclusion follows as before.
\end{proof}

\begin{prop}\label{p:ep4}
If $(G,H) = (E_6,P_2)$ then $b^0(G,H) = b(G,H) = b^1(G,H) = 5$.
\end{prop}

\begin{proof}
By Proposition \ref{p:parab} we have $b^1(G,H) \leqs 5$, so it suffices to show that the intersection of
four generic conjugates of $H$ in $G$ is positive-dimensional. By Lemma \ref{l:parab},
the generic $2$-point stabilizer is $L = A_5T_1$. By applying Theorem \ref{t:abs} we deduce
that $Q \cong Q_1 \times Q_2$ (as an $L$-variety), where $Q_1 \cong L(\l_3)$ and
$Q_2 \cong 0$ as $KL'$-modules (note that $L(\l_3) = \L^3(W)$, where $W$ is the natural module for $A_5$).
If $v \in Q_1$ is generic then $C_{A_5}(v)$ is a $\C_2$-subgroup of type ${\rm GL}_{3} \wr S_2$. By Theorem \ref{t:cmain} (see the proof of Lemma \ref{p:aic}),
the intersection of any two such centralizers in $A_5$ is positive-dimensional, so the generic $2$-point stabilizer of $L$ on $Q$
is also positive-dimensional, whence $b^0(G,H)>4$ as required.
\end{proof}

\vs

This completes the proof of Theorem \ref{t:mainep1}.

\vs

\begin{remk}
A similar approach can also be used to investigate the remaining cases where $G = F_4,G_2$ or $(G,H) = (E_6,P_4)$. However, the analysis here is more complicated and we do not get better results than the bounds provided in Proposition \ref{p:parab}. If $G$ is \emph{special}, that is, if $(G,p) = (F_4,2),(G_2,3)$ or $(G_2,2)$, then in these cases we can calculate $b^0(G,H)$ via Proposition \ref{p:bb2}(ii) and a suitable computation with $F_4(4)$, $G_2(3)$ and $G_2(4)$, using {\sc Magma}
\cite{Magma}. We find that $b^0(G,H)=4$ if $(G,p) = (F_4,2)$ and $H=P_1$ or $P_4$, otherwise $b^0(G,H)=3$.
\end{remk}

\subsection{Non-parabolic actions}

In this section we complete the proof of Theorem \ref{t:mainep2} on non-parabolic actions of exceptional groups. By Theorem \ref{e:max}, one of the following holds:
\begin{itemize}\addtolength{\itemsep}{0.2\baselineskip}
\item[{\rm (i)}] $G=E_7$,  $H=(2^2 \times D_4).S_3$ and $p \neq 2$;
\item[{\rm (ii)}] $G=E_8$, $H=A_1 \times S_5$ and $p \neq 2,3,5$;
\item[{\rm (iii)}] $H=N_{G}(X)$, with $X$ given in Table \ref{t:max}.
\end{itemize}

We adopt the notation introduced earlier (see the discussion preceding the statement of Lemma \ref{l:e8ai}). In particular, $\mathrm{Lie}(G)$ is the Lie algebra of $G$ and
$C_{\mathrm{Lie}(G)}(x)$ denotes the fixed point space of $x \in G$ on $\mathrm{Lie}(G)$, with respect to the adjoint representation. Note that
$$\dim C_G(x)  \leqs \dim C_{\mathrm{Lie}(G)}(x)$$
for all $x \in G$, with equality if $x$ is semisimple (see \cite[Section 1.10]{HCC}, for example). Given a simple algebraic group $X$, we will write $W(\l)$ for the Weyl module for $X$ with highest weight $\l$, and we will express $\l$ in terms of a set of fundamental dominant weights $\{\l_1, \l_2, \ldots\}$ for $X$ (unless $X=A_1$, when we write $W(m)$ rather than $W(m\l_1)$). We denote the trivial $1$-dimensional $KX$-module by $0$ and we will write $\mathcal{P}$ for the set of elements in $H$ of prime order (including all nontrivial unipotent elements if $p=0$). We will use the Aschbacher-Seitz \cite{AS} notation for involutions in classical groups when $p=2$.

\begin{prop}
Theorem \ref{t:mainep2} holds for $G=E_8$.
\end{prop}

\begin{proof}
If $x \in G$ is nontrivial then $\dim x^{G}\geqs 58$ (minimal if $x$ is a long root element), hence Corollary \ref{c:con} immediately implies that
\begin{equation}\label{e:all22}
b^0(G,H)=b(G,H)=b^1(G,H)=2
\end{equation}
if $\dim H<29$. For the remainder, let us assume $\dim H \geqs 29$.

By Theorem \ref{e:max}, we have $H=N_G(X)$ with $X$ given in Table \ref{t:max}. First assume $H$ is not a maximal rank subgroup of $G$, so $H^0=A_1G_2^2$ ($p \neq 2$) or $G_2F_4$ since $\dim H \geqs 29$. In both cases we claim that
\begin{equation}\label{e:bb22}
\dim x^H < \frac{1}{2}\dim x^G
\end{equation}
for all $x \in \mathcal{P}$, so \eqref{e:all22} follows from Corollary \ref{c:con}. This is clear if $H^0=A_1G_2^2$ since $\dim x^H \leqs 26$ for all $x \in H$ (note that if $x \in H \setminus H^0$ has prime order then $x$ is a semisimple involution and thus $\dim x^G \geqs 112$; see \cite[Table 4.3.1]{GLS}, for example).

Next assume $H=H^0=G_2F_4$. Here $\dim x^H \leqs 60$ for all $x \in H$, so we may assume that $\dim x^G \leqs 120$. Suppose $x$ is unipotent, so the bound on $\dim x^G$ implies that $x$ belongs to one of the $G$-classes labelled $A_1, 2A_1, 3A_1$ or $A_2$ (see \cite[Table 2]{LLS}). The fusion of unipotent classes in $H$ is described in \cite[Table 38]{Lawunip} and we
quickly deduce that \eqref{e:all22} holds.

Now assume $x$ is semisimple and $\dim x^G \leqs 120$, so $C_G(x)=A_1E_7$ or $E_7T_1$. If $C_G(x)=A_1E_7$ then $p \neq 2$ and $x$ is an involution, so $\dim x^H \leqs 36$ since there is a unique class of involutions in $G_2$ (of dimension $8$), and exactly two such classes in $F_4$ (dimensions $28$ and $16$). Finally, assume $C_G(x)=E_7T_1$, so $\dim x^G = 114$ and $\dim C_{\mathrm{Lie}(G)}(x)=134$. Now
$$\mathrm{Lie}(G) \downarrow G_2F_4 = \mathrm{Lie}(G_2F_4) \oplus (W(\l_1) \otimes W(\l_4))$$
(see \cite[Proposition 2.4]{LSM}).
If $\dim x^H \geqs 58$ then $C_{H}(x)=T_6$ or $A_1T_5$, and from the above description of $\mathrm{Lie}(G) \downarrow G_2F_4$ it is straightforward to see that $\dim C_{\mathrm{Lie}(G)}(x)<134$, which is a contradiction. For example, suppose $x=x_1x_2$ and $C_{H}(x)=T_6$. Up to conjugacy,  $x_1$ acts on $W(\l_1)$ as a diagonal matrix $[I_3,\l I_2,\l^{-1}I_2]$ for some $\l \in K^*$ with $\l \neq \pm 1$, so \cite[Lemma 3.7]{LSh2} implies that $\dim C_{W(\l_1) \otimes W(\l_4)}(x) \leqs 78$ and thus $\dim C_{\mathrm{Lie}(G)}(x) \leqs 84$.
This establishes \eqref{e:bb22} and we conclude that \eqref{e:all22} holds (see Corollary \ref{c:con}).

For the remainder we may assume $H=N_G(X)$ is a maximal rank subgroup with
$$X \in \{A_2^4, D_4^2, A_4^2, A_2E_6, A_8, A_1E_7, D_8\}$$
(see Table \ref{t:max}). The cases $H^0=D_8$ and $A_1E_7$ were handled in Lemma \ref{l:e8ai}. In each of the remaining cases we claim that \eqref{e:bb22} holds for all $x \in \mathcal{P}$, in which case Corollary \ref{c:con} implies that \eqref{e:all22} holds.

First assume $H^0=A_2^4$, so $H/H^0 = {\rm GL}_{2}(3)$ and $\dim H = 32$. If $x \in G$ is a long root element then $x \in H^0$ (see \cite[Proposition 1.13(iii)]{LLS}), so $\dim x^H \leqs 24$, $\dim x^G=58$ and the required bound follows. On the other hand, if $x$ is not a long root element then $\dim x^G \geqs 92$ and again the claim holds.

Next suppose $H^0=D_4^2$. Here $H/H^0  = Z_2 \times S_3$, where $Z_2$ swaps the two factors, and $S_3$ induces graph automorphisms (simultaneously on the two $D_4$ factors). Now $\dim x^H \leqs 48$ for all $x \in H$, so we may assume $x$ is a unipotent element in one of the $G$-classes labelled $A_1$ or $2A_1$ (with respective dimensions $58$ and $92$). By \cite[Proposition 2.1]{LSM} we have
\begin{align*}
\mathrm{Lie}(G) \downarrow D_4D_4 = & \; \mathrm{Lie}(D_4D_4) \oplus (W(\l_1) \otimes W(\l_1)) \oplus (W(\l_3) \otimes W(\l_3)) \\
& \; \oplus (W(\l_4) \otimes W(\l_4)),
\end{align*}
where $W(\l_1)$ is the natural $D_4$-module, and $W(\l_3), W(\l_4)$ are the two distinct irreducible spin modules for $D_4$.

First assume $p \neq 2$.
We claim that $\dim (x^G \cap H)=10$ if $x \in A_1$, and $\dim (x^G \cap H) = 20$ if $x \in 2A_1$. To see this, let $u,v \in D_4$ be elements with respective Jordan forms $[J_2^2,J_1^4]$ and $[J_3,J_1^5]$ on the natural module $W(\l_1)$. We calculate that
$[J_3,J_2^8,J_1^9]$ and $[J_3^6,J_1^{10}]$ are the respective Jordan forms of $u$ and $v$ on $\mathrm{Lie}(D_4)$. In addition, we note that $u$ has Jordan form $[J_2^2,J_1^4]$ on both $W(\l_3)$ and $W(\l_4)$, and $v$ has Jordan form $[J_2^4]$ on these modules. Using the above decomposition for $\mathrm{Lie}(G) \downarrow D_4D_4$ we can calculate the Jordan form of
$(u,1),(v,1),(u,u) \in H^0$ on $\mathrm{Lie}(G)$, and then use \cite[Table 9]{lawthercom} to determine the $G$-class of these elements. In this way, we deduce that $(u,1) \in A_1$ and $(v,1),(u,u) \in 2A_1$. Moreover, one can check that these elements represent the only $H$-classes that are in $A_1$ and $2A_1$. (For example, we find that $(u,v) \in 3A_1$ and $(v,v) \in A_2$. Also, if $p=3$ and $x \in H \setminus H^0$ induces a triality automorphism on each $D_4$ factor then $x$ cyclically permutes the modules $W(\l_1),W(\l_3)$ and $W(\l_4)$, so the Jordan form of $x$ on $\mathrm{Lie}(G)$ has at least $64$ Jordan blocks of size $3$ and thus  $x$ is not in $A_1$ nor $2A_1$.) This justifies the claim.

Similarly, if $p=2$ then careful calculation reveals that $x^G\cap H$ is a union of two $H$-classes when $x \in A_1$, with representatives $(a_2,1), (b_1,b_1) \in H^0$ (in the notation of \cite{AS}), whence $\dim (x^G\cap H) = 14$. Similarly, if $x \in 2A_1$ then $x^G\cap H$ comprises two $H$-classes, with representatives $(a_2,a_2)$ and $(c_2,1)$, so
$\dim (x^G\cap H)=20$ as before. (Note that if $x \in H \setminus H^0$ interchanges the two $D_4$ factors then the Jordan form of $x$ on $\mathrm{Lie}(G)$ has at least $96$ Jordan blocks of size $2$, so $x$ is not in $A_1$ nor $2A_1$.) We conclude that \eqref{e:bb22} holds if $H^0=D_4^2$.

Next consider the case $H^0=A_4^2$. Here $H/H^0=Z_4$ and $\dim x^H \leqs 40$ for all $x \in H$, so we may assume $x \in G$ is a long root element. In particular, $x \in H^0$ (see \cite[Proposition 1.13(iii)]{LLS}) and by inspecting \cite[Table 26]{Lawunip} we deduce that $\dim x^H \leqs 8$.

It remains to deal with the cases $H^0 = A_2E_6$ and $H^0=A_8$. First suppose $H^0=A_2E_6$. Here $H/H^0=Z_2$ and $\dim x^H \leqs 78$ for all $x \in H$, so we may assume $\dim x^G\leqs 156$. In particular, if $x$ is semisimple then $C_G(x) = E_7A_1$, $E_7T_1$, $D_8$ or $D_7T_1$ (see \cite{FJ2}), and by applying \cite[Theorem 2]{LLS} we deduce that \eqref{e:bb22} holds. Now assume $x$ is unipotent. If $x \in H^0$ then the fusion information in \cite[Table 24]{Lawunip} is sufficient, so let us assume $p=2$ and $x \in H \setminus H^0$.
There are two $H$-classes of involutions in $H \setminus H^0$, represented by $x_1$ and $x_2$ say, where $C_{H^0}(x_1) = A_1F_4$ and $C_{H^0}(x_2) = A_1C_{F_4}(t)$, where $t \in F_4$ is a long root element (each $x_i$ acts as a graph automorphism on the $A_2$ and $E_6$ factors). By \cite[Proposition 2.1]{LSM} we have
$$\mathrm{Lie}(G)\downarrow A_2E_6 = \mathrm{Lie}(A_2E_6) \oplus (W(\l_1) \otimes W(\l_6)) \oplus (W(\l_2) \otimes W(\l_1))$$
and using this we calculate that the Jordan form of $x_1$ and $x_2$ on $\mathrm{Lie}(G)$ is $[J_{2}^{110},J_1^{28}]$ and $[J_{2}^{120},J_1^{8}]$, respectively. Therefore, by inspecting \cite[Table 9]{lawthercom}, we see that $x_1 \in 3A_1$ and $x_2 \in 4A_1$, so $\dim x_1^H = 31$, $\dim x_1^G = 112$ and $\dim x_2^H = 47$, $\dim x_2^G = 128$.

Finally suppose $H^0=A_8$, so $H/H^0=Z_2$ and we may assume $\dim x^G \leqs 144$ since
$\dim x^H \leqs 72$ for all $x \in H$. If $x$ is semisimple then $C_G(x)=E_7A_1$, $E_7T_1$ or $D_8$, and it is easy to check that \cite[Theorem 2]{LLS} is sufficient. Similarly, if $x \in H^0$ is unipotent then the desired bound follows from the information in \cite[Table 25]{Lawunip}. Finally, suppose $p=2$ and $x \in H \setminus H^0$ is an involution (so $x$ is a graph automorphism of $A_8$). Now
$$\mathrm{Lie}(G)\downarrow A_8 = \mathrm{Lie}(A_8) \oplus W(\l_3) \oplus W(\l_6)$$
(see \cite[Proposition 2.1]{LSM}) and we deduce that $x$ has Jordan form $[J_2^{120},J_1^{8}]$ on
$\mathrm{Lie}(G)$. Therefore \cite[Table 9]{lawthercom} indicates that $x$ is in the $G$-class labelled $4A_1$, so $\dim x^H = 44$ and $\dim x^G = 128$.
\end{proof}

\begin{prop}
Theorem \ref{t:mainep2} holds for $G=E_7$.
\end{prop}

\begin{proof}
If $x \in G$ is nontrivial then $\dim x^{G}\geqs 34$ (minimal if $x$ is a long root element), so we may as well assume $\dim H \geqs 17$. According to Theorem \ref{e:max}, one of the following holds:
\begin{itemize}\addtolength{\itemsep}{0.2\baselineskip}
\item[(i)] $p \neq 2$ and $H=(2^2 \times D_4).S_3$; or
\item[(ii)] $H=N_G(X)$ with
$X \in \{A_1G_2,A_1F_4,G_2C_3, T_1E_6,A_1D_6,A_7, A_2A_5, A_1^3D_4,A_1^7\}$.
\end{itemize}

If $H^0=A_7$, $A_1D_6$ or $T_1E_6$ then $H$ is an involution-type subgroup and we refer the reader to Lemma \ref{l:e7ai}. In each of the remaining cases we claim that \eqref{e:bb22} holds for all $x \in \mathcal{P}$, so 
$$b^0(G,H)=b(G,H)=b^1(G,H)=2.$$ 
Let $V_{56}$ be the $56$-dimensional irreducible $KG$-module.

First suppose $H=(2^2 \times D_4).S_3$. Here $p \neq 2$ and $\dim x^H \leqs 24$ for all $x \in \mathcal{P}$. In particular, if $x$ is not a long root element then $\dim x^G \geqs 52$ and thus
$\dim x^H < \frac{1}{2}\dim x^G$, so it suffices to show that there are no long root elements in $H$. To see this, first observe that $H^0=D_4$ belongs to an $A_7$ subgroup of $G$ (embedded via the natural $8$-dimensional module for $D_4$), and the root subgroups of this $A_7$ are also root subgroups of $G$. Since ${\rm SO}_{8}$ does not contain any transvections when $p \neq 2$, it follows that there are no long root elements of $G$ in $H^0$. By \cite[Proposition 1.13(iii)]{LLS}, there are also no long root elements in $H \setminus H^0$. Therefore, \eqref{e:bb22} holds for all $x \in \mathcal{P}$.

If $H=H^0=A_1G_2$ then $\dim x^H \leqs 14$ for all $x \in \mathcal{P}$, and the claim follows since $\dim x^G \geqs 34$. Next suppose $H^0=A_1^7$, so $H/H^0={\rm GL}_{3}(2)$. If $x \in H^0$ then $\dim x^H \leqs 14$ and the result follows. On the other hand, if $x \in H \setminus H^0$ then $x$ is not a long root element (see \cite[Proposition 1.13(iii)]{LLS}), so $\dim x^G \geqs 52>2\dim H$.

Now consider the case $H=H^0=A_1F_4$. Here we may assume $\dim x^G \leqs 100$ since $\dim x^H \leqs 50$ for all $x \in H$. The fusion of unipotent elements in $H$ can be read off from the information in \cite[Section 5.12]{Lawunip} and we quickly deduce that \eqref{e:bb22} holds for all unipotent elements $x \in \mathcal{P}$. Now assume $x \in \mathcal{P}$ is semisimple. The possibilities for $D=C_{G}(x)$ with $\dim x^G \leqs 100$ are listed in Table \ref{t:100} (see \cite{FJ1}, for example).

\begin{table}[h]
$$\begin{array}{llllllll} \hline
D^0 & \dim x^G & & & & & & \\ \hline
A_3^2A_1 & 100 & & A_5A_1T_1 & 94 \hspace{5mm} & & A_7 & 70 \\
D_4A_1T_2 & 100 & & A_5A_2 & 90 & & D_6T_1 & 66 \\
A_4A_2T_1 & 100 & & D_5T_2 & 86 & & D_6A_1 & 64 \\
D_4A_1^2T_1 & 98 & & D_5A_1T_1 & 84 & & T_1E_6 & 54 \\
A_5T_2 & 96 & & A_6T_1 & 84 & & & \\ \hline
\end{array}$$
\caption{$D=C_G(x)$, $x$ semisimple, $\dim x^G \leqs 100$}
\label{t:100}
\end{table}

First assume $p \neq 2$ and $x$ is an involution, so $D^0 = D_6A_1$, $A_7$ or $T_1E_6$ (see \cite[Table 4.3.1]{GLS}). The largest class of involutions in $F_4$ has dimension $28$, so $\dim x^H \leqs 30$ and thus we may assume $D^0=T_1E_6$, whence $\dim x^G = 54$. Write $x=x_1x_2$, where $x_1 \in A_1$ and $x_2 \in F_4$. If $C_{F_{4}}(x_2) \neq A_1C_3$ then $\dim x^H \leqs 18$, so let us assume $C_{F_{4}}(x_2) =A_1C_3$. According to \cite[Proposition 2.4]{LSM} we have
\begin{equation}\label{e:a1f4e}
\mathrm{Lie}(G) \downarrow A_1F_4 = \mathrm{Lie}(A_1F_4) \oplus (W(2\l_1) \otimes W(\l_4)),
\end{equation}
and we note that $\dim C_{\mathrm{Lie}(G)}(x) = \dim C_{G}(x) = 79$. However, $\dim C_{W(\l_4)}(x_2)=14$ and we deduce that $\dim C_{\mathrm{Lie}(G)}(x) =69$ if $x_1=1$, and
$\dim C_{\mathrm{Lie}(G)}(x) =63$ if $x_1 \ne 1$. This is a contradiction and
thus \eqref{e:bb22} holds for all involutions.

For the remainder, we may assume that $x \in H$ has odd prime order. Suppose $\dim x^G = 98$ or $100$, so $\dim C_{\mathrm{Lie}(G)}(x) = 33$ or $35$. Write $x=x_1x_2$ as before. We may assume that $\dim x^H = 50$, so $x_1$ and $x_2$ are both regular. In particular, since $x_1$ is regular, \cite[Lemma 3.7]{LSh2} implies that $\dim C_{W(2\l_1) \otimes W(\l_4)}(x) \leqs 26$, whence $\dim C_{\mathrm{Lie}(G)}(x) \leqs 5+26$, a contradiction.

Now assume $D^0=A_5T_2$, so $\dim x^G = 96$, $\dim C_{\mathrm{Lie}(G)}(x) = 37$ and we reduce to the case $\dim x_2^{F_4} = 46$ or $48$. If $x_1 \neq 1$ then $\dim C_{W(2\l_1) \otimes W(\l_4)}(x) \leqs 26$ as above, so $\dim C_{\mathrm{Lie}(G)}(x) \leqs 7+26$ and we reach  a contradiction. Now suppose $x_1=1$, so $x_2$ is regular and $\dim C_{\mathrm{Lie}(G)}(x) =7+3\a$, where $\a=\dim C_{W(\l_4)}(x_2)$. We claim that $\a \leqs 8$.

First observe that we may assume $x_2 \in D_4<D_5<E_6$. Let $\{\omega_1, \ldots, \omega_6\}$ and $\{\xi_1, \ldots, \xi_5\}$ be fundamental dominant weights for $E_6$ and $D_5$, respectively, and let $V_{27}$ be the $27$-dimensional irreducible module for $E_6$ with highest weight $\omega_1$. By \cite[Table 8.7]{LSM} we have
\begin{equation}\label{v27}
V_{27} \downarrow D_5 = W(\xi_1) \oplus W(\xi_4) \oplus 0,
\end{equation}
where $W(\xi_1)$ is the natural module for $D_5$, and $W(\xi_4)$ is one of the irreducible spin modules. The $26$-dimensional $F_4$-module $W(\l_4)$ is a section of $V_{27}$. Since $x_2 \in D_4$ is regular it follows that $\dim C_{W(\xi_1)}(x_2) \leqs 4$. Now the restriction of the $D_5$ spin module $W(\xi_4)$ to $D_4$ is a sum of two non-isomorphic spin modules for $D_4$. Therefore the regularity of $x_2$ implies that $\dim C_{W(\xi_4)}(x_2) \leqs 2+2$ and thus $\a \leqs 4+2+2 =8$ as claimed. In particular, $\dim C_{\mathrm{Lie}(G)}(x) \leqs 7+24 = 31$, which is a contradiction. The case $D^0 = A_5A_1T_1$ is entirely similar. If $D^0=A_5A_2$ then $x$ has order $3$ and we deduce that $\dim x^H \leqs 2+36=38<\frac{1}{2}\dim x^G$ (see \cite[Table 4.7.1]{GLS}).

Next consider the case $D^0=D_5T_2$, so $\dim x^G = 86$ and we may assume $\dim x^H \geqs 44$. As before, if $x_1 \neq 1$ then \eqref{e:a1f4e} implies that $\dim C_{\mathrm{Lie}(G)}(x) \leqs 11+26$, which is a contradiction. Now suppose $x_1=1$, so $\dim C_{F_4}(x_2) = 4$, $6$ or $8$, and
\begin{equation}\label{e:vl4}
\dim C_{\mathrm{Lie}(G)}(x) = 3+\dim C_{F_4}(x_2)+3\a
\end{equation}
with $\a=\dim C_{W(\l_4)}(x_2)$. Since $\dim C_{\mathrm{Lie}(G)}(x) =47$, we reduce to the case $\dim C_{F_4}(x_2) = 8$, so $C_{F_4}(x_2)^0 = T_2A_1^2$ is the only possibility. Here $x_2 \in D_4 < D_5$ and using \eqref{v27} we calculate that $\a \leqs 10$. For example, if $C_{D_4}(x_2) = {\rm GL}_{2} \times {\rm GL}_{2}$ then $\dim C_{W(\xi_1)}(x_2) = 2$ and the proof of \cite[Lemma 7.4]{Bur2} yields
$\dim C_{W(\xi_4)}(x_2) \leqs 8$. Therefore $\dim C_{\mathrm{Lie}(G)}(x) \leqs 41 < 47$. We conclude that \eqref{e:bb22} holds when $D^0=D_5T_2$. A similar argument applies when $D^0=D_5A_1T_1$ or $A_6T_1$.

To complete the analysis of the case $H=A_1F_4$ we may assume $D^0=D_6T_1$ or $E_6T_1$ (in the latter case we may also assume $x$ has odd order). Suppose $D^0=D_6T_1$, so $\dim x^G = 66$ and we may assume $\dim x^H \geqs 34$. If $x=x_1x_2$ and $x_1$ is nontrivial then the usual argument implies that $\dim C_{\mathrm{Lie}(G)}(x) \leqs 21+26$, which is a contradiction. Now assume $x_1=1$ and note that \eqref{e:vl4} holds. By arguing as in the proof of \cite[Lemma 7.4]{Bur2} we calculate that $\a \leqs 14$. However, $\dim C_{F_4}(x_2) \leqs 18$ since $x_1=1$ and $\dim x^H \geqs 34$, so $3+\dim C_{F_4}(x_2)+3\a \leqs 63$, which contradicts \eqref{e:vl4} since $\dim C_{\mathrm{Lie}(G)}(x) = 67$. An entirely similar argument applies if $D^0=E_6T_1$. We conclude that \eqref{e:all22} holds when $H=A_1F_4$.

Now suppose $H=H^0=G_2C_3$. First note that we may assume $\dim x^G \leqs 60$. If $x$ is unipotent then the relevant classes are labelled $A_1$, $2A_1$ and $(3A_1)''$, with respective dimensions $34$, $52$ and $54$. As explained in \cite[Section 5.12]{Lawunip}, complete information on the fusion of unipotent classes can be deduced from \cite[Table 38]{Lawunip}, and it is straightforward to check that \eqref{e:bb22} holds.

Now assume $x \in H$ is semisimple. Here the hypothesis $\dim x^G \leqs 60$ implies that $C_G(x)=E_6T_1$, hence $\dim x^G = 54$ and we may assume $\dim x^H \geqs 28$, so $C_{H}(x)=T_5$ or $A_1T_4$, and thus $x$ has odd prime order (if $x \in H$ is an involution then $\dim x^H \leqs 20$). By \cite[Proposition 2.4]{LSM} we have
$$\mathrm{Lie}(G) \downarrow G_2C_3 = \mathrm{Lie}(G_2C_3) \oplus (W(\l_1) \otimes W(\l_2))$$
and we note that $\dim C_{\mathrm{Lie}(G)}(x) = 79$. Write $x=x_1x_2$, where $x_1 \in G_2$ and $x_2 \in C_3$. Let $s$ denote the codimension of the largest eigenspace of $x_1$ on $W(\l_1)$. Since $C_{H}(x)=T_5$ or $A_1T_4$, we calculate that $s \geqs 2$ and thus
\cite[Lemma 3.7]{LSh2} implies that the codimension of the largest eigenspace of $x$ on
$W(\l_1) \otimes W(\l_2)$ is at least $28$. Therefore, $\dim C_{\mathrm{Lie}(G)}(x) \leqs 7+70$, a contradiction.

Next suppose $H^0=A_1^3D_4$. Here $H/H^0=S_3$ and $\dim x^H \leqs 30$ for all $x \in H$, so we may assume $\dim x^G \leqs 60$. If $x$ is semisimple then $C_G(x)=E_6T_1$ and \cite[Theorem 2]{LLS} implies that $\dim x^H \leqs 18 < \frac{1}{2}\dim x^G$. Now assume $x$ is unipotent. The relevant $G$-classes are labelled $A_1$, $2A_1$ and $(3A_1)''$, and we calculate that
\begin{align}\label{e:ppp}
V_{56} \downarrow A_1^3D_4 = & \; (W(1)\otimes 0 \otimes 0 \otimes W(\l_1)) \oplus (0\otimes W(1) \otimes 0 \otimes W(\l_3))  \\
& \oplus (0\otimes 0 \otimes W(1) \otimes W(\l_4)) \oplus (W(1)\otimes W(1) \otimes W(1) \otimes 0). \nonumber
\end{align}
If $p=2$ then $x$ is an involution, so $\dim x^H \leqs 22$ and therefore we may assume $x$ is a long root element. In particular, $x \in H^0$ (see \cite[Proposition 1.13(iii)]{LLS}). According to \cite[Table 7]{lawthercom}, a long root element has Jordan form $[J_{2}^{12},J_1^{32}]$ on $V_{56}$, and by considering the above decomposition \eqref{e:ppp} we deduce that $x$ is $H$-conjugate to $x_1x_2x_3x_4 \in A_1^3D_4$, where
\begin{itemize}\addtolength{\itemsep}{0.2\baselineskip}
\item[(i)] $x_1=J_2$ and $x_{i}=1$ for all $i \geqs 2$; or
\item[(ii)] $x_4=a_2$ and $x_{i}=1$ for all $i \leqs 3$.
\end{itemize}
Therefore, $\dim x^H \leqs 10$ and the result follows. Finally, suppose $p \neq 2$ and $x$ is  unipotent. If $p=3$ and $x \in H \setminus H^0$ then we calculate that $x$ has Jordan form $[J_3^{18},J_1^2]$ on $V_{56}$, so $x$ belongs to one of the classes labelled $2A_2$ or $2A_2+A_1$ (see \cite[Table 7]{lawthercom}). In particular, $\dim x^H < \frac{1}{2}\dim x^G$ as required. Similarly, if $x \in H^0$ then we can determine the Jordan form of $x$ on $V_{56}$; in this way, the reader can check that if $x$ is in one of the relevant classes $A_1, 2A_1$ or $(3A_1)''$ then $\dim x^H \leqs 14<\frac{1}{2}\dim x^G$.

Finally, let us consider the case $H^0=A_2A_5$. Here $H/H^0 = Z_2$ and $\dim x^H \leqs 36$ for all $x \in H$, so we may assume $\dim x^G \leqs 72$. In particular, if $x$ is semisimple then $C_{G}(x)^0=D_6A_1$, $A_7$, $E_6T_1$ or $D_6T_1$, and the bound provided by
\cite[Theorem 2]{LLS} is sufficient. For example, if $C_G(x)^0$ has a $D_6$ factor then \cite[Theorem 2]{LLS} yields $\dim x^G - \dim x^H \geqs 39$, so $\dim x^H \leqs 27<\frac{1}{2}\dim x^G$. Now assume $x$ is unipotent. Here the relevant classes are labelled $A_1$, $2A_1$, $(3A_1)''$, $(3A_1)'$, $A_2$ and $4A_1$. If $x \in H \setminus H^0$ then $p=2$ and $x$ acts as a graph automorphism on the $A_2$ and $A_5$ factors of $H^0$, so $\dim x^H \leqs 5+20=25$. In addition, \cite[Proposition 1.13(iii)]{LLS} implies that $x$ is not a long root element, so $\dim x^G \geqs 52$ and the desired bound follows. Finally, if $x \in H^0$ is unipotent then the $G$-class of $x$ is given in \cite[Table 21]{Lawunip} and the result quickly follows.
\end{proof}

\begin{prop}
Theorem \ref{t:mainep2} holds for $G=E_6$.
\end{prop}

\begin{proof}
We may assume that $\dim H \geqs 11$ since $\dim x^G \geqs 22$ for all nontrivial $x \in G$. According to Theorem \ref{e:max}, we have $H=N_{G}(X)$ with
$$X \in \{G_2,A_2G_2,T_2D_4,A_2^3,F_4,A_1A_5, T_1D_5, C_4\,(p \neq 2)\}.$$
If $H^0=F_4$, $A_1A_5$, $T_1D_5$ or $C_4$ (with $p \neq 2$) then $H$ is an involution-type subgroup, and these cases have already been dealt with in Lemma \ref{l:e6ai}. In each of the remaining cases we claim that \eqref{e:bb22} holds for all $x \in \mathcal{P}$, so \eqref{e:all22} follows.
Let $\{\omega_1, \ldots, \omega_6\}$ be a set of fundamental dominant weights for $G$, and let $V_{27}$ be the $27$-dimensional irreducible $KG$-module with highest weight $\omega_1$.

If $H=H^0=G_2$ then \cite[Table 31]{Lawunip} indicates that there are no long root elements in $H$, so
$\dim x^G \geqs 32$ for all $x \in \mathcal{P}$ and the claim follows since $\dim x^H \leqs 12$.

Next assume $H^0=A_2G_2$. Here $H/H^0=Z_2$ and $\dim x^H \leqs 18$ for all $x \in \mathcal{P}$, so we may assume $\dim x^G \leqs 36$. Suppose $x \in H$ is unipotent, so the relevant classes are labelled $A_1$ and $2A_1$ (with respective dimensions $22$ and $32$). If $x \in H \setminus H^0$ then $p=2$ and \cite[Proposition 1.13(iii)]{LLS} implies that $x$ is not a long root element, so $\dim x^G \geqs 32$. Moreover, $\dim x^H \leqs 5+8=13$
since $x$ is an involution, and the result follows. For unipotent elements $x \in H^0$, Ross Lawther has determined the $G$-class of $x$ (using the method described in \cite{Lawunip}), and we quickly deduce that \eqref{e:bb22} holds in all cases. For completeness, we record this information in Table \ref{t:fus}. For instance, if $x=x_1x_2 \in A_2G_2$ is unipotent, where $x_1 \in A_2$ is regular and $x_2 \in G_2$ belongs to the class labelled $\tilde{A}_1^{(3)}$ (in which case $p=3$), then $x$ is in the $G$-class labelled $2A_2+A_1$ and thus $\dim x^H = 6+8=14$, $\dim x^G = 54$.

\begin{table}[h]
$$\begin{array}{r|llllll}
& 1 & A_1 & \tilde{A}_1 & \tilde{A}_1^{(3)} & G_2(a_1) & G_2 \\ \hline
1 & 1 & A_1 & 3A_1 & 3A_1 & A_2 & D_4 \\

A_1 & 2A_1 & 3A_1 & A_2+A_1  & A_2+A_1 & A_2+2A_1 & D_5(a_1) \\
& & & 3A_1\,(p=2) & &  &  D_4\,(p=2) \\
A_2 & 2A_2 & 2A_2+A_1 & A_3+A_1  & 2A_2+A_1 & D_4(a_1) & E_6(a_3) \\
& & & 2A_2\,(p=3) & & 2A_2+A_1\,(p=3) & D_5(a_1)\,(p=2)
\end{array}$$
\caption{The fusion of unipotent classes, $A_2G_2< E_6$}
\label{t:fus}
\end{table}

Now suppose $x \in H$ is semisimple, so $C_G(x)^0=D_5T_1$ since $\dim x^G \leqs 36$. Here $\dim x^G = 32$, so we may as well assume $\dim x^H \geqs 16$, whence $C_{H}(x)^0=T_4$ or $A_1T_3$. Now $\dim C_{\mathrm{Lie}(G)}(x)=46$ and \cite[Proposition 2.4]{LSM} gives
$$\mathrm{Lie}(G) \downarrow A_2G_2 = \mathrm{Lie}(A_2G_2) \oplus (\mathrm{Lie}(A_2) \otimes W(\l_1)).$$
First assume $C_H(x)^0=T_4$. Write $x=x_1x_2 \in A_2G_2$. Since $x_1$ and $x_2$ are both regular semisimple elements, with respect to suitable bases we calculate that $x_1$ acts on $\mathrm{Lie}(A_2)$ as the diagonal matrix $[I_2,\l I_2, \l^{-1}I_2,\l^2,\l^{-2}]$, and $x_2$ acts on $W(\l_1)$ as $[I_3, \mu I_2, \mu^{-1} I_2]$, for some $\l,\mu \in K^*$. In particular, $\dim C_{\mathrm{Lie}(A_2) \otimes W(\l_1)}(x) \leqs 18$ and thus
$\dim C_{\mathrm{Lie}(G)}(x) \leqs 22<46$, which is a contradiction. Similar reasoning eliminates the case $C_H(x)^0=A_1T_3$, and we conclude that \eqref{e:bb22} holds for all $x \in \mathcal{P}$.

Next suppose $H^0=T_2D_4$, so $H/H^0=S_3$.
Since $\dim x^H \leqs 24$ for all $x \in H$, we may assume $\dim x^G \leqs 48$. In particular, if $x$ is semisimple then $C_G(x)^0=T_1D_5$, $A_5A_1$, $A_5T_1$ or $T_2D_4$, and the bound supplied by \cite[Theorem 2]{LLS} is sufficient. Now assume $x \in H$ is unipotent. By \cite[Proposition 2.3]{LSM} we have
$$V_{27}\downarrow D_4 = W(\l_1)\oplus W(\l_3) \oplus W(\l_4)\oplus 0^3$$
and in the usual way we can compute the Jordan form of $x$ on $V_{27}$ (and subsequently determine the $G$-class of $x$ via \cite[Table 5]{lawthercom}). In particular, if $p=2$ and $x \in H \setminus H^0$ then $x$ induces a $b_1$ or $b_3$ involution on the $D_4$ factor (in the notation of \cite{AS}); in the former case, $x$ has Jordan form $[J_2^{10},J_{1}^7]$, otherwise it is $[J_2^{12},J_1^{3}]$; it follows that the respective $G$-classes are $2A_1$ and $3A_1$,  and the result follows. Similarly, if $p=3$ and $x \in H \setminus H^0$ has order $3$ then $x$ induces a triality graph automorphism on the $D_4$ factor and we calculate that $x$ has Jordan form $[J_3^9]$ on $V_{27}$ (there are two classes of triality graph automorphisms; they have the same Jordan form on $V_{27}$). Therefore, $x$ is in one of the classes $2A_2$ or $2A_2+A_1$, so $\dim x^H \leqs 20 < \frac{1}{2}\dim x^G$ as required.

Finally suppose $H^0=A_2^3$, in which case $H/H^0=S_3$.
Here $\dim x^H \leqs 18$ for all $x \in \mathcal{P}$, so we may assume $\dim x^G \leqs 36$. As before, if $x \in H$ is semisimple then \cite[Theorem 2]{LLS} is sufficient (note that $C_G(x)^0=D_5T_1$ is the only possibility with $\dim x^G \leqs 36$), so let us assume $x$ is unipotent. Here the relevant $G$-classes are labelled $A_1$ and $2A_1$, with respective dimensions $22$ and $32$. If $x \in H^0$ then the desired bound quickly follows from the information in \cite[Table 18]{Lawunip}. Now assume $x \in H \setminus H^0$, so $p=2$ or $3$. By \cite[Proposition 2.3]{LSM} we have
$$V_{27} \downarrow A_2^3 = (W(\l_1) \otimes W(\l_2) \otimes 0) \oplus (W(\l_2) \otimes 0 \otimes W(\l_1)) \oplus (0 \otimes W(\l_1) \otimes W(\l_2)).$$

If $p=3$ then $x$ cyclically permutes the $A_2$ factors of $H^0$, so from the above decomposition we deduce that $x$ has Jordan form $[J_3^9]$ on $V_{27}$ and thus \cite[Table 5]{lawthercom} indicates that $x$ is in one of the classes labelled $2A_2$ or $2A_2+A_1$, a contradiction.
Finally, let us assume $p=2$ and $x \in H \setminus H^0$ is an involution. By \cite[Proposition 1.13(iii)]{LLS}, $x$ is not a long root element, so we may assume $x \in 2A_1$ and thus $\dim x^G = 32$. Now
$x$ acts as a transposition on the $A_2$ factors, and it either centralizes or induces an involutory automorphism on the fixed factor. Therefore $\dim x^H \leqs 8+5=13< \frac{1}{2}\dim x^G$ as required.
\end{proof}

\begin{prop}
Theorem \ref{t:mainep2} holds for $G=F_4$.
\end{prop}

\begin{proof}
In view of Theorem \ref{e:max} and Lemma \ref{l:f4ai}, we may assume $H=N_{G}(X)$ with 
$$X \in \{D_4, \tilde{D}_{4}\,(p=2), A_1, G_2, A_1G_2,  A_2\tilde{A}_{2}\}.$$ 
Note that if $p=2$ then the subgroups $D_4$ and $\tilde{D}_4$ are interchanged by a graph automorphism of $G$, so we only need to consider $D_4$, which is generated by the long root subgroups. Let $\{\omega_1, \ldots, \omega_4\}$ be a set of fundamental dominant weights for $G$, and let $V_{26} = W(\omega_4)$ be the $26$-dimensional Weyl module for $G$ with highest weight $\omega_4$.

First consider the case $H^0=D_4$. Here $H/H^0=S_3$ and $\dim H>\frac{1}{2}\dim G$, so $b^0(G,H) \geqs 3$ by Proposition \ref{p:bb}(iii). We claim that $\dim x^H < \frac{2}{3}\dim x^G$ for all $x \in \mathcal{P}$, so
$$b^0(G,H)=b(G,H)=b^1(G,H)=3$$
(see Corollary \ref{c:con}). Note that
\begin{equation}\label{e:v26d}
V_{26} \downarrow D_4 = W(\l_1) \oplus W(\l_3) \oplus W(\l_4) \oplus 0^2.
\end{equation}
If $x$ is semisimple then the claim follows from \cite[Theorem 2]{LLS}, so let us assume $x$ is unipotent. If $x \in H^0$ then we can use the above decomposition \eqref{e:v26d} to determine the $G$-class of $x$. For example, suppose $p \neq 2$ and $x \in H^0$ has Jordan form $[J_3, J_1^5]$ on the natural $D_4$-module $W(\l_1)$. Then $x$ has Jordan form $[J_2^4]$ on $W(\l_3)$ and $W(\l_4)$ (the two spin modules for $D_4$), so $[J_3,J_2^8,J_1^7]$ is the Jordan form of $x$ on $V_{26}$ and thus \cite[Table 3]{lawthercom} indicates that $x$ belongs to the $G$-class labelled $\tilde{A}_{1}$. In this way, it is straightforward to verify the claim for all unipotent elements $x \in H^0$. (Note that if $p=2$ then $x \in D_4<B_4$ and the fusion of unipotent $B_4$-classes is stated explicitly in the proof of \cite[Lemma 4.6]{LLS}; in particular, involutions of type $c_2$ or $a_4$ in $H^0$ belong to the $G$-class labelled $\tilde{A}_{1}^{(2)}$.)

To complete the analysis of the case $H^0=D_4$ we may assume $x \in H \setminus H^0$ is unipotent. Suppose $p=2$ and $x$ is an involution; there are two such $H$-classes in $H\setminus H^0$. As before, using the decomposition \eqref{e:v26d}, it is easy to calculate the Jordan form of $x$ on $V_{26}$; if $x$ is a $b_1$-involution we get $[J_2^{10},J_1^6]$, and $[J_2^{12},J_1^2]$ is the Jordan form of a $b_3$-involution. The result now follows by inspecting \cite[Table 3]{lawthercom}. Finally, suppose $x \in H \setminus H^0$ and $p=3$, so $x$ acts on $D_4$ as a triality graph automorphism. There are two such $H$-classes in $H\setminus H^0$, and we calculate that $x$ has Jordan form $[J_3^8,J_2]$ on $V_{26}$, so $x$ is in one of the classes labelled $\tilde{A}_2$ or $\tilde{A}_2+A_1$ (with respective dimensions $30$ and $36$). In particular, if $C_{D_4}(x)=G_2$ then $\dim x^H = 14 < \frac{2}{3}\dim x^G$. On the other hand, if $C_{D_4}(x)\neq G_2$ then $\dim x^H = 20$ and we need to show that $x$ belongs to the $G$-class labelled $\tilde{A}_2+A_1$. To see this, first observe that
$$\mathrm{Lie}(G)\downarrow D_4 = \mathrm{Lie}(D_4) \oplus W(\l_1) \oplus W(\l_3) \oplus W(\l_4).$$
(see \cite[Table 8.4]{LSM}). Now $x$ has Jordan form $[J_3^8,J_2^2]$ on $\mathrm{Lie}(D_4)$, so $x$ has Jordan form $[J_3^{16},J_2^2]$ on $\mathrm{Lie}(G)$. In particular, \cite[Table 4]{lawthercom} indicates that $x$ is in one of the $G$-classes labelled $A_2+\tilde{A}_1$ or $\tilde{A}_2+A_1$, but we know that $x$ has Jordan form $[J_3^8, J_2]$ on $V_{26}$, so $x$ must be in the class $\tilde{A}_2+A_1$, as required.

In each of the remaining cases we claim that \eqref{e:bb22} holds for all $x \in \mathcal{P}$ (and thus \eqref{e:all22} follows). Since $\dim x^G \geqs 16$, the case $H^0=A_1$ is clear.

Next suppose $H^0=A_2\tilde{A}_2$. Here $H/H^0=Z_2$ and $\dim x^H \leqs 12$ for all $x \in \mathcal{P}$, so we may assume $\dim x^G \leqs 24$. The claim quickly follows from \cite[Theorem 2]{LLS} if $x$ is semisimple, so let us assume $x$ is unipotent. If $x \in H^0$ then the fusion information in
\cite[Table 16]{Lawunip} is sufficient, so we may assume $p=2$ and $x \in H \setminus H^0$ is an involution. Here $x$ acts as a graph automorphism on each $A_2$ factor, so
$\dim x^H = 10$. By \cite[Proposition 1.13(iii)]{LLS}, $x$ is not a root element, so $\dim x^G \geqs 22$ and the result follows.

Now assume $H=H^0=G_2$. Here $p=7$ (see \cite[Corollary 2]{LieS}) and we may as well assume $\dim x^G \leqs 24$. In particular, if $x$ is unipotent then $x$ belongs to one of the $G$-classes labelled $A_1$ or $\tilde{A}_{1}$, and the information in \cite[Table 28]{Lawunip} is sufficient. Now assume $x$ is semisimple. If $x$ has odd order then $\dim x^G \geqs 30$, so we may assume $x$ is an involution. There is a unique class of involutions $x \in G_2$ (this class has dimension $8$) and it suffices to show that  $C_G(x) = A_1C_3$ (rather than $B_4$). To see this, first note that
$$\mathrm{Lie}(G) \downarrow G_2 = \mathrm{Lie}(G_2) \oplus W(\l_1+\l_2)$$
(see \cite[Proposition 2.4]{LSM}). Now $x$ acts on $\mathrm{Lie}(G_2)$ as $[-I_8, I_6]$, so it remains to show that $x$ acts on $W(\l_1+\l_2)$ as $[-I_{20}, I_{18}]$, rather than $[-I_{8},I_{30}]$. In terms of fundamental dominant weights $\{\mu_1,\mu_2\}$, the restriction of the $G_2$-module $W(\l_1+\l_2)$ to $A_2<G_2$ is given by
$$W(\l_1+\l_2) \downarrow A_2 = \mathrm{Lie}(A_2) \oplus W(2\mu_1+\mu_2) \oplus W(\mu_1+2\mu_2).$$
Now $x = [-I_2,I_1] \in A_2$ acts on $\mathrm{Lie}(A_2)$ as $[-I_4,I_4]$ and \cite[Theorem 8.3]{GS} rules out the possibility that $x$ acts on $W(2\mu_1+\mu_2)$ (and also $W(\mu_1+2\mu_2) = W(2\mu_1+\mu_2)^*$) as $[-I_2,I_{13}]$. Therefore, $x$ must act on $W(\l_1+\l_2)$ as $[-I_{20}, I_{18}]$, so $C_G(x) = A_1C_3$ as claimed.

Finally, let us assume $H=H^0=A_1G_2$, so $p \neq 2$ (see \cite[Corollary 2]{LieS}). Since $\dim x^H \leqs 14$ for all $x \in H$, we may assume that $\dim x^G \leqs 28$. Consequently, if $x$ is unipotent then $x$ belongs to one of the $G$-classes labelled $A_1$, $\tilde{A}_1$ or $A_1\tilde{A}_1$, and in each case the required bound follows from the fusion information in \cite[Table 29]{Lawunip}. If $x$ is semisimple
and $\dim x^G \leqs 28$ then $C_{G}(x)=B_4$ or $A_1C_3$. In particular, $x$ is an involution. Since $\dim x^H \leqs 10$ for all involutions $x \in H$, we reduce to the case $C_{G}(x)=B_4$. We claim that $\dim x^H=2$. By \cite[Proposition 2.4]{LSM} we have
$$\mathrm{Lie}(G) \downarrow A_1G_2 = \mathrm{Lie}(A_1G_2) \oplus (W(4) \otimes W(\l_1)).$$
Write $x=x_1x_2$, where $x_1 \in A_1$ and $x_2 \in G_2$. Suppose $x_1$ and $x_2$ are both nontrivial. There is a unique class of involutions in both $A_1$ and $G_2$, and it is easy to see that $x_1$ acts on $W(4)$ as $[-I_{2},I_{3}]$, and $x_2$ acts on $W(\l_1)$ as $[-I_{4},I_{3}]$. It follows that $x$ acts on $\mathrm{Lie}(G)$ as $[-I_{28},I_{24}]$, so $C_{G}(x)=A_1C_3$. In this way, we deduce that $C_{G}(x)=B_4$ if and only if $x_1 \neq 1$ and $x_2=1$, whence $\dim x^H = 2$ as claimed.
\end{proof}

\begin{prop}
Theorem \ref{t:mainep2} holds for $G=G_2$.
\end{prop}

\begin{proof}
According to Theorem \ref{e:max} we have $H=N_{G}(X)$ with 
$$X \in \{A_1\tilde{A}_1, A_1, A_2, \tilde{A}_2\,(p=3)\}.$$ 
Note that if $p=3$ then the subgroups $A_2$ and $\tilde{A}_2$ are interchanged by a graph automorphism of $G$, so we only need to consider $A_2$, which is generated by the long root subgroups. The case $H^0=A_1\tilde{A}_1$  corresponds to an involution-type subgroup and this has already been dealt with in Lemma \ref{l:g2ai}. If $H=A_1$ then $\dim x^H \leqs 2$ for all $x \in \mathcal{P}$, so $\dim x^H <\frac{1}{2}\dim x^G$ (since $\dim x^G \geqs 6$) and thus Corollary \ref{c:con} implies that \eqref{e:all22} holds.

Finally, suppose $H^0=A_2$. Here $H/H^0=Z_2$ and $\dim H = 8 >\frac{1}{2}\dim G$ so $b^0(G,H) \geqs 3$. We claim that $\dim x^H \leqs \frac{2}{3}\dim x^G$ for all $x \in \mathcal{P}$, with equality if and only if $x \in H^0$ is a long root element. In particular, Corollary \ref{c:lr} implies that $b^1(G,H) \leqs 3$, so
$$b^0(G,H)=b(G,H)=b^1(G,H)=3.$$
First assume $x \in H$ is a unipotent element. If $x \in H^0$ then the claim follows from the information in \cite[Table 11]{Lawunip}, so let us assume $p=2$ and $x \in H \setminus H^0$ is an involution. Here $x$ acts as a graph automorphism on $H^0$, so $\dim x^H = 5$ and we calculate that $x$ has Jordan form $[J_2^3,J_1]$ on the $7$-dimensional Weyl module $V_7$ (since $V_7 \downarrow A_2 = W(\l_1) \oplus W(\l_2)\oplus 0$). Therefore \cite[Table 1]{lawthercom} indicates that $x$ is in the $G$-class $\tilde{A}_{1}$, and thus $\dim x^G = 8$. Finally, suppose $x$ is semisimple. Since $\dim x^H \leqs 6$ we may as well assume $\dim x^G \leqs 8$, so $C_{G}(x)^0=A_1\tilde{A}_1$ or $A_2$. In the latter case, $x$ centralizes $H$ so assume $C_{G}(x)^0=A_1\tilde{A}_1$. Here $p \neq 2$, $x$ is an involution and $\dim x^H = 4$. The result follows.
\end{proof}

\vs

This completes the proof of Theorem \ref{t:mainep2}.

\section{Normalizers of tori}\label{s:zin}

In this final section we prove Theorem \ref{t:zin}. Let $G$ be a simple algebraic group of rank $r$ over an algebraically closed field of characteristic $p \geqs 0$, let $T$ be a maximal torus of $G$ and consider the action of $G$ on $\Omega=G/H$, where $H=N_G(T)$. Recall that Theorem \ref{t:zin} states that either $b^1(G,H)=2$, or $G=A_1$ and the generic $2$-point stabilizer has order $2$.

Suppose $G=A_r$, so $H$ is a $\C_2$-subgroup of $G$ of type ${\rm GL}_{1} \wr S_{r+1}$. If $r=1$ and $p \neq 2$ then the desired result follows from Theorem
\ref{inv:main}, and the case $p=2$ is handled in the proof of Lemma \ref{p:aic}. If $r>1$ then Theorem \ref{t:cmain} yields $b^1(G,H)=2$ (see Section \ref{ss:irred}). Similarly, if $G=D_r$ (with $r \geqs 4$) then $H$ is a $\C_2$-subgroup of type $O_2 \wr S_r$ and once again the result follows from Theorem \ref{t:cmain}. In each of the remaining cases, we claim that
\begin{equation}\label{e:xhxg}
\dim x^H < \frac{1}{2}\dim x^G
\end{equation}
for all $x \in \mathcal{P}$, where $\mathcal{P}$ is the set of elements of prime order in $H$ (including all nontrivial unipotent elements if $p=0$). In particular, by applying Corollary \ref{c:con}, we deduce that $b^1(G,H)=2$.

If $G$ is an exceptional algebraic group then $\dim x^G >2r$ for all $x \in \mathcal{P}$, whence \eqref{e:xhxg} holds. Indeed, $\dim x^G \geqs \a$ where $\a$ is defined as follows:
$$\begin{array}{cccccc} \hline
G & E_8 & E_7 & E_6 & F_4 & G_2 \\ \hline
\a & 58 & 34 & 22 & 16 & 6 \\ \hline
\end{array}$$

Now, if $G=C_r$ (with $r \geqs 2$) then either $x$ is a long root element and $\dim x^G=2r$, or $\dim x^G \geqs 4r-4$ (see \cite[Proposition 2.9]{Bur2}). We immediately deduce that \eqref{e:xhxg} holds, unless $x$ is a long root element, or if $r=2$, $p \neq 2$ and $x=[-I_2,I_2]$ is an involution. In the latter case we calculate that $\dim (x^G \cap H)=1$ and $\dim x^G=4$. Similarly, if $x \in H$ is a long root element then $p=2$ and $\dim (x^G \cap H)=1$. Finally, let us assume $G=B_r$ (with $r \geqs 3$ and $p \neq 2$). Here $\dim x^G \geqs 2r$, with equality if and only if $x$ is a long root element or $x=[-I_{2r},I_1]$. But $H$ does not contain any long root elements, and if $x=[-I_{2r},I_1]$ we calculate that $\dim(x^G \cap H) = 1$. The result follows.

\vs

This completes the proof of Theorem \ref{t:zin}.

\vs

Finally, we sketch a proof of
Corollary \ref{c:zin}. Define $G$, $H$ and $r$ as above, and assume that $G$ is of adjoint type,  whence
 the center of  $\mathrm{Lie}(G)$ is trivial. Let $X$ be the product variety  $G/H \times G/H$.

We first show that the action of $G$ on $X$ is generically free. 
By Theorem \ref{t:zin}, the generic orbits of $G$ on $X$ are free, so
we need to show that there exists a non-empty open subvariety $U$ of $X$ such that
for all $x \in U$, the map $\phi_x$ sending $\mathrm{Lie}(G)$ to the tangent space of $x$ in $X$ is injective.
Of course, this is equivalent to the fact that $H$ acts generically freely on $G/H$.  

By  \cite[Proposition XI.5.9]{SGA3},  $H$ is a smooth group scheme.  This implies that
the kernel of $\phi_x$ is just the intersection of two maximal Cartan subalgebras of
$\mathrm{Lie}(G)$. It is easy to see that generically this intersection is just the center of the Lie algebra, 
which is trivial. Indeed, typically a Cartan subalgebra is the centralizer of a generic regular semisimple element, 
and two such elements generically generate the Lie algebra, so the common centralizer is the center.

By \cite[Example 7.3(b)]{DR},  the action of  $H$ on $G/H$ is versal.    Recall that one of
the equivalent definitions of ${\rm ed}(H)$ is the minimal value of
$\dim X - \dim H$, where  $X$ ranges over all generically free versal $H$-varieties (see 
\cite[Remark 2.6]{ICM}).  Thus, 
$${\rm ed}(H) \leqs \dim G - 2r$$ 
where $r = \dim H$ is the
rank of $G$. As we have noted, the inequality ${\rm ed}(G) \leqs {\rm ed}(H)$ was 
already proved by Springer (see \cite[Proposition 4.3]{Re}). 

\begin{remk}
This result was improved to ${\rm ed}(H) \leqs \dim G - 2r -1$ in \cite{GaG}.   The idea is to
show that $N_G(H)$ acts generically freely on the adjoint module for $G$ (by Theorem \ref{t:zin}) which gives the bound in Corollary \ref{c:zin}. Then one passes to projective space to improve the bound by $1$.   
\end{remk}

\end{document}